%% file: nonparametric_homogeneity_circle_arXiv.tex
\documentclass[oneside,11pt]{article}

\input{preamble_arXiv.tex}

\usepackage[pdftex,final,bookmarksnumbered,bookmarksopen=false,breaklinks,colorlinks]{hyperref}
\hypersetup{
	final,
	bookmarksnumbered=true,
	bookmarksopen=true,
	bookmarksopenlevel=0,
	unicode=false,
	pdftoolbar=true,
	pdfmenubar=true,
	pdffitwindow=false,
	pdftitle={A class of nonparametric homogeneity tests on the circle},
	pdfdisplaydoctitle=true,
	pdflang={English},
	pdfauthor={Alberto Fernandez-de-Marcos, Eduardo Garcia-Portugues},
	pdfsubject={arXiv paper},
	pdfcreator={Alberto Fernandez-de-Marcos},
	pdfproducer={Alberto Fernandez-de-Marcos},
	pdfkeywords={Circular data}{Homogeneity tests}{Multisample tests}{Sobolev tests}{Uniform scores},
	pdfnewwindow=true,
	breaklinks=true,
	hidelinks,
	linkcolor=black,
	citecolor=black,
	filecolor=magenta,
	urlcolor=cyan
}

\newif\ifmain
\maintrue
\newif\ifsupplement
\supplementtrue

\newif\iffigstabs
\figstabstrue

\begin{document}

\ifmain

\title{A class of nonparametric homogeneity tests on the circle}
\setlength{\droptitle}{-1cm}
\predate{}%
\postdate{}%
\date{}

\author{Alberto Fern\'andez-de-Marcos$^{1,2}$ and Eduardo Garc\'ia-Portugu\'es$^{1}$}
\footnotetext[1]{Department of Statistics, Universidad Carlos III de Madrid (Spain).}
\footnotetext[2]{Corresponding author. e-mail: \href{mailto:albertfe@est-econ.uc3m.es}{albertfe@est-econ.uc3m.es}.}
\maketitle

\begin{abstract}
    We develop a unified framework for $c$-sample homogeneity testing on the circle. The proposed class of \emph{$c$-sample Sobolev tests} generalizes the two-sample Sobolev tests based on uniform scores and encompasses the existing multisample tests as particular cases. We further embed this class into a broader aggregation framework, where the samplewise Sobolev components are combined through general merging functions, including average-type, maximum-type, and interpolating families. Within this class, we introduce the first Anderson--Darling-type homogeneity test for circular data, and we further propose two new tests designed to detect multimodal departures from homogeneity, constructed from softmax and Poisson kernels. We derive the asymptotic null distribution of the class, prove its consistency against a broad family of fixed alternatives, and obtain the asymptotic distribution under shift-type local alternatives. The tests are distribution-free and therefore do not require resampling. A comprehensive simulation study demonstrates the strong power of the Anderson--Darling-type test compared with Cramér--von Mises-type competitors across several scenarios, and the effectiveness of the softmax and Poisson tests under several alternatives. The testing toolbox is applied to analyze the nursing patterns of polar bears.
\end{abstract}
\begin{flushleft}
\end{flushleft}

\section{Introduction}
\label{sec:intro}

Testing the equality of two or more populations based on observed samples is a fundamental problem in statistical inference. Many of the proposed solutions adapt methods for the one-sample goodness-of-fit problem, see, e.g., \citet[Part~II]{Thas2010}. In the context of directional data, specifically circular data, the approach resembles that on the real line, with some additional difficulties due to the absence of a canonical ordering.

The homogeneity problem with $c\geq 2$ samples on the circle $\Sp^1:=\{\bx\in\R^{2}:\|\bx\|=1\}$ is formalized as follows. Let $\smash{\{\bX_i^{(\ell)}\}_{i=1}^{n_{\ell}}}$, with $\smash{\bX_i^{(\ell)}\sim {\rm P}_{\ell}}$ for all $i=1,\ldots,n_\ell$, be the $\ell$th sample for $\ell=1,\ldots,c$. We are concerned with testing
$$
\Hcal_0: {\rm P}_1 =\cdots = {\rm P}_c\quad
\text{against the most general alternative }\quad
\Hcal_1: \neg\Hcal_0.
$$
In the rest of the paper, every ${\rm P}_{\ell}$, $\ell=1,\ldots,c$, is an absolutely continuous distribution, and all the observations $\bX_{i}^{(\ell)}$ are mutually independent. For the sake of clarity, we work with polar coordinates, i.e., the samples are expressed as the angles $\{\Theta_i^{(\ell)}\}_{i=1}^{n_\ell}$, $\ell=1,\ldots,c$, with $\bx=\lrp{\cos\theta,\sin\theta}^{\top}\in\Sp^1$ and $0\leq\theta<2\pi$. The total sample size is denoted by $N:=\sum_{\ell=1}^{c}n_{\ell}$.

In the circular setting, the two-sample problem ($c=2$) has been extensively studied. As in the linear case, using the empirical cumulative distribution function (ecdf) offers a natural approach. However, classical univariate ecdf-based tests, such as the Kolmogorov--Smirnov and Cramér--von Mises tests, must be adapted to account for the geometry of the circle. Specifically, they must be invariant under rotations and reflections, so that neither the arbitrary starting point nor the direction of accumulation of the ecdf influences the test statistic. Let $\widehat{F}_{\ell,n_\ell}$ denote the angular ecdf of the $\ell$th sample with starting point $\theta = 0$, and the ecdf of the pooled sample be $\widehat{H}_N:=\sum_{\ell=1}^{c}\pi_{\ell,N} \widehat{F}_{\ell,n_\ell}$, with $\pi_{\ell,N}:=n_{\ell}/N$ denoting the proportion of the $\ell$th sample. The first test for $c=2$ was proposed by \cite{Kuiper1960}, measuring the dissimilarity between the two ecdfs using a rotation-invariant version of the Kolmogorov--Smirnov (KS) distance. Let $D_{n_1,n_2}^{+}:=\sup_{x\in[0,2\pi)}\big\{\widehat{F}_{1,n_1}(x) - \widehat{F}_{2,n_2}(x)\big\}$ and $D_{n_1,n_2}^{-}:=\sup_{x\in[0,2\pi)}\big\{\widehat{F}_{2,n_2}(x) - \widehat{F}_{1,n_1}(x)\big\}$. The Kuiper test statistic is
$$
V_{n_1,n_2}:=D_{n_1,n_2}^{+}+D_{n_1,n_2}^{-},
$$
which, unlike the KS statistic $\max\{D^{+}_{n_1,n_2},D^{-}_{n_1,n_2}\}$, is rotation-invariant.
The Cramér--von Mises (CvM) counterpart is the \cite{Watson1962} test, which uses a rotation-invariant version of the CvM distance between ecdfs, based on the empirical variance of the difference between the two ecdfs. The Watson test statistic is
\begin{align*}
    U_{n_1,n_2}^2:=\dfrac{n_1n_2}{N}\int_{0}^{2\pi}\bigg[\widehat{F}_{1,n_1}(x)-\widehat{F}_{2,n_2}(x)-\int_{0}^{2\pi}(\widehat{F}_{1,n_1}(y)-\widehat{F}_{2,n_2}(y))\,\mathrm{d} \widehat{H}_N(y)\bigg]^2\,\mathrm{d} \widehat{H}_N(x).
\end{align*}
The limiting null distributions of both statistics, as $N\to\infty$ and under non-extreme asymptotic rates for each sample, i.e., $\pi_{1,N}\to \pi_1\in(0,1)$, are distribution-free and coincide with the asymptotic null distribution of the \cite{Kuiper1960} and \cite{Watson1961} test statistics of uniformity, respectively.

A different way of building distribution-free two-sample tests on $\Sp^1$ was unveiled by \cite{Wheeler1964}. In that work, the \cite{Rayleigh1919} test of uniformity is applied to the \emph{uniform scores}, a transformation of the ranks, thereby yielding a homogeneity test. The testing procedure begins by choosing an arbitrary orientation and an origin on the circle. Then, the pooled sample is sorted, with $\br^{(\ell)}:=(r_1^{(\ell)},\ldots, r_{n_\ell}^{(\ell)})$ denoting the ranks of the $\ell$th sample with respect to the pooled sample. The uniform scores are defined as $\bc^{(\ell)}:=2\pi\br^{(\ell)}/N$, and the Rayleigh uniformity test is applied to $\bc^{(1)}$. That is,
\begin{align}\label{eq:Rayleigh-us}
    R_N:=\dfrac{2}{n_1}\bigg[\Big(\sum_{i=1}^{n_{1}}\cos(c_i^{(1)})\Big)^2 + \Big(\sum_{i=1}^{n_{1}}\sin(c_i^{(1)})\Big)^2\bigg].
\end{align}
Its asymptotic null distribution was found in \cite{Mardia1969a} and shown to be equal to that of the Rayleigh uniformity test, i.e., $R_N\inlaw\chi^2_2$. Another interesting feature to note is its invariance under interchanging the two samples. Indeed, the Rayleigh statistic depends on the length of the first sample's uniform scores resultant vector $\smash{\bR_1:=\sum_{i=1}^{n_1} (\cos c_i^{(1)},\sin c_i^{(1)})}$. Due to the equally spaced construction of the uniform scores, $\bR = \bR_1 + \bR_2=\mathbf{0}$, and therefore $\|\bR_1\|=\|\bR_2\|$.

Later, this uniform scores-based test was naturally extended from the particular Rayleigh test to the family of \emph{Sobolev} tests of uniformity introduced in \cite{Beran1968, Beran1969}, thereby unlocking a whole new class of two-sample tests on the circle. The class of test statistics presented in \cite{Beran1969a} is defined as
\begin{align}\label{eq:beran}
    B_{n_1,n_2}^{\psi}:=\sum_{i,j=1}^{n_1}\psi((r_i^{(1)}-r_j^{(1)})/N),
\end{align}
where $\psi:[-1,1]\to\R$ is given by the Fourier expansion $\psi(x):=2\sum_{k=1}^{\infty}v_k^2 \cos(2\pi k x)$ for $x\in[-1,1]$, with $(v_k)_{k\geq 1}$ being a real sequence such that $\sum_{k\geq 1}v_k^2<\infty$. Thus, $\psi$ is continuous, symmetric with respect to $0$ and $\pm 1/2$, and periodic of period $1$. In that work, it was shown that the \cite{Watson1962} two-sample test, $U_{n_1,n_2}^2$, belongs to this class. The asymptotic null distribution of $(N-1)B_{n_1,n_2}^{\psi}/(n_1n_2)$ is shown to be equal to that of the corresponding uniformity test, which is a weighted sum of independent chi-squared random variables, $\sum_{k\geq 1}v_k^2 \chi^2_2$. \cite{Schach1969b} also developed a class of two-sample nonparametric tests closely related to the Sobolev family of \cite{Beran1969a}. As in this class, the test statistics are $V$-statistics based on the uniform scores. However, in Schach's approach, the kernel $\psi_N:[-1,1]\to\R$ is a step function satisfying the same symmetry and periodicity conditions as $\psi$, but is allowed to depend on $N$. This dependence leads to a more refined treatment of the asymptotic behavior.

Considerably less attention has been devoted to the multisample $c>2$ case. To the best of our knowledge, the only ecdf-based proposal is due to \cite{Maag1966}, who introduced a natural extension of the \cite{Watson1962} test given by
$$
U^2_{N,c}:=\sum_{\ell=1}^{c} n_{\ell}\int_{0}^{2\pi}\left[\widehat{F}_{\ell,n_{\ell}}(x)- \widehat{H}_N(x)-\int_{0}^{2\pi}(\widehat{F}_{\ell,n_{\ell}}(y)- \widehat{H}_N(y))\,\rd  \widehat{H}_N(y)\right]^2\,\rd  \widehat{H}_N(x),
$$
which satisfies $U^2_{N,2}=U^2_{n_1,n_2}$. As an extension of the uniform scores-based tests, \cite{Mardia1972a} proposed a generalization of the \cite{Wheeler1964} test, with test statistic
$$
R_{N,c}:=\sum_{\ell=1}^{c} R_{N}(\bc^{(\ell)}),
$$
where $R_{N}(\bc^{(\ell)})$ denotes the Rayleigh statistic defined in~\eqref{eq:Rayleigh-us} but applied to the uniform scores of the $\ell$th sample. Note that $R_{N,2}=R_{N}/\pi_{2,N}$. Under the usual asymptotic regime, the null distribution satisfies $R_{N,c}\inlaw \chi^2_{2(c-1)}$. Later, \cite{Mardia1973} introduced a $q$-modal $c$-sample test on the circle, in the spirit of \cite{Mardia1972a}, but designed to detect multimodal alternatives,
\begin{align*}
R_{N,q,c}:=\sum_{\ell = 1}^{c}\dfrac{2}{n_\ell}\bigg[\Big(\sum_{i=1}^{n_{\ell}}\cos(q c_i^{(\ell)})\Big)^2 + \Big(\sum_{i=1}^{n_{\ell}}\sin(q c_i^{(\ell)})\Big)^2\bigg].
\end{align*}
This test statistic was developed considering $q$-modal von Mises distributions. For $q=1$, it reduces to $R_{N,c}$. Interestingly, for all $q\geq1$, its asymptotic null distribution coincides with that of $R_{N,c}$.

This paper addresses four main objectives in the study of circular $c$-sample homogeneity tests. First, given the scattered literature on the $c>2$ case, we establish a general $c$-sample framework, namely, \emph{$c$-sample Sobolev tests}, in which previously developed multisample tests---such as those of \cite{Mardia1972a}, \cite{Mardia1973}, and \cite{Maag1966}---appear as particular cases. This class of distribution-free tests naturally generalizes the \cite{Beran1969a} two-sample Sobolev tests, which are recovered as a special case. Second, we introduce a broader aggregation framework. Instead of summing the $c$ individual Sobolev statistics, we allow them to be combined through a general \emph{merging function} that satisfies certain properties. The usual $c$-sample Sobolev statistic corresponds to the sum merging function. Within this framework, we build a max-type statistic, which can be more powerful when the departure from homogeneity is present in a few samples. We also introduce two flexible interpolation families, the $M$-family and the log-sum-exp family, which range between sum-type and max-type behavior. Third, we introduce the first Anderson--Darling-type homogeneity test for circular data, building on the corresponding uniformity test introduced in \cite{Garcia-Portugues2020b}. In addition, we propose two new tests for homogeneity designed for multimodal alternatives, based on recent uniformity tests: the softmax-based test of \cite{Fernandez-de-Marcos2023b} and the Poisson-kernel-based test of \cite{Pycke2010}. Fourth, we derive the asymptotic theory of the general aggregation $c$-sample Sobolev tests. Under the null hypothesis, we show the test statistics are distribution-free, hence avoiding the need for resampling. We also correct an overlooked assumption on the summability of the coefficients $v_k$ in Theorem 1 of \cite{Beran1969a} and show that ignoring this assumption leads to nonstandard asymptotic behavior of the test statistic, due to the discreteness of the uniform scores. The consistency of $c$-sample Sobolev tests is established against absolutely continuous fixed alternatives. We also derive the asymptotic distribution of these homogeneity tests under novel shift-type local alternatives that have not previously been studied for circular data. Numerical experiments demonstrate the improved power of the Anderson--Darling test compared to the CvM-type test, as well as the effectiveness of the multimodal tests against various alternative scenarios. In a real-data application we evaluate whether the diel patterns of polar bears nursing their cubs follow the same distribution across different seasons and different cubs' ages.\nowidow[3]

The remainder of the paper is organized as follows. Section~\ref{sec:unif-scores} introduces the $c$-sample Sobolev class of test statistics and establishes several of its properties. In Section~\ref{sec:specific} we explore the connection between specific kernels and existing test statistics, and we present both the new Anderson--Darling $c$-sample test and the tests designed for multimodal alternatives. Section~\ref{sec:general-agg} introduces the general aggregation class, defining a merging function and exploring the different choices available. Section~\ref{sec:null-asymp} derives the null distribution of the class, and Section~\ref{sec:nonnull-asymp} shows the asymptotic behavior under non-null alternatives, including the consistency against fixed alternatives and the asymptotic distribution under shift-type local alternatives. Section~\ref{sec:sim} presents the simulation results showing the finite-sample performance of the tests, using different kernels and merging functions. Section~\ref{sec:app_bear} illustrates the application of these tests to a real dataset. Concluding remarks are provided in Section~\ref{sec:discussion}. All proofs and additional simulations are provided in the Supplementary Materials (SM).

\section{\texorpdfstring{$c$}{c}-sample Sobolev test based on uniform scores}
\label{sec:unif-scores}

\subsection{Background}
\label{subsec:unif-scores:basics}

We begin by reviewing the required basics of the Sobolev class of uniformity test statistics on $\Sp^1$. Given a circular sample $\bTheta:=\{\Theta_i:i=1,\ldots, n\}$, the Sobolev test statistic of uniformity based on the kernel $\phi$ is defined as
\begin{align}\label{eq:sob-def}
    S^{\phi}_n(\bTheta):=\dfrac{1}{n}\sum_{i,j=1}^{n}\phi(\cos(\Theta_i - \Theta_j)),
\end{align}
where $\phi:[-1,1]\to\R$ with $\phi\in L^2[-1,1]$, the space of square integrable functions defined on $[-1,1]$ with respect to the weight $t\mapsto(1-t^2)^{-1/2}$. Thus, $\phi$ can be expanded in terms of the Chebyshev polynomials $T_{k}(\cos x):=\cos(kx)$, which form an orthogonal basis of $L^2[-1,1]$. Sobolev statistics are designed such that the expectation of $\phi$ under the uniform distribution is zero, which makes $b_0(\phi)=0$ in the following expansion. Using the addition formula, we obtain
\begin{align}\label{eq:sob-exp}
    \phi(\cos(\theta_1 - \theta_2))&=\sum_{k=1}^{\infty}b_k(\phi)\cos(k(\theta_1-\theta_2))\nonumber\\
    &=\sum_{k=1}^{\infty}b_k(\phi)\big(\cos(k\theta_1)\cos(k\theta_2) + \sin(k\theta_1)\sin(k\theta_2)\big),
\end{align}
where 
$$
b_k(\phi)=\dfrac{2}{\pi (1+\delta_{k0})}\int_{-1}^{1}\phi(t)\cos{(k(\cos^{-1} t))}(1-t^2)^{-1/2}\,\rd t
$$
are the projections of the kernel onto the orthogonal basis $\{T_k\}_{k\geq 0}$, and $\delta_{k\ell}$ denotes the usual Kronecker delta.
In addition, the kernel $\phi$ is assumed to meet the summability condition 
\begin{align}\label{eq:summ-cond}
    \sum_{k=1}^{\infty}|b_k(\phi)|<\infty,
\end{align}
which ensures that $\phi(1)$ is well defined and also that the series \eqref{eq:sob-exp} converges uniformly on $[-1,1]$, which in turn ensures that $\phi$ is continuous, as is the case for the kernel $\psi$ in \eqref{eq:beran}.

\subsection{The class of \texorpdfstring{$c$}{c}-sample Sobolev statistics}
\label{subsec:unif-scores:class}

We now introduce the class of $c$-sample Sobolev tests, for $c\geq 2$, as Sobolev statistics acting on uniform scores. Let $\br^{(\ell)}$ be the ranks of the $\ell$th sample with respect to the pooled sample, $\br^{(\ell)}:=(r_1^{(\ell)},\ldots, r_{n_\ell}^{(\ell)})$, and $\bc^{(\ell)}$ be the corresponding uniform scores, $\bc^{(\ell)}:=2\pi\br^{(\ell)}/N$. We consider the statistic
\begin{align}\label{eq:unif-sc-k}
    T_{N,c}^{\phi}:=\sum_{\ell=1}^{c}S^{\phi}_{n_\ell}(\bc^{(\ell)}).
\end{align}
Henceforth, we restrict~\eqref{eq:unif-sc-k} to samples with size satisfying $\min_{1\leq \ell\leq c}n_\ell\geq 2$, and hence $N\geq 2c$. Expression~\eqref{eq:sob-exp} allows writing~\eqref{eq:unif-sc-k} in a form that is crucial for obtaining the asymptotic behavior of the statistic and proving certain properties:
\begin{align*}
    T_{N,c}^{\phi}&=\sum_{\ell=1}^{c}\dfrac{1}{n_{\ell}}\sum_{i,j=1}^{n_{\ell}}\sum_{k=1}^{\infty}b_k(\phi)\big(\cos(kc_i^{(\ell)})\cos(kc_j^{(\ell)}) + \sin(kc_i^{(\ell)})\sin(kc_j^{(\ell)})\big)\nonumber\\
    &=\sum_{\ell=1}^{c}\sum_{k=1}^{\infty}b_k(\phi)\dfrac{1}{n_{\ell}}\Big(\sum_{i,j=1}^{n_{\ell}}\cos(kc_i^{(\ell)})\cos(kc_j^{(\ell)}) + \sum_{i,j=1}^{n_{\ell}}\sin(kc_i^{(\ell)})\sin(kc_j^{(\ell)})\Big)\nonumber\\
    &=\sum_{\ell=1}^{c}\sum_{k=1}^{\infty}b_k(\phi)\dfrac{1}{n_{\ell}}\bigg(\Big(\sum_{i=1}^{n_{\ell}}\cos(kc_i^{(\ell)})\Big)^2 + \Big(\sum_{i=1}^{n_{\ell}}\sin(kc_i^{(\ell)})\Big)^2\bigg)\nonumber\\
    &=\sum_{\ell=1}^{c}\sum_{k=1}^{\infty}\sum_{r=1}^{2} 2^{-1}b_k(\phi) \dfrac{1}{n_{\ell}}\Big(\sum_{i=1}^{n_{\ell}}g_{k,r}(c_i^{(\ell)})\Big)^2,
\end{align*}
where $g_{k,1}(\theta):=\sqrt{2}\cos(k\theta)$ and $g_{k,2}(\theta):=\sqrt{2}\sin(k\theta)$. Note that in the second equality we used the absolute summability of coefficients in \eqref{eq:summ-cond}.

Proposition~\ref{prp:unif-scores-symm}\eqref{cond:symm-2} reveals that the statistic~\eqref{eq:unif-sc-k} encompasses the class of \cite{Beran1969a} for $c=2$ and kernels $\phi$ with sequences of coefficients $b_k(\phi)\geq0$ for all $k\geq 1$.

We now review certain desired properties that the statistic \eqref{eq:unif-sc-k} satisfies. As with the two-sample class of statistics $\smash{B_{n_1,n_2}^{\psi}}$, the statistic $\smash{T_{N,c}^{\phi}}$ is rotation-invariant. Indeed, Sobolev statistics only depend on the pairwise angles between the uniform scores, which remain constant under rotations.
The statistic is also symmetric or invariant under relabeling of the samples. This is immediate from commutativity. However, it is not immediately apparent if the test based on $\smash{B_{n_1,n_2}^{\psi}}$ is symmetric. Proposition~\ref{prp:unif-scores-symm}\eqref{cond:symm-1} answers this question affirmatively, while unlocking the connection between the proposed class and the \cite{Beran1969a} class.
\begin{proposition}
\label{prp:unif-scores-symm}
Let $c=2$. Then:
\begin{enumerate}[label=(\textit{\roman*}),ref=\textit{\roman*}]
        \item $n_1 S_{n_1}^{\phi}(\bc^{(1)})-n_2 S_{n_2}^{\phi}(\bc^{(2)})=(n_1-n_2)N\sum_{k=1}^{\infty} b_{kN}(\phi)$.\label{cond:symm-1}
        \item $T_{N,2}^{\phi}=\dfrac{N}{n_1n_2}B^{\psi}_{n_1,n_2}-\Big(\dfrac{n_1}{n_2}-1\Big)N\sum_{k=1}^{\infty} b_{kN}(\phi)
        $, where $\psi(x):=\phi(\cos(2\pi x))$ for all $x\in[-1,1]$ with $b_k(\phi)\geq0$ for all $k\geq 1$.\label{cond:symm-2}
    \end{enumerate}
\end{proposition}
We can observe that the statistic $B_{n_1,n_2}^{\psi}$ is not symmetric in general, since $n_1 S_{n_1}^{\phi}(\bc^{(1)})=B^{\psi}_{n_1,n_2}\neq B^{\psi}_{n_2,n_1}=n_2 S_{n_2}^{\phi}(\bc^{(2)})$ when $n_1 \neq n_2$ and $N\sum_{k=1}^{\infty} b_{kN}(\phi)\neq 0$. However, the difference in~\eqref{cond:symm-1} is deterministic, so that the corresponding test is indeed invariant under sample relabeling. Moreover, the statistic is invariant when one of the following conditions holds: (\emph{a}) $n_1 = n_2$, or (\emph{b}) $N\sum_{k=1}^{\infty} b_{kN}(\phi)=0$. For kernels with a nonnegative sequence of coefficients, the latter happens precisely when $\{k\geq 1: b_k(\phi)\neq 0\}\cap N\Z$ is empty. This gives an alternative explanation for why the Rayleigh statistic, $R_N$, is indeed symmetric, since $b_k=0$ for all $k>1$, yielding an empty intersection for all $N>1$.

\subsection{Discrete nature of the statistics}
\label{subsec:unif-scores:discrete}

The statistics based on uniform scores are discrete, since the rank nature of uniform scores collapses different configurations of the sample into a finite number of equivalence classes. Thus, obtaining the exact distribution for small sample sizes is feasible. However, the number of different configurations grows exponentially, which, in turn, motivates the use of the asymptotic distribution in practice for moderate sample sizes. To give some intuition on the finite-sample behavior of the statistics, we use \emph{bracelets}, combinatorial objects whose enumeration counts the distinct configurations of $N$ elements on $\Sp^1$ up to rotations and reflections.

A bracelet is the lexicographically minimal representative of an equivalence class of $c$-ary strings under rotation and reversal. The set of all bracelets of length $N$ over an alphabet of $c$ symbols with content $\bn:=(n_1,\ldots,n_c)$, denoting the number of elements for each respective symbol, is denoted by $\bB_c(\bn)$. Let
$C(\bn):=N!/(\prod_{\ell=1}^{c} n_{\ell}!)$ be the number of strings with this
content. Then, using Lemma 2.2 in \cite{Csiszar2004} and the fact that each equivalence class under rotation and reversal contains at most $2N$ strings, the following lower bounds follow
\begin{align*}
    \dfrac{\exp(N h(\bn/N))}{2N(N+1)^c} \leq \dfrac{C(\bn)}{2 N}\leq \#\bB_c(\bn),
\end{align*}
where $h(\bpi):=-\sum_{\ell=1}^c \pi_\ell\log\pi_\ell$, and $\#\bB_c(\bn)$ grows exponentially with $N$ under the usual asymptotic regime. \cite{Karim2013} proposes an efficient algorithm to list $\bB_c(\bn)$.

Denoting the set of possible values of $T_{N,c}^{\phi}$ with $\Omega$, we have that $\#\Omega\leq \#\bB_c(\bn)$. Indeed, $\#\Omega$ may be reduced by intrinsic symmetries of the kernel $\phi$ and, in balanced settings, by invariance of the statistic under relabeling of the samples. Figure~\ref{fig:combs_us_cmp} shows $\#\Omega$ for three different kernels in the balanced case $n_1=\cdots=n_c$, together with the upper bound of $\#\bB_c(\bn)$ and a polynomial (quadratic) lower bound. It shows that $\#\Omega$ grows nearly exponentially for the three kernels and therefore the discreteness of the test statistics vanishes rapidly.

\begin{figure}[!ht]
    \centering
    \includegraphics[width=0.75\linewidth, clip=true, trim={0cm 0.56cm 0cm 0.75cm}]{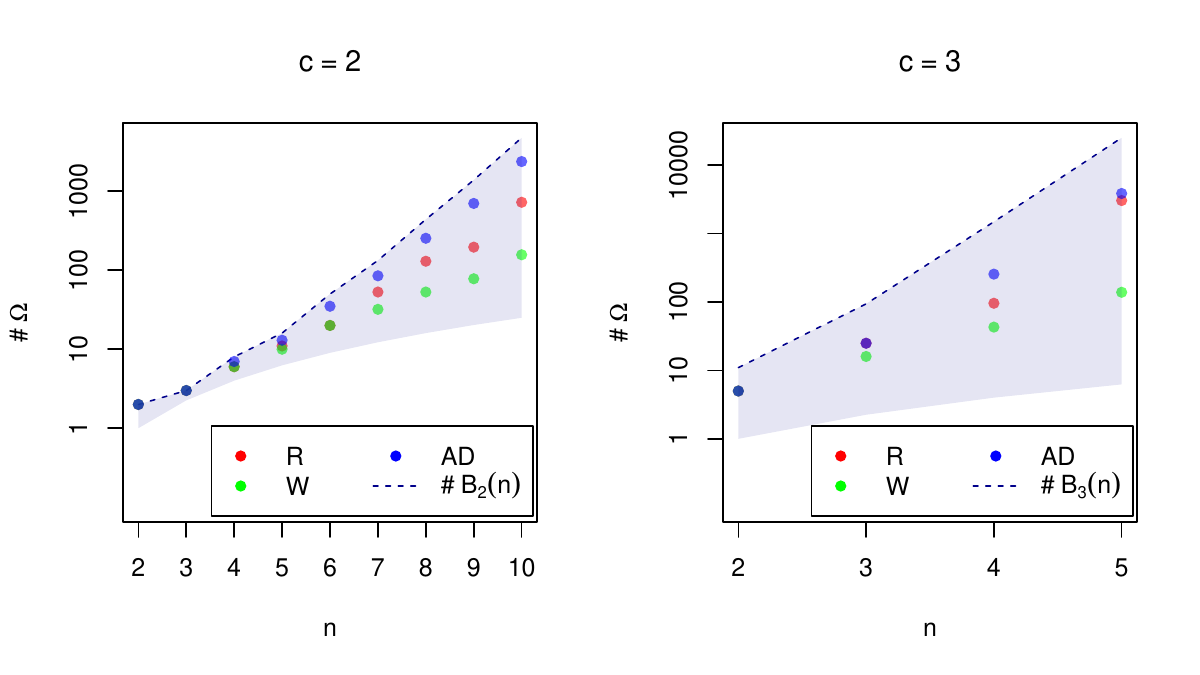}
    \caption{\small Total number of different values of $T_{cn,c}^{\phi}$ for $c$ balanced samples as a function of each sample size, $n$, for $\phi^{\rm R}$, $\phi^{\rm W}$, and $\phi^{\rm AD}$ (see Section~\ref{sec:specific}). Shaded areas range between the curves $n\mapsto n^2/4$ and $n\mapsto\#\bB_c(n,\ldots,n)$. The vertical axes use a logarithmic scale.}
    \label{fig:combs_us_cmp}
\end{figure}

\section{Specific cases and connections}
\label{sec:specific}

In this section, we explore the connection
between specific kernels and existing test statistics, and we present new tests for circular homogeneity that have not been considered previously. Throughout, for $x\in[-1,1]$, we write $\theta_x:=\arccos(x)\in[0,\pi]$.

\subsection{Classical tests that belong to the class}
\label{subsec:classic}

In this section, we review kernels that yield classical $c$-sample homogeneity tests from the literature.

\emph{Rayleigh test statistic.} The tests developed by \cite{Wheeler1964} and \cite{Mardia1972a} arise with the kernel $\phi^{\rm R}(x):=2x$, with coefficients $b_k(\phi^{\rm R})=2\delta_{k1}$, $k\geq1$. In particular, 
\begin{align*}
    T_{N,c}^{{\rm R}}
    =\sum_{\ell = 1}^{c} \dfrac{2}{n_{\ell}}\sum_{i,j=1}^{n_{\ell}} \cos(c_{i}^{(\ell)} - c_{j}^{(\ell)})
    =\sum_{\ell = 1}^{c} \dfrac{2}{n_{\ell}}\bigg[\Big(\sum_{i=1}^{n_{\ell}}\cos(c_i^{(\ell)})\Big)^2 + \Big(\sum_{i=1}^{n_{\ell}}\sin(c_i^{(\ell)})\Big)^2\bigg]=R_{N,c}.
\end{align*}

\emph{$q$-modal test statistics.}
Let $q\geq 1$. The \cite{Mardia1973} class of statistics comes from the kernel $\phi^{{\rm M}_q}(x):=2\cos(q\theta_x)$, with coefficients $b_k(\phi^{{\rm M}_q})=2\delta_{kq}$, $k\geq1$. In particular, $T^{{\rm M}_q}_{N,c}=R_{N,q,c}$. Also note that the Rayleigh statistic is a particular case of this class, with $q=1$, as well as the Bingham-type statistic, for $q=2$.

\emph{Watson test statistic.} The test proposed by \cite{Watson1962} and \cite{Maag1966} is obtained with the kernel
\begin{align*}
    \phi^{\rm W}(x):=\frac{1}{12}+\frac{\theta_x}{4 \pi}\left(\frac{\theta_x}{2 \pi}-1\right),
\end{align*}
with coefficients $b_k(\phi^{\rm W})=(2 \pi^2 k^2)^{-1}$, $k\geq 1$, as shown next.

\begin{proposition}[Equivalence of the \cite{Watson1962} and \cite{Maag1966} tests with Watson's uniform scores test]\label{prp:Maag-unifsc}
    For $c\geq 2$, we have
    \begin{align*}
        U^2_{N,c}
        &=T^{W}_{N,c}-\frac{1}{12N}.
    \end{align*}
\end{proposition}

\emph{Rothman test statistic.}
Let $t\in(0,1)$ and $t_\wedge:=\min(t, 1-t)$. The statistic based on \cite{Rothman1972}'s test of uniformity is given by the kernel 
\begin{align*}
    \phi^{{\rm R}_{t}}(x):=\left(t_\wedge-\frac{\theta_x}{2 \pi}\right)_{+}-t_\wedge^2,
\end{align*}
with coefficients
\begin{align*}
    b_k(\phi^{{\rm R}_{t}})=2\sin^2(k\pi t_\wedge) (\pi k)^{-2},
\end{align*}
specified in \cite{Garcia-Portugues2020b}. For $t=1/2$, the kernel is equal to the \cite{Ajne1968} kernel, whose two-sample version based on uniform scores was shown to be a particular case in \cite{Beran1969a}.

\subsection{Area-based test statistic}
\label{subsec:area}

\cite{Shirahata1990} proposed a test of homogeneity in $\R$ based on ranks transformed to points equally spaced on a semicircle. This test can be adapted to use uniform scores with the kernel
$$
\phi^{\rm A}(x)=\frac{2}{\pi} - \sin\bigg(\frac{\theta_x}{2}\bigg)
$$
and coefficients
$$
b_k(\phi^{\rm A})=\frac{4}{\pi(4k^2 - 1)},
$$
leading to a new homogeneity test for circular data. Note that the coefficients are remarkably similar to those of Watson, $b_k(\phi^{\rm W})$. 

The interest of this statistic resides in its geometric motivation, adapted from \cite{Shirahata1990}. Let $\smash{\alpha_i^{(\ell)}:={\pi}/({N+1})(r_i^{(\ell)} - ({N+1})/{2})}$ so that $\smash{\alpha_i^{(\ell)}\in(-\pi/2,\pi/2)}$, and let $\smash{\bs_i^{(\ell)}\in\Sp^1}$ be the unit vector such that $\smash{\arg(\bs_i^{(\ell)})=\alpha_i^{(\ell)}}$. Then, we build the $2n_\ell$-polygon whose edges are obtained by sequentially connecting the vectors $\smash{\bs_{(1)}^{(\ell)},\ldots,\bs_{(n_\ell)}^{(\ell)}, -\bs_{(1)}^{(\ell)},\ldots,\allowbreak-\bs_{(n_{\ell})}^{(\ell)},}$ where $\bs_{(i)}^{(\ell)}$ corresponds to the unit vector for the ordered $\smash{\alpha_{(i)}^{(\ell)}}$. The area corresponding to this polygon is computed by $$T^{(\ell)}:=\sum_{i<j}\sin|\alpha_i^{(\ell)}-\alpha_j^{(\ell)}|=\sum_{i<j}\sin|N/(2(N+1))(c_i^{(\ell)}-c_j^{(\ell)})|.$$
Then, the statistic $A_{N,c}:=2N/\pi-2\sum_{\ell=1}^c T^{(\ell)}/n_{\ell}$ is expected to be small under $\Hcal_0$, as the area of the constructed polygon is maximized when it is regular, and
\begin{align*}
    A_{N,c}&=\frac{2N}{\pi}-2\sum_{\ell=1}^c \frac{1}{n_{\ell}} T^{(\ell)}
    =\sum_{\ell=1}^c \frac{1}{n_{\ell}} \sum_{i,j=1}^{n_\ell}\bigg(\frac{2}{\pi}-\sin\frac{N|c_i^{(\ell)}-c_j^{(\ell)}|}{2(N+1)}\bigg).
\end{align*}
Thus, $A_{N,c}$ is based on a finite-sample kernel that, as $N\to\infty$,
\begin{align*}
    \frac{2}{\pi}-\sin\frac{N \abs{\theta}}{2(N+1)}\to \phi^{\rm A}(\cos\theta),\qquad\text{for all }\theta\in[-2\pi,2\pi].
\end{align*}

\subsection{New proposals}
\label{subsec:new}

\emph{Anderson--Darling test statistic.}
We define the circular Anderson--Darling $c$-sample test employing the Anderson--Darling kernel introduced in \cite{Garcia-Portugues2020b},
$$
\phi^{\rm AD}(x):=\begin{cases}
    1-2\log(2\pi)+\dfrac{1}{\pi}\big\{\theta_x\log\theta_x + (2\pi - \theta_x)\log(2\pi-\theta_x)\big\},&\theta_x\in(0,\pi],\\
    1,&\theta_x = 0,
\end{cases}
$$
with coefficients
$$
    b_{k}(\phi^{\rm AD})=\frac{1}{\pi k^2} \int_0^\pi \frac{1-\cos (2 k \theta)}{(\pi-\theta) \theta}\, \mathrm{d} \theta.
$$

\emph{Softmax test statistic.}
Let $\kappa>0$. We propose using the smooth maximum kernel introduced in \cite{Fernandez-de-Marcos2023b},
\begin{align*}
    \phi^{\rm SM_{\kappa}}(x):=e^{-\kappa}(e^{\kappa x} - \Ical_0(\kappa)),
\end{align*}
with coefficients
\begin{align*}
    b_k(\phi^{\rm SM_{\kappa}})=2 e^{-\kappa}\Ical_k(\kappa),
\end{align*}
where $\Ical_k$ is the modified Bessel function of the first kind and $k$th order.

\emph{Poisson test statistic.}
Let $0<\rho<1$. Using the Poisson kernel introduced by \cite{Pycke2010} in uniformity testing, we obtain
\begin{align*}
    \phi^{\rm P_{\rho}}(x):=\dfrac{(1-\rho)^2}{1-2 \rho x+\rho^2} - \dfrac{(1-\rho)^2}{1-\rho^2},
\end{align*}
with coefficients
\begin{align*}
    b_k(\phi^{\rm P_{\rho}})=\dfrac{2(1-\rho)^2}{1-\rho^2}\rho^k.
\end{align*}

In \cite{Fernandez-de-Marcos2023b}, the latter two flexible test statistics were shown to be especially sensitive to multimodality in uniformity testing when the corresponding parameters, $\kappa$ and $\rho$, take larger values. A similar behavior is expected in the homogeneity problem.

\section{Generalized aggregation class}
\label{sec:general-agg}

The statistic $T_{N,c}^{\phi}$ is obtained by summing the individual uniformity test statistics computed from each sample's uniform scores. Although this aggregation rule is simple and analytically tractable, it is not the only possible way of combining test statistics. In this section, we introduce a general framework for constructing aggregation rules with desirable asymptotic properties.

For an aggregation map $m:\R^c\to\R$, define
\begin{align*}
    T_{N,c}^{\phi,m}
    :=m\bigl(
        S^\phi_{n_1}(\bc^{(1)}),\ldots,
        S^\phi_{n_c}(\bc^{(c)})
    \bigr).
\end{align*}
A \emph{merging function} is an increasing map $m:\R^c\to\R$ satisfying:
\begin{enumerate}[label=(M$\arabic*$),ref=M$\arabic*$]
    \item\label{m:lip} $m$ is Lipschitz continuous with respect to the supremum norm, that is, there exists a constant $L>0$ such that for all $\bx,\by\in\R^c$, $|m(\bx)-m(\by)|\leq L\|\bx-\by\|_{\infty}$.
    \item\label{m:hom} There exists a function $m_{\infty}:\R^c\to\R$ such that $\lim_{N\to\infty} N^{-1}m(N\bx)=m_{\infty}(\bx)$ pointwise for all $\bx\in\R^c$.
    \item\label{m:pos} $m_{\infty}(\bx)>0$ for all $\bx\in[0,\infty)^c$ for which there exist two distinct indices $\ell_1,\ell_2\in\{1,\ldots,c\}$ such that $x_{\ell_1}>0$ and $x_{\ell_2}>0$.
    \item\label{m:str-inc} $m:\R^c\to\R$ is strictly increasing with respect to strict componentwise order, that is, if $x_\ell<y_\ell$ for every $\ell=1,\ldots,c$, then $m(\bx)<m(\by)$.
\end{enumerate} 
A merging function $m$ is \emph{symmetric} if $m(\bx)$ is invariant under any permutation of $\bx$. This symmetry ensures that the corresponding statistic is invariant under sample relabeling.

\begin{remark}\label{rmk:m-uniform-lim}
    Under condition~\eqref{m:hom}, condition~\eqref{m:lip} implies that the convergence is uniform over any compact set $K\subset\R^c$. This follows from the fact that the family $(h_N)$, with $h_N(\bx):=N^{-1}m(N\bx)$, is equi-Lipschitz.
\end{remark}

For the null asymptotic result, only condition \eqref{m:lip} is needed. Condition \eqref{m:hom} ensures the function is asymptotically homogeneous of degree one, which comes into play for the asymptotic behavior of the test statistic under fixed alternatives. This condition is imposed to maintain the linear divergence rate of the individual uniformity statistics under fixed alternatives. Note that for any homogeneous function of degree one $m_1:\R^c\to\R$, \eqref{m:hom} holds directly with $m_{\infty}=m_1$, and the resulting statistic is scale-equivariant. More generally, under \eqref{m:lip}--\eqref{m:hom}, $m_{\infty}$ inherits the properties of $L$-Lipschitz continuity and monotonicity.
Condition~\eqref{m:pos} is used to obtain consistency under fixed alternatives. Under these alternatives, the pooled limiting distribution is a convex combination of the $c$ component distributions with all weights positive, so at least two component distributions must differ from the pooled limiting distribution. Consequently, the limiting signal vector in Theorem~\ref{thm:asymp-Halt-general} has at least two positive coordinates. Requiring positivity when only one coordinate is positive would also imply consistency, but would unnecessarily exclude useful aggregations such as weighted sums with one zero weight. 
Finally, condition~\eqref{m:str-inc} is not necessary, but natural, since it rules out plateau regions of $m$ that will not capture changes in the test statistics.

A consequence of Proposition~\ref{prp:unif-scores-symm}\eqref{cond:symm-1} and~\eqref{m:str-inc} is that, for $c=2$, merging functions cannot yield distinct upper-tail tests. Thus, different choices of $m$ only matter when $c>2$. The next result proves this fact.

\begin{corollary}\label{cor:agg-equiv-c2}
    Let $c=2$. Suppose that $m:\R^2\to\R$ is a function satisfying \eqref{m:str-inc}. Then, all upper-tail tests based on $T_{N,2}^{\phi,m}$ are equivalent.
\end{corollary}

We now present several representative examples of merging functions.

\begin{example}[Maximum]\label{merging:max}
    The function $\max(\bx):=\max_{1\leq \ell\leq c}x_\ell$ is a symmetric merging function.
\end{example}

\begin{example}[Weighted sum]\label{merging:ws}
    Let $\bw$ be such that $w_\ell\geq 0$ for all $\ell=1,\ldots,c$, $\sum_{\ell=1}^{c}w_\ell=1$, and at most one $w_\ell$ is zero. The weighted sum given by
    $$
    {\rm WS}_{\bw}(\bx):=\sum_{\ell=1}^{c}w_\ell x_\ell
    $$
    is a merging function. In particular, the choice $\bw=c^{-1}\mathbf{1}_{c}$ gives the \emph{average} aggregation, $T_{N,c}^{\phi,{\rm avg}}:=\smash{T_{N,c}^{\phi, {\rm WS}_{c^{-1}\mathbf{1}_{c}}}}$. This statistic yields a test equivalent to the one based on the \emph{sum} statistic in~\eqref{eq:unif-sc-k}, since $\smash{T_{N,c}^{\phi}=cT_{N,c}^{\phi, {\rm avg}}}$. Moreover, the average is the only symmetric member of the $\rm WS$ family.
\end{example}

\begin{example}[Incomplete average]\label{merging:inc_avg}
    Another member of the $\rm WS$ family is the \emph{incomplete average}. It arises when $\bw=(c-1)^{-1}(\mathbf{1}_{c-1},0)^{\top}$, thus considering the sum of the uniformity statistics of only $c-1$ samples, instead of all $c$ samples as in~\eqref{eq:unif-sc-k}. This leads to the incomplete statistic defined as
    \begin{align*}
        T_{N, c}^{\phi, {\rm inc}}:=\sum_{\ell=1}^{c-1}S^{\phi}_{n_\ell}(\bc^{(\ell)})=(c-1)T_{N,c}^{\phi,{\rm WS}_{(c-1)^{-1}(\mathbf{1}_{c-1},0)^{\top}}}.
    \end{align*}
    When $c=2$, the incomplete statistic is deterministically related to the proposed statistic \eqref{eq:unif-sc-k}, as
    $$
    T_{N,2}^{\phi,{\rm inc}}= \dfrac{n_2}{N} T_{N,2}^{\phi}-(n_2-n_1)\sum_{k=1}^{\infty} b_{kN}(\phi),
    $$
    and the relation with the test statistic of \cite{Beran1969a} is given directly by
    $$
    n_1 T_{N,2}^{\phi,{\rm inc}}=B_{n_1,n_2}^{\psi},
    $$
    when $b_k(\phi)\geq0$ for all $k\geq1$. Therefore, the test based on $T_{N,2}^{\phi,{\rm inc}}$ is also invariant under relabeling of the samples. However, in general, for $c>2$, the incomplete test is not invariant under sample relabeling, and thus, to ensure this property, we need to adopt a decision rule to choose an ordering for the $c$ samples. This is a substantial practical drawback compared to the (complete) average test, as seen in Section~\ref{subsec:sim:incomplete}.
\end{example}

These merging functions behave differently under alternative scenarios. The test based on the average favors alternatives where signal is spread across samples, while the maximum aggregation favors alternatives where the departure is concentrated in a small number of samples. Building a family of tests where these two behaviors, along with a range of intermediate behaviors, can be achieved is valuable. The following two merging families extend the maximum and average merging functions in two alternative ways, both indexed by a tuning parameter.

\begin{example}[$M$-family]\label{merging:M-family}
    Let $r\geq 1$. The function given by
    $$
    M_{r}(\bx):=\Big(\frac{1}{c}\sum_{\ell=1}^{c}x_{\ell}^r\Big)^{1/r}
    $$
    can be considered only for nonnegative components, that is, $M_{r}:[0,+\infty)^c\to\R$ is a merging function, and it is symmetric. Despite this nonnegativity restriction, the family is broadly applicable, as most statistics $S_{n}^{\phi}$ are nonnegative (see, e.g., Section~\ref{sec:specific}).
    The average corresponds to $r=1$, while the maximum arises as $r\to+\infty$.
\end{example}

\begin{example}[Log-sum-exp family]
    The log-sum-exp merging function, given by 
    $$
    {\rm LSE}_{\kappa}(\bx):=\kappa^{-1}\log\Big(\frac{1}{c}\sum_{\ell=1}^{c}\exp(\kappa x_\ell)\Big),
    $$ 
    where $\kappa>0$, is symmetric. This family interpolates between the maximum $(\kappa\to+\infty)$ and the average $(\kappa\to0)$ merging functions.
    In contrast to the previous examples, the {\rm LSE} is not homogeneous.
    However, it follows that
    $
    N^{-1}{\rm LSE}_{\kappa}(N\bx)\to\max(\bx)
    $
    as $N\to\infty$ uniformly over all $\bx\in\R^c$, ensuring \eqref{m:hom} holds.
\end{example}

\section{Null asymptotics}
\label{sec:null-asymp}

Throughout this section, we derive all the results under the conventional asymptotic regime where $\lim_{N\to\infty}\pi_{\ell,N}=\pi_\ell\in(0,1)$ for all $\ell=1,\ldots,c$. We denote the vector of limiting proportions with $\bpi:=(\pi_1,\ldots,\pi_c)^{\top}$ and weak convergence with $\inlaw$.

Under the null hypothesis, Proposition~\ref{prp:gkr-asymp-H0} establishes a joint CLT for finite vectors of empirical harmonic averages evaluated at the uniform scores. The asymptotic null distribution of $T_{N,c}^{\phi,m}$ in Theorem~\ref{thm:asymp-H0-general} is then obtained through Slutsky's theorem and the continuous mapping theorem, and a truncation argument for tail control. This last step is where the additional condition on the kernel coefficients $b_k(\phi)$ is needed. Corollary~\ref{cor:asymp-H0-sum} finally diagonalizes the limiting Gaussian quadratic form, yielding the explicit weighted chi-squared series for the statistic $\smash{T_{N,c}^{\phi}}$.

\begin{proposition}\label{prp:gkr-asymp-H0}
    Let $K\geq 1$ be fixed, and let $\bG_N:=(\bG_{1,N}^{\top},\ldots, \bG_{K,N}^{\top})^{\top}$ with $\bG_{k,N}:=\break(\bG_{k,1,N}^{\top}, \bG_{k,2,N}^{\top})^{\top}$ where $\bG_{k,r,N}:=(G_{k,r,1,N},\ldots, G_{k,r,c,N})^{\top}$, and $G_{k,r,\ell,N}={n_{\ell}}^{-1}\sum_{i=1}^{n_\ell}g_{k,r}(c_i^{(\ell)})$. Then, under $\Hcal_0$ and as $N\to\infty$, we have
    $$
    N^{1/2}\bG_N\inlaw \mathcal{N}_D(\mathbf{0},\bSigma_0),
    $$
    with $D:=2Kc$, $\bSigma_0:=\bI_{2K}\otimes \bS$, and $\bS:= \operatorname{diag}(1/\bpi) - \mathbf{1}_{c}\mathbf{1}_{c}^{\top}$.
\end{proposition}

\begin{theorem}\label{thm:asymp-H0-general}
    Let $m:\R^c\to\R$ be a function satisfying \eqref{m:lip}. Assume that $|b_k(\phi)|=O(k^{-(1+\delta)})$ for some $\delta>0$. Then, under $\Hcal_0$ and as $N\to\infty$,
    \begin{align}\label{eq:Tcmp_asymp_H0_general}
        T_{N, c}^{\phi, m}\inlaw T_{\infty, c}^{\phi, m}:=m\bigg(\Big(\pi_\ell \sum_{k=1}^{\infty}\sum_{r=1}^{2} 2^{-1}b_k(\phi) Z_{k,r,\ell}^2\Big)_{\ell=1}^{c}\bigg),
    \end{align}
    where $\{\bZ_{k,r}:=(Z_{k,r,1},\ldots,Z_{k,r,c})\sim \mathcal{N}_{c}(\mathbf{0}, \bS): k\geq1, r=1,2\}$ is a collection of mutually independent random vectors.
\end{theorem}

\begin{corollary}\label{cor:asymp-H0-sum}
    Assume that $|b_k(\phi)|=O(k^{-(1+\delta)})$ for some $\delta>0$. Under $\Hcal_0$ and as $N\to\infty$,
    \begin{align}\label{eq:Tcmp_asymp_H0}
        T_{N,c}^{\phi}\inlaw T_{\infty,c}^{\phi}:=\sum_{k=1}^{\infty}2^{-1}b_k(\phi)Y_k,
    \end{align}
    where $\{Y_k\sim \chi^2_{2(c-1)}: k\geq 1\}$ is a collection of mutually independent chi-squared distributed random variables.
\end{corollary}

\begin{remark}
    The decay condition $|b_k(\phi)|=O(k^{-(1+\delta)})$ was overlooked in Theorem 1 of \cite{Beran1969a}. In the derivation of his equation (2.21), the calculations fail when $k\in N\Z_{+}$, see Lemma~\ref{lemma:exp-G2}. Therefore, Theorem 1, without any further assumption on the decay of $|b_k(\phi)|$, is not true. 
    Indeed, the following is a counterexample. Consider $T_{N,c}^{\phi_b}$, where the kernel $\phi_b$ has coefficients
    \begin{align}\label{eq:bk-2/3}
        b_k=\begin{cases}
            {1}/{m^2},&\text{if }k=m^3\text{ for some }m\in\N,\\
            0,&\text{otherwise.}
        \end{cases}
    \end{align}
    It is immediate that $\sum_{k\geq 1}|b_k|< \infty$. However, $f(N):=N\sum_{k\in N\Z_+} b_{k}$ does not converge as $N\to\infty$. Indeed, the subsequence $(f(n^3))_{n\in\N}$ diverges, as
    $$n^3\sum_{k\geq 1}b_{kn^3}=n^3\sum_{k\in\{m^3:m\in\N\}} b_{kn^3}=n^3\sum_{m\in\N} 1/(mn)^2
    =n\pi^2/6.$$
    Together with Corollary~\ref{cor:centered_asymp}, it can be shown that
    $\smash{T_{n^3,c}^{\phi_b}\inprob +\infty}$ under $\Hcal_0$. In contrast, along other subsequences, such as $(p_n)_{n\in \N}$, where $p_n$ denotes the $n$th prime number, $T_{p_n,c}^{\phi_b}$ converges in law to the weak limit~\eqref{eq:Tcmp_asymp_H0}, since for any prime $p$ and $m\in\N$, $p\mid m$ if and only if $p\mid m^3$, so that
    $$p_n\sum_{k\in p_n\Z_+} b_{k}
    =p_n\sum_{k\in \{m^3: m\in p_n\Z_+\}} b_{k}=p_n\sum_{\ell\in\N} \dfrac{1}{(p_n\ell)^2}=\dfrac{\pi^2}{6p_n}\to 0,\qquad\text{as } n\to\infty.$$
    Figure~\ref{fig:asymp0-vMF-kappa10-k-2/3} (top) shows that $T_{N,c}^{\phi_b}$ under $\Hcal_0$ diverges in probability along the sequence of sample sizes $(n^{3})_{n\in \N}$ even for $c = 2$ (left figure), while being stable for the sequence of sample sizes $(p_n)_{n\in\N}$ (right figure).

\end{remark}
\begin{figure}[ht!]
    \centering
    \begin{subfigure}{0.633\linewidth}
        \includegraphics[width=\linewidth, clip=true, trim={0cm, 0.6cm, 0cm, 2cm}]{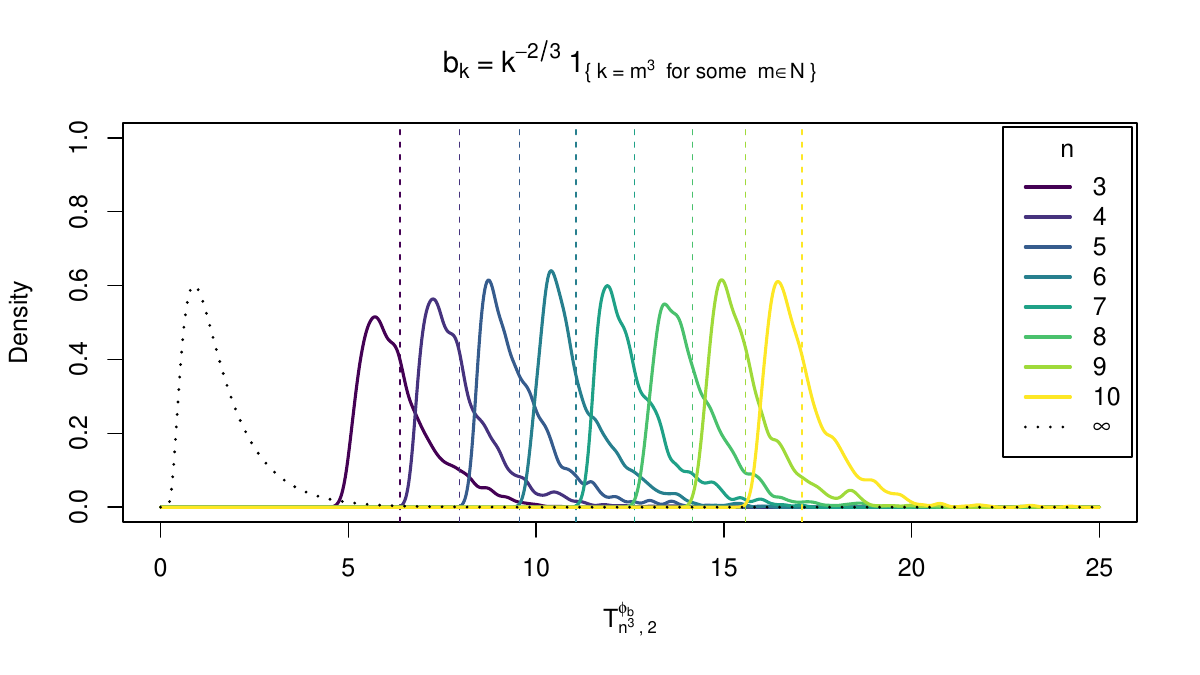}
    \end{subfigure}
    \begin{subfigure}{0.325\linewidth}
        \includegraphics[width=\linewidth, clip=true, trim={1cm, 0.6cm, 0cm, 2cm}]{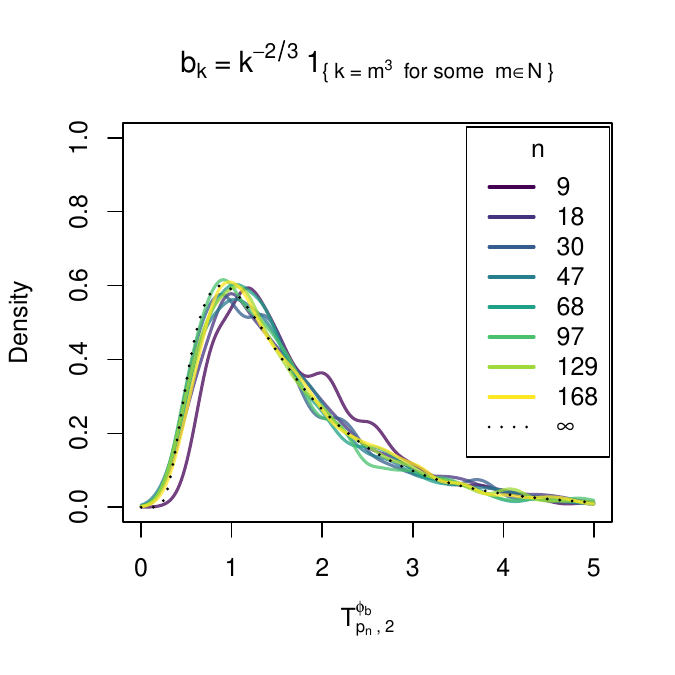}
    \end{subfigure}
    
    \begin{subfigure}{0.633\linewidth}
        \includegraphics[width=\linewidth, clip=true, trim={0cm, 0.6cm, 0cm, 1.75cm}]{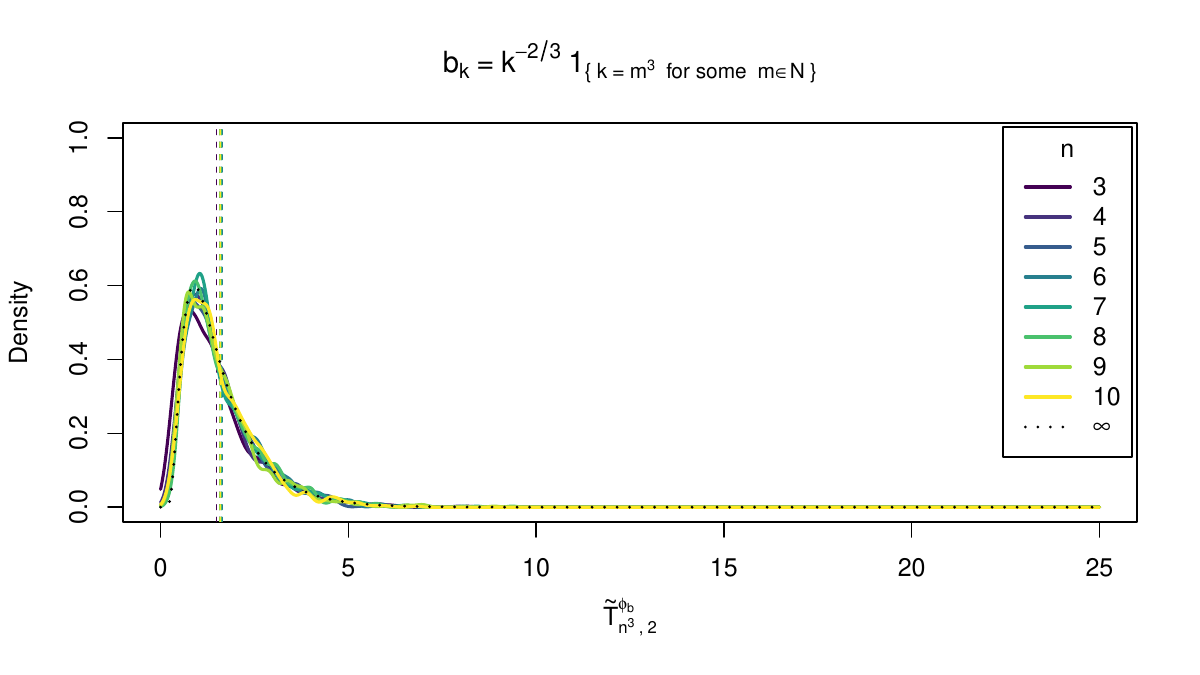}
    \end{subfigure}
    \begin{subfigure}{0.325\linewidth}
        \includegraphics[width=\linewidth, clip=true, trim={1cm, 0.6cm, 0cm, 1.75cm}]{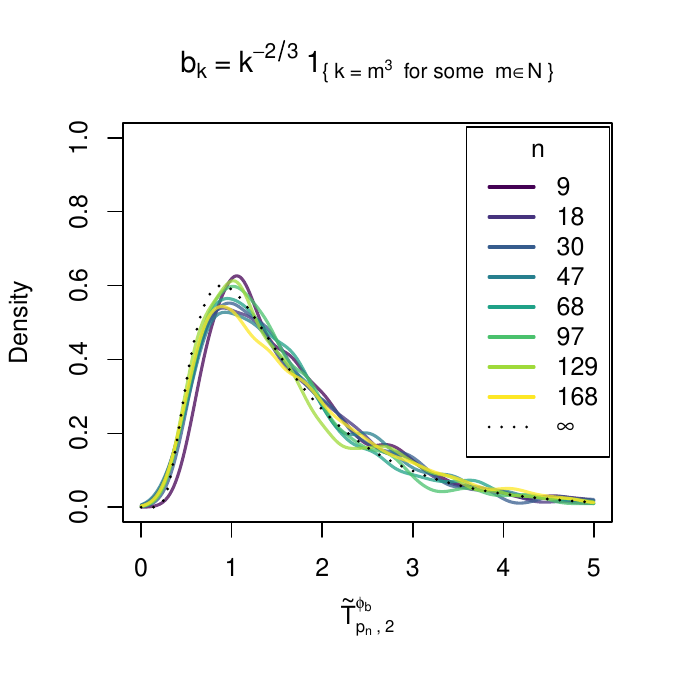}
    \end{subfigure}
    \caption{\small Top row: Density of $T_{N,2}^{\phi_b}$ for approximately balanced samples and sequences of sample sizes $\smash{N=(n^3)_{n\in\N}}$ (left) and $\smash{N=(p_n)_{n\in\N}}$ (right) under $\Hcal_0$. Bottom row: Same with the centered statistic, $\smash{\widetilde{T}^{\phi_b}_{N,2}}$. The density is estimated via kernel density estimation with $M=10^3$ Monte Carlo samples. Each Monte Carlo sample consists of $N$ iid observations drawn from a von Mises distribution with $\kappa=10$ and $\mu=\pi/2$. The vertical lines show the sample mean for each $N$. The dotted black density corresponds to the random variable $T_{\infty,2}^{\phi_b}$ given in Corollary~\ref{cor:asymp-H0-sum}.}
    \label{fig:asymp0-vMF-kappa10-k-2/3}
\end{figure}

\begin{remark}
    The absolute summability of $b_k(\phi)$ in \eqref{eq:summ-cond} is not sufficient to ensure $|b_k(\phi)|=O(k^{-(1+\delta)})$ for some $\delta>0$.
    Indeed, the sequence~\eqref{eq:bk-2/3} serves as a counterexample.
    Along the subsequence $k=m^3$, one has $b_k=k^{-2/3}$. Hence, $b_k$ is not $O(k^{-(1+\delta)})$ for any $\delta>0$.
\end{remark}

To remove the deterministic contribution from the resonant harmonics $k\in N\Z_+$, define the centered statistic
$$\widetilde{T}^{\phi,m}_{N,c}:=m\bigg(\Big(\sum_{\substack{k\geq1\\k\notin N\Z_{+}}}\sum_{r=1}^{2} 2^{-1}b_k(\phi) n_{\ell}G_{k,r,\ell,N}^2\Big)_{\ell=1}^{c}\bigg),$$
which in the case of the sum merging function reduces to the shifted version
\begin{align*}
    \widetilde{T}^{\phi}_{N,c}:=T_{N,c}^{\phi}-N\sum_{k\in N\Z_+}b_k(\phi).
\end{align*}

This modification eliminates the divergence in probability on the subsequences along which the uncentered statistic diverges and, as the next result shows, yields convergence in law of $\widetilde{T}^{\phi,m}_{N,c}$. Note that the centering term is always finite by the summability condition~\eqref{eq:summ-cond}.
In addition, after the centering, $\smash{\Esbig{\widetilde{T}^{\phi}_{N,c}}{0}}$ under $\Hcal_0$ is uniformly bounded for all $N\geq 2c$ and is given by
\begin{align*}
    \Esbig{\widetilde{T}^{\phi}_{N,c}}{0}&=\frac{(c-1)N}{N-1}\sum_{k\notin N\Z_+}b_k(\phi).
\end{align*}

\begin{corollary}\label{cor:centered_asymp}
    Let $m:\R^c\to\R$ be a function satisfying \eqref{m:lip}. Under $\Hcal_0$ and as $N\to\infty$, $\widetilde{T}_{N,c}^{\phi,m}\inlaw T_{\infty,c}^{\phi,m}$,
    with $T_{\infty,c}^{\phi,m}$ defined in~\eqref{eq:Tcmp_asymp_H0_general}.
\end{corollary}

\section{Non-null asymptotics}
\label{sec:nonnull-asymp}

In this section, we derive the results under alternatives, using the non-extreme asymptotic regime and notation introduced in Section~\ref{sec:null-asymp}.

\subsection{Consistency against fixed alternatives}
\label{subsec:nonnull:consistency}

We consider fixed absolutely continuous distributions ${\rm P}_1, {\rm P}_2,\ldots,{\rm P}_c$ that do not depend on $N$. For each $\ell=1,\ldots,c$, let $F_\ell$ denote the angular cdf of ${\rm P}_\ell$, with accumulation origin at $\theta=0$.

Denote the pooled cdf by $H_N:=\sum_{\ell=1}^{c}\pi_{\ell,N}F_\ell$ and its limit as $N\to\infty$ by
\begin{align}\label{eq:H_limit}
    H:=\sum_{\ell=1}^{c}\pi_\ell F_\ell.
\end{align}

We say that $({\rm P}_1,{\rm P}_2,\ldots,{\rm P}_c)$ is a fixed alternative if it belongs to
$$
\Hcal_{1,{\rm fix}}:= \{({\rm P}_1,{\rm P}_2,\ldots,{\rm P}_c): \exists i\neq j\text{ such that }{\rm P}_i\neq {\rm P}_j\}.
$$

The next two results show that the test based on a statistic with a kernel $\phi$ with $b_k(\phi)>0$ for all $k\geq 1$ is universally consistent. This implies that of all the proposed kernels in Section~\ref{sec:specific}, only the Rayleigh, the $q$-modal, and  the Rothman (with $t\in\mathbb{Q}$) tests are not omnibus.

\begin{theorem}\label{thm:asymp-Halt-general}
    Let $m:\R^c\to\R$ be a function satisfying \eqref{m:lip}--\eqref{m:hom}. Under the fixed-distribution setting, as $N\to\infty$, almost surely
    \begin{align*}
        N^{-1}T_{N,c}^{\phi,m}\to m_{\infty}\bigg(\Big(\pi_{\ell}\sum_{k=1}^{\infty} b_k(\phi) \|\bga_{\ell,k}\|^2\Big)_{\ell=1}^{c}\bigg),
    \end{align*}
    where $m_{\infty}$ is the function appearing in \eqref{m:hom}, and $\bga_{\ell,k}:=(\gamma_{\ell,k,1},\gamma_{\ell,k,2})^{\top}$ is given by
    $$
    \gamma_{\ell,k,r}:=\dfrac{1}{\sqrt{2}}\int_{0}^{2\pi}g_{k,r}(2\pi H(\theta))\,\rd F_{\ell}(\theta),\qquad r=1,2.
    $$
\end{theorem}

Under $\Hcal_{1,{\rm fix}}$, there exist $\ell_1\neq\ell_2$ and $k_1,k_2\geq1$ such that $\bga_{\ell_j,k_j}\neq \mathbf{0}$ for $j=1,2$. Thus, the limit in Theorem~\ref{thm:asymp-Halt-general} is positive due to \eqref{m:pos} as long as $\phi$ has positive Fourier coefficients, and the following consistency result follows.

\begin{corollary}\label{cor:consistency-Halt-general}
    Let $m:\R^c\to\R$ be a merging function and let $\phi$ be a kernel such that $b_k(\phi)>0$ for all $k\geq 1$.
    Then, the test based on $T_{N,c}^{\phi,m}$ that rejects for large values is consistent against~$\Hcal_{1,{\rm fix}}$, i.e., $$\mathrm{P}(T_{N,c}^{\phi,m}>c)\to 1\text{ as }N\to\infty\text{ for any }c\in\R.$$
\end{corollary}

\subsection{Asymptotics under local alternatives}
\label{subsec:nonnull:localt}

In this section, we derive the asymptotic behavior of the truncated test statistic under a sequence of local alternatives indexed by the total sample size $N$. For each $N$ with $\min_{1\leq \ell\leq c} n_\ell\geq 2$ and every $\ell=1,\ldots,c$, let $\Theta_{N,1}^{(\ell)},\ldots,\Theta_{N,n_\ell}^{(\ell)}$ be iid random variables on $[0,2\pi)$ with cdf $F_{\ell,N}$, and assume that the $c$ samples are mutually independent.

Let $F$ denote an angular cdf with starting point $\theta=0$, whose density $f$ has a $2\pi$-periodic extension belonging to $C^1(\R)$. For some $\bdelta\in\R^c$, let $\delta_\ell$ be its $\ell$th component and set $\Delta_{\ell,N}:=N^{-1/2}\delta_\ell$. We consider the location-shift local model
\begin{align}\label{eq:localt}
\Theta_{N,i}^{(\ell)} &\sim F_{\ell,N}, \quad i=1,\ldots,n_\ell,\quad\text{with}\quad f_{\ell,N}(\theta)= f(\theta+\Delta_{\ell,N}),\quad \ell=1,\ldots,c,
\end{align}
and write $H_N:=\sum_{\ell=1}^c \pi_{\ell,N}F_{\ell,N}$.
Thus, the sequence of alternatives approaches the null hypothesis at the rate $N^{-1/2}$.

We first derive the local asymptotic behavior of the averages of the spherical harmonics evaluated at the uniform scores, from which the local asymptotic distribution of the truncated statistic 
\begin{align*}
    T_{N,c,K}^{\phi,m}
    :=\,&m\bigg(\Big(\sum_{k=1}^{K}\sum_{r=1}^{2} 2^{-1}b_k(\phi) n_{\ell}G_{k,r,\ell,N}^2\Big)_{\ell=1}^{c}\bigg)
\end{align*}
is derived in Theorem~\ref{thm:asymp-Hlocalt-general}. Then, Corollary~\ref{cor:asymp-Hlocalt} gives an explicit expression for the case of the truncated sum statistic.

\begin{proposition}\label{prp:gkr-asymp-Hlocalt}
    Let $K\geq 1$ be fixed and let $\bG_N$ be as defined in Proposition~\ref{prp:gkr-asymp-H0}. Then, under the sequence of location-shift local alternatives~\eqref{eq:localt} and as $N\to\infty$,
    \begin{align*}
        N^{1/2}\bG_N\inlaw \mathcal{N}_D(\bbeta,\bSigma_0),
    \end{align*}
    with $D:=2Kc$, $\bbeta:=(\eta_{k,r,\ell})$ where
    \begin{align*}
        \eta_{k,r,\ell}:=\beta_{k,r}(\bpi-\be_{\ell})^{\top}\bdelta,
    \end{align*}
    with $\be_\ell$ denoting the $\ell$th canonical vector of $\R^c$,
    \begin{align*}
        \beta_{k,1}=-2\pi\sqrt{2}\,k\int_{0}^{2\pi}\sin(2\pi kF(\theta))f(\theta)^2\,\rd \theta\quad\text{and}\quad
        \beta_{k,2}=2\pi\sqrt{2}\,k\int_{0}^{2\pi}\cos(2\pi kF(\theta))f(\theta)^2\,\rd \theta,
    \end{align*}
    and covariance matrix $\bSigma_0$ defined in Proposition~\ref{prp:gkr-asymp-H0}.
\end{proposition}

\begin{theorem}\label{thm:asymp-Hlocalt-general}
    Fix $K\geq1$. Let $m:\R^c\to\R$ be a function satisfying \eqref{m:lip}. Under the sequence of location-shift local alternatives~\eqref{eq:localt} and as $N\to\infty$,
    \begin{align*}
        T_{N, c, K}^{\phi, m}\inlaw T_{\infty,c,K}^{\phi,m,{\rm loc}}:=m\bigg(\Big(\pi_\ell \sum_{k=1}^{K}\sum_{r=1}^{2} 2^{-1}b_k(\phi) Z_{k,r,\ell}^2\Big)_{\ell=1}^{c}\bigg),
    \end{align*}
    where $\{\bZ_{k,r}:=(Z_{k,r,1},\ldots,Z_{k,r,c})\sim \mathcal{N}_{c}(\bbeta_{k,r}, \bS): 1\leq k\leq K, r=1,2\}$ is a collection of mutually independent random vectors, and $\bbeta_{k,r}:=(\eta_{k,r,1},\ldots,\eta_{k,r,c})^{\top}$.
\end{theorem}

\begin{corollary}\label{cor:asymp-Hlocalt}
    Fix $K\geq1$. Under the sequence of location-shift local alternatives~\eqref{eq:localt} and as $N\to\infty$,
    \begin{align*}
        T_{N,c,K}^{\phi}\inlaw T_{\infty,c,K}^{\phi,{\rm loc}}:=\sum_{k=1}^{K}2^{-1}b_k(\phi)Y_k,
    \end{align*}
    where $\{Y_k\sim\chi^2_{2(c-1)}(\lambda_k): 1\leq k\leq K\}$ is a collection of mutually independent random variables, with
    \begin{align*}
        \lambda_k
        &=(\beta_{k,1}^2+\beta_{k,2}^2)\lrp{\bdelta^{\top}\bPi\bdelta-(\bpi^{\top}\bdelta)^2},
    \end{align*}
    where $\beta_{k,r}$ are defined in Proposition~\ref{prp:gkr-asymp-Hlocalt}, and $\bPi:=\diag{\bpi}$.
\end{corollary}

\begin{remark}
    Note that $\lambda_k=0$ for all $k\geq 1$ if and only if one of the following conditions holds:
    \begin{enumerate}[label=(\textit{\roman*}),ref=\textit{\roman*}]
        \item $\beta_{k,1}=\beta_{k,2}=0$ for all $k\geq 1$, that is, $F$ is the cdf of the uniform distribution, and all samples are equally distributed.
        \item $\bdelta\propto\mathbf{1}$, that is, all the samples undergo the same local shift and, thus, are equally distributed.
    \end{enumerate}
\end{remark}

We leave the extension of Theorem~\ref{thm:asymp-Hlocalt-general} to $K=\infty$ for future work. Such an extension would require a decay condition on the sequence $b_k(\phi)$, analogous to the null case, to control the contribution of resonant harmonics. However, the proof seems substantially more delicate, since it requires knowledge of the exact moments of spherical harmonics under the local alternative scenario.

\section{Numerical experiments}
\label{sec:sim}

\subsection{Null asymptotic distribution of test statistics}
\label{subsec:sim:asymp0}

We empirically validate the asymptotic distribution of $\smash{T_{N,c}^{\phi}}$ under $\Hcal_0$. We consider the common distribution to be a von Mises distribution with density $f_{\rm vM}(\theta; \mu, \kappa)=c_{{\rm vM},\kappa}\exp({\kappa \cos(\theta-\mu)})$, where $c_{{\rm vM},\kappa}$ denotes the normalizing constant, $\mu = \pi/2$, and $\kappa = 10$. A total of $M=10^4$ Monte Carlo replications are generated under $\Hcal_0$. For each $c\in\{3,5,10\}$, the $c$ samples are balanced with $n_\ell=\floor{500/c}$ observations, $\ell=1,\ldots, c$. The statistics $T_{N,c}^{\phi}$ corresponding to ${\rm R}_{1/2}$, $\rm W$, $\rm AD$, ${\rm SM}_{1}$, and ${\rm P}_{1/2}$ are computed.

Figure~\ref{fig:asymp0-vMF-kappa10} shows the resulting histograms along with the corresponding asymptotic densities. The latter are approximated by differentiating the truncated distribution $x\mapsto\mathrm{P}(\sum_{k=1}^{K_{\rm tr}}2^{-1}b_k(\phi)Y_k\leq x)$ with $K_{\rm tr}=10^3$, using the exact method of \cite{Imhof1961} and a faster approximation, the Hall--Buckley--Eagleson (HBE) method, which performs a three-moment match to a Gamma distribution, see \cite{Buckley1988}. Both methods yield virtually the same curves, and the histograms of the exact statistics agree remarkably well with the asymptotic curves.

\begin{figure}[ht!]
    \centering
    \includegraphics[width=1\linewidth,trim={0cm 0cm 0cm 1.2cm},clip]{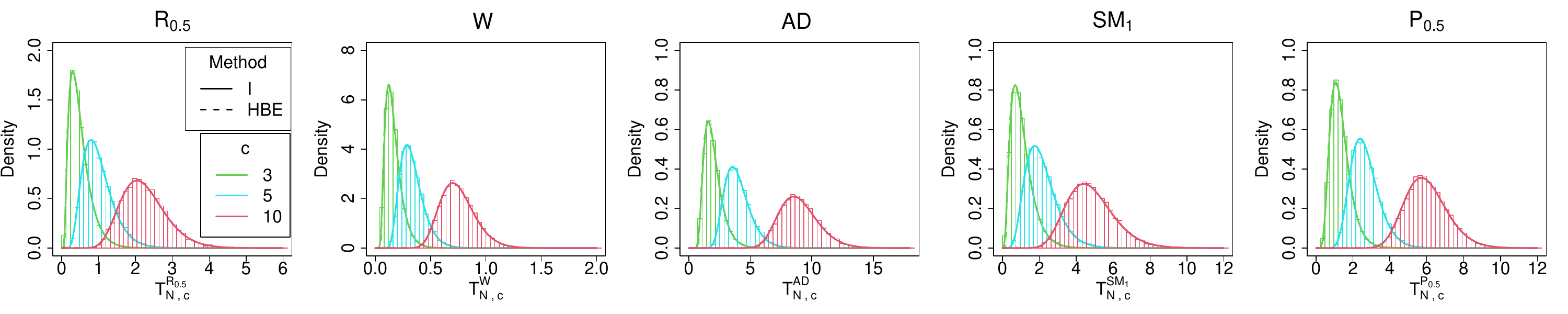}
    \caption{\small Histograms of $T_{N,c}^{\phi}$ with $N=c\floor{500/c}$ against their asymptotic densities computed using \cite{Imhof1961} and HBE methods. The kernel $\phi$ is indicated by the column.}
    \label{fig:asymp0-vMF-kappa10}
\end{figure}

\subsection{Power under fixed alternatives}
\label{subsec:sim:pow}

In this section, we study the finite-sample power of the tests based on $T_{N,c}^{\phi}$ for the kernels introduced in Section~\ref{sec:specific} under certain fixed alternatives. The common framework consists of $c\in\{3,5,10\}$ samples, only one of which has a different distribution, i.e., ${\rm P}_1=\cdots={\rm P}_{c-1}$ and ${\rm P}_c\neq {\rm P}_1$. Two types of alternatives are considered: (\emph{i}) location-shift alternatives, where the angular density $f_1=\cdots=f_{c-1}=f(\cdot;\mu,\kappa)$ while $f_c=f(\cdot;\mu+\Delta_\mu,\kappa)$, and (\emph{ii}) different concentration alternatives, where $f_1=\cdots=f_{c-1}=f(\cdot;\mu,\kappa)$ and $f_c=f(\cdot;\mu,\kappa+\Delta_\kappa)$. Within each scenario, all subsamples are generated from the same parametric family with density $f$, the specific values of $\Delta_\mu$ and $\Delta_\kappa$ are indicated in the captions, and the following distribution families are considered:
\begin{enumerate}[label={}]
    \item[(vM)] The von Mises distribution with location parameter $\mu=\pi/2$ and concentration $\kappa>0$.\label{fsim1}
    \item[(W)] The Watson distribution, with density $f_{\rm W}(\theta; \mu,\kappa)= c_{{\rm W},\kappa}\exp(\kappa \cos^2(\theta-\mu))$.\label{fsim2}
    \item[(SC)] The Small Circle distribution, with density $f_{\rm SC}(\theta; \mu,\kappa,\nu)= c_{{\rm SC, \kappa, \nu}}\exp(-\kappa(\cos(\theta-\mu)-\nu)^2)$ and $\nu=1/3$, which controls the modal locations.\label{fsim3}
    \item[(MvM)] A mixture of $m$ equally weighted vM distributions with density $f_{\rm MvM}(\theta;\mu,\kappa,m)=\break m^{-1}\sum_{j=1}^{m}f_{\rm vM}(\theta;\mu_j,\kappa)$ where $\mu_j := \mu + 2\pi (j-1)/m$ for $j=1,\ldots, m$, and common concentration $\kappa$. Considered for $m=3,4$.\label{fsim4}
\end{enumerate}

For each alternative scenario with fixed $n$, $c$, and $f$, $M=10^4$ Monte Carlo samples are generated. Each sample consists of a total of $N=nc$ observations, with $n\in\{25, 50, 100, 200\}$ observations corresponding to each subsample. For every Monte Carlo sample, the tests based on $T^{\phi}_{N,c}$ are computed using the asymptotic critical value at significance level $\alpha=0.05$, which is obtained using \cite{Imhof1961}'s method as in Section~\ref{subsec:sim:asymp0}, and kernels $\phi\in\{{\rm R},{\rm M}_2, {\rm W}, {\rm A}, {\rm R}_t, {\rm AD}, {\rm SM}_{\kappa}, {\rm P}_{\rho}\}$. The empirical rejection proportions for the tests are reported in Tables~\ref{tbl:pow-shift}--\ref{tbl:pow-conc}, separately for location and concentration alternatives.

The simulation results lead to the following conclusions.
\begin{enumerate}[label=(\textit{\roman*}),ref=\textit{\roman*}]
    \item The test $\rm R$ is not generally the most powerful test, even when all samples belong to the vM family. Nevertheless, in that setting, its power remains close to that of the best-performing tests, which are typically $\rm P_{0.5}$, ${\rm R}_{t}$, $t\in\{1/3,1/4\}$, and $\rm AD$.
    \item The ${\rm M}_2$ test outperforms the competing procedures under the bimodal alternatives (W and SC) in which the contaminated sample follows a bimodal distribution whose modes are located between (i.e., alternate with respect to) the modes of the remaining samples.\label{conc:bimod}
    \item Under the multimodal alternatives (MvM), the highest power is attained by ${\rm SM}_{10}$ and ${\rm P}_{0.9}$. The next best-performing test is ${\rm AD}$, but with a significant power gap.
    \item The tests $\rm W$ and $\rm AD$ exhibit similar power performance across scenarios. However, $\rm AD$ consistently achieves slightly higher rejection rates.
    \item Due to the similarity of the coefficients defining $\rm W$ and the Area-based tests, they both display comparable performance under vM and multimodal alternatives, with low rejection rates in the latter. In contrast, under bimodal alternatives of the type described in \eqref{conc:bimod}, $\rm W$ yields significantly higher rejection proportions. This behavior has also been noted on the real line, see \cite{Shirahata1990}.
    \item \label{conc:power-c-onediff} For fixed $n$, power generally decreases as $c$ increases across the kernels considered. Since only one of the $c$ samples differs from the others, increasing $c$ dilutes the relative contribution of the affected sample in the sum statistic. The effect is mainly visible before the power saturates, and it disappears as $n$ grows.
\end{enumerate}

For representative ${\rm AD}$ and ${\rm SM}_{10}$ kernels, additional simulations to study the effect of imbalanced samples are available in Section~\ref{subsec:sims:imbalanced} of the~SM.

\begin{table}[!htbp]
\centering
\scalebox{0.67}{
\begin{tabular}{ccc|r|r|r|r|rrr|r|rrr|rrr}
\toprule
\multirow{2}{*}{Family} & \multirow{2}{*}{$c$} & \multirow{2}{*}{$n$} & \multicolumn{1}{c|}{\multirow{2}{*}{R}} & \multicolumn{1}{c|}{\multirow{2}{*}{${\rm M}_2$}} & \multicolumn{1}{c|}{\multirow{2}{*}{W}} & \multicolumn{1}{c|}{\multirow{2}{*}{A}} & \multicolumn{3}{c|}{${\rm R}_t$} & \multicolumn{1}{c|}{\multirow{2}{*}{AD}} & \multicolumn{3}{c|}{${\rm SM}_{\kappa}$} & \multicolumn{3}{c}{${\rm P}_{\rho}$} \\ 
& & & & & & & $1/4$ & $1/3$ & $1/2$ & & $0.1$ & $1$ & $10$ & $0.1$ & $0.5$ & $0.9$\\
\midrule
\multirow{12}{*}{$\substack{{\rm vM}\\\\\centering\includegraphics[width=1.75cm, clip=true, trim={1.6cm 1.5cm 1.75cm 0.5cm}]{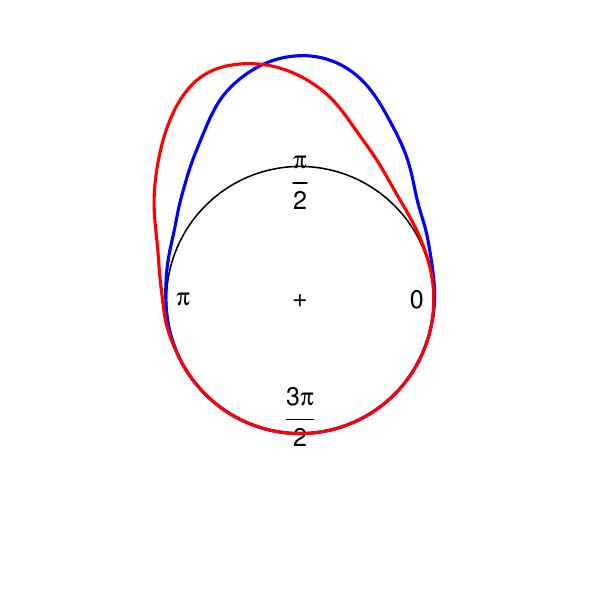}}$} & \multirow{4}{*}{$3$} & 25 & 28.3 & 10.2 & $\mathbf{30.7}$ & 30.2 & 30.4 & 29.9 & 28.9 & $\mathbf{31.2}$ & 28.6 & 30.0 & 24.7 & 29.0 & 30.2 & 11.1 \\ 
& & 50 & 55.5 & 18.3 & 60.0 & 59.1 & $\mathbf{60.8}$ & 59.3 & 56.6 & $\mathbf{61.2}$ & 55.9 & 59.1 & 56.6 & 57.2 & $\mathbf{61.1}$ & 40.9 \\ 
& & 100 & 87.9 & 34.7 & 91.7 & 91.2 & 92.2 & 91.1 & 88.8 & 92.4 & 88.4 & 91.0 & 91.6 & 89.4 & $\mathbf{93.0}$ & 84.1 \\ 
& & 200 & 99.6 & 66.0 & $\mathbf{99.9}$ & $\mathbf{99.9}$ & $\mathbf{99.9}$ & $\mathbf{99.9}$ & 99.7 & $\mathbf{99.9}$ & 99.6 & $\mathbf{99.9}$ & $\mathbf{99.9}$ & 99.7 & $\mathbf{99.9}$ & 99.8 \\ 
\cline{2-17}
& \multirow{4}{*}{$5$} & 25 & 23.8 & 9.5 & $\mathbf{25.5}$ & 25.1 & $\mathbf{25.7}$ & 25.1 & 24.3 & $\mathbf{26.1}$ & 23.9 & 24.9 & 22.4 & 24.4 & $\mathbf{25.7}$ & 13.7 \\ 
& & 50 & 50.9 & 16.5 & 55.6 & 54.7 & $\mathbf{56.0}$ & 54.7 & 52.3 & $\mathbf{56.5}$ & 51.3 & 54.2 & 51.4 & 52.7 & $\mathbf{56.5}$ & 37.3 \\ 
& & 100 & 86.3 & 33.8 & 90.8 & 90.2 & 91.4 & 90.3 & 87.4 & $\mathbf{91.8}$ & 86.8 & 89.9 & 90.6 & 88.2 & $\mathbf{92.2}$ & 81.6 \\ 
& & 200 & 99.6 & 66.1 & $\mathbf{99.9}$ & 99.8 & $\mathbf{99.9}$ & 99.8 & 99.7 & $\mathbf{99.9}$ & 99.6 & 99.8 & $\mathbf{99.9}$ & 99.7 & $\mathbf{99.9}$ & 99.7 \\ 
\cline{2-17}
& \multirow{4}{*}{$10$} & 25 & 18.2 & 8.7 & $\mathbf{19.8}$ & $\mathbf{19.5}$ & $\mathbf{20.1}$ & $\mathbf{19.5}$ & 18.7 & $\mathbf{20.2}$ & 18.3 & 19.4 & 18.1 & 18.8 & $\mathbf{20.0}$ & 13.4 \\ 
& & 50 & 39.7 & 14.2 & 43.9 & 43.0 & 44.0 & 43.4 & 40.7 & $\mathbf{45.2}$ & 40.2 & 42.6 & 40.7 & 41.2 & $\mathbf{44.9}$ & 30.8 \\ 
& & 100 & 76.4 & 27.6 & 82.3 & 81.2 & 83.3 & 81.3 & 78.1 & $\mathbf{83.5}$ & 77.0 & 80.7 & 80.5 & 78.4 & $\mathbf{84.0}$ & 68.4 \\ 
& & 200 & 99.1 & 57.9 & $\mathbf{99.7}$ & 99.6 & $\mathbf{99.7}$ & 99.6 & 99.3 & $\mathbf{99.8}$ & 99.1 & 99.6 & 99.6 & 99.3 & $\mathbf{99.8}$ & 98.4 \\ 
\midrule
\multirow{12}{*}{$\substack{{\rm W}\\\\\centering\includegraphics[width=1.75cm, clip=true, trim={1.6cm 1.5cm 1.75cm 0.5cm}]{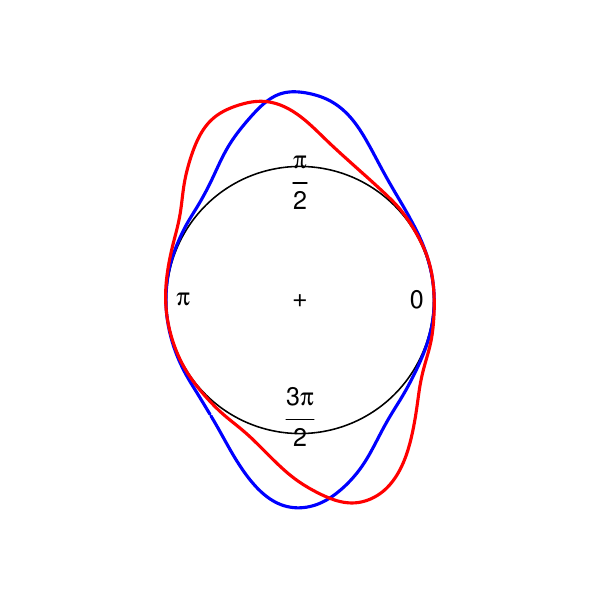}}$} & \multirow{4}{*}{$3$} & 25 & 6.2 & $\mathbf{46.2}$ & 11.8 & 10.2 & 21.6 & 11.6 & 6.4 & 14.6 & 6.6 & 11.1 & 32.5 & 7.8 & 22.1 & 16.8 \\ 
& & 50 & 6.5 & $\mathbf{83.7}$ & 24.5 & 18.7 & 52.5 & 24.0 & 6.7 & 34.8 & 7.3 & 21.7 & 73.0 & 10.6 & 54.5 & 63.1 \\ 
& & 100 & 6.0 & $\mathbf{99.3}$ & 65.3 & 49.2 & 92.0 & 64.0 & 6.3 & 81.4 & 7.6 & 56.7 & 98.6 & 16.7 & 94.0 & 97.9 \\ 
& & 200 & 6.0 & $\mathbf{100.0}$ & 99.2 & 96.6 & $\mathbf{100.0}$ & 99.1 & 6.5 & $\mathbf{99.9}$ & 10.2 & 97.7 & $\mathbf{100.0}$ & 44.7 & $\mathbf{100.0}$ & $\mathbf{100.0}$ \\ 
\cline{2-17}
& \multirow{4}{*}{$5$} & 25 & 5.6 & $\mathbf{44.6}$ & 10.9 & 9.5 & 18.8 & 10.7 & 5.7 & 13.3 & 5.9 & 10.1 & 29.8 & 7.0 & 19.9 & 21.3 \\ 
& & 50 & 5.7 & $\mathbf{83.2}$ & 21.9 & 16.8 & 45.6 & 21.3 & 6.1 & 30.8 & 6.5 & 18.9 & 69.3 & 9.6 & 48.1 & 61.3 \\ 
& & 100 & 6.0 & $\mathbf{99.5}$ & 55.8 & 41.0 & 88.2 & 54.9 & 6.3 & 73.4 & 7.5 & 47.1 & 98.1 & 15.8 & 90.8 & 97.0 \\ 
& & 200 & 5.4 & $\mathbf{100.0}$ & 97.8 & 92.1 & $\mathbf{100.0}$ & 97.6 & 5.8 & 99.8 & 9.0 & 94.8 & $\mathbf{100.0}$ & 36.4 & $\mathbf{100.0}$ & $\mathbf{100.0}$ \\ 
\cline{2-17}
& \multirow{4}{*}{$10$} & 25 & 5.1 & $\mathbf{35.7}$ & 9.3 & 8.2 & 14.7 & 9.3 & 5.3 & 11.7 & 5.3 & 8.4 & 22.8 & 6.2 & 15.2 & 19.8 \\ 
& & 50 & 5.3 & $\mathbf{73.5}$ & 17.0 & 13.5 & 32.4 & 16.8 & 5.5 & 23.0 & 5.9 & 14.6 & 54.8 & 8.1 & 35.0 & 50.4 \\ 
& & 100 & 5.4 & $\mathbf{98.5}$ & 41.2 & 30.3 & 75.3 & 40.3 & 5.6 & 57.2 & 6.5 & 33.5 & 94.6 & 12.9 & 79.4 & 92.7 \\ 
& & 200 & 5.3 & $\mathbf{100.0}$ & 89.4 & 74.5 & 99.6 & 88.2 & 5.4 & 97.7 & 8.3 & 79.6 & $\mathbf{100.0}$ & 25.7 & 99.8 & $\mathbf{100.0}$ \\ 
\midrule
\multirow{12}{*}{$\substack{{\rm SC}\\\\\centering\includegraphics[width=1.75cm, clip=true, trim={1.6cm 1.5cm 1.75cm 0.5cm}]{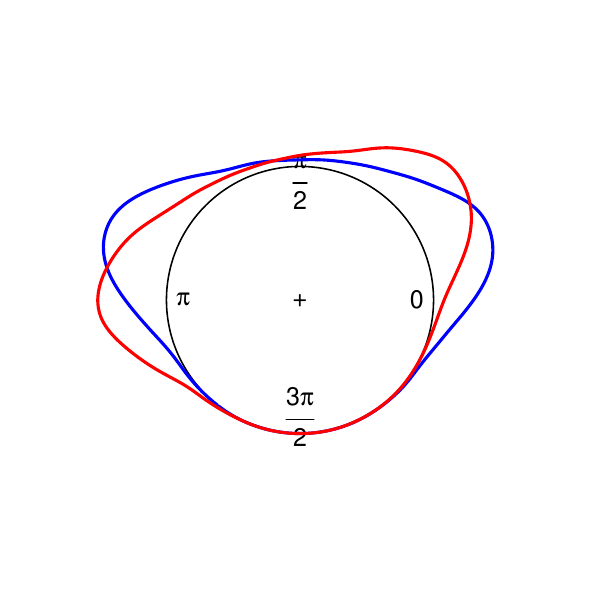}}$} & \multirow{4}{*}{$3$} & 25 & 6.4 & $\mathbf{36.3}$ & 10.1 & 9.2 & 16.7 & 10.0 & 6.5 & 12.1 & 6.7 & 9.6 & 24.6 & 7.5 & 16.9 & 11.8 \\ 
& & 50 & 6.6 & $\mathbf{72.0}$ & 19.5 & 15.5 & 40.2 & 18.9 & 6.8 & 26.8 & 7.2 & 17.6 & 58.9 & 9.8 & 41.9 & 46.6 \\ 
& & 100 & 7.2 & $\mathbf{97.3}$ & 49.1 & 36.6 & 82.3 & 47.9 & 7.8 & 65.3 & 8.7 & 42.9 & 94.6 & 15.5 & 84.5 & 91.1 \\ 
& & 200 & 8.6 & $\mathbf{100.0}$ & 94.5 & 85.8 & 99.6 & 93.9 & 9.4 & 98.6 & 12.0 & 89.9 & $\mathbf{100.0}$ & 34.6 & 99.8 & $\mathbf{100.0}$ \\ 
\cline{2-17}
& \multirow{4}{*}{$5$} & 25 & 5.8 & $\mathbf{34.8}$ & 10.2 & 8.9 & 15.7 & 10.1 & 6.1 & 12.1 & 6.1 & 9.4 & 23.1 & 7.0 & 16.0 & 15.6 \\ 
& & 50 & 5.7 & $\mathbf{70.1}$ & 16.7 & 13.6 & 33.8 & 16.2 & 6.1 & 22.5 & 6.3 & 14.8 & 53.1 & 8.7 & 35.8 & 43.8 \\ 
& & 100 & 6.6 & $\mathbf{97.0}$ & 41.6 & 30.5 & 76.7 & 40.2 & 7.1 & 57.8 & 7.8 & 35.1 & 92.6 & 13.5 & 79.4 & 88.9 \\ 
& & 200 & 8.3 & $\mathbf{100.0}$ & 90.0 & 77.6 & 99.5 & 89.0 & 9.2 & 97.4 & 11.6 & 83.1 & $\mathbf{100.0}$ & 30.7 & 99.6 & $\mathbf{100.0}$ \\ 
\cline{2-17}
& \multirow{4}{*}{$10$} & 25 & 5.1 & $\mathbf{27.0}$ & 8.3 & 7.5 & 12.1 & 8.3 & 5.2 & 9.8 & 5.3 & 7.9 & 18.0 & 6.1 & 12.4 & 14.8 \\ 
& & 50 & 5.3 & $\mathbf{58.6}$ & 13.8 & 11.4 & 25.0 & 13.6 & 5.6 & 18.1 & 5.7 & 12.3 & 41.6 & 7.7 & 26.4 & 36.1 \\ 
& & 100 & 6.0 & $\mathbf{93.8}$ & 30.0 & 22.3 & 60.0 & 29.2 & 6.4 & 42.1 & 6.9 & 24.9 & 83.2 & 11.3 & 63.0 & 78.5 \\ 
& & 200 & 6.6 & $\mathbf{100.0}$ & 75.3 & 58.4 & 97.2 & 73.4 & 7.4 & 90.3 & 9.1 & 64.4 & 99.8 & 22.1 & 98.0 & 99.7 \\ 
\midrule
\multirow{12}{*}{$\substack{{\rm MvM (3)}\\\\\centering\includegraphics[width=1.75cm, clip=true, trim={1.6cm 1.5cm 1.75cm 0.5cm}]{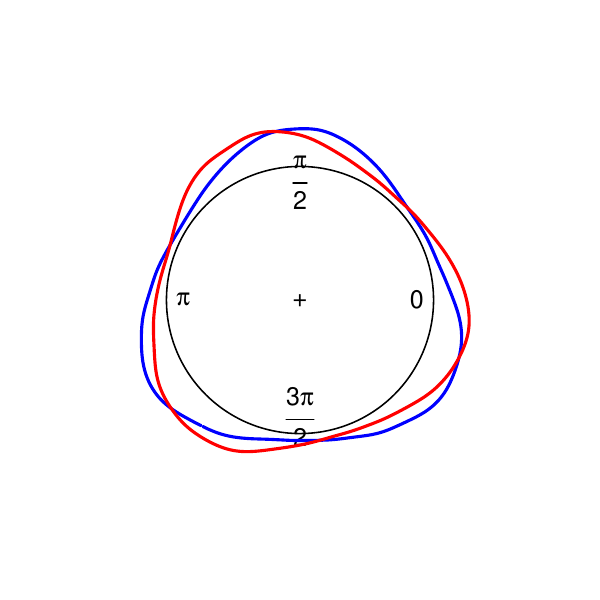}}$} & \multirow{4}{*}{$3$} & 25 & 5.7 & 5.9 & 6.3 & 6.2 & 6.3 & 5.7 & 6.3 & 7.0 & 5.7 & 5.8 & $\mathbf{9.7}$ & 5.8 & 7.3 & 4.2 \\ 
& & 50 & 5.3 & 6.8 & 7.2 & 6.7 & 7.4 & 5.6 & 6.7 & 8.2 & 5.3 & 6.0 & $\mathbf{21.0}$ & 5.5 & 10.4 & 17.8 \\ 
& & 100 & 4.9 & 6.7 & 9.5 & 8.2 & 9.9 & 5.3 & 8.8 & 13.0 & 4.9 & 6.2 & $\mathbf{47.4}$ & 5.2 & 18.6 & $\mathbf{47.0}$ \\ 
& & 200 & 5.3 & 7.0 & 17.5 & 13.4 & 18.3 & 5.7 & 15.8 & 29.5 & 5.3 & 7.9 & 88.4 & 5.9 & 48.7 & $\mathbf{89.6}$ \\ 
\cline{2-17}
& \multirow{4}{*}{$5$} & 25 & 5.5 & 5.8 & 6.2 & 6.1 & 6.4 & 5.8 & 6.1 & 6.6 & 5.5 & 5.9 & $\mathbf{9.7}$ & 5.7 & 7.2 & 6.9 \\ 
& & 50 & 5.3 & 5.7 & 7.2 & 6.6 & 7.4 & 5.5 & 7.1 & 8.4 & 5.3 & 5.9 & $\mathbf{18.0}$ & 5.4 & 9.7 & 16.9 \\ 
& & 100 & 4.8 & 5.9 & 8.8 & 7.5 & 9.1 & 5.1 & 8.4 & 12.2 & 4.8 & 5.9 & 41.4 & 5.1 & 17.0 & $\mathbf{42.4}$ \\ 
& & 200 & 4.8 & 5.5 & 15.3 & 12.0 & 15.5 & 5.2 & 14.4 & 24.3 & 4.8 & 7.3 & 82.7 & 5.3 & 39.0 & $\mathbf{84.9}$ \\ 
\cline{2-17}
& \multirow{4}{*}{$10$} & 25 & 4.8 & 4.9 & 5.5 & 5.3 & 5.5 & 5.1 & 5.4 & 5.9 & 4.7 & 4.9 & $\mathbf{8.0}$ & 4.9 & 6.0 & $\mathbf{8.1}$ \\ 
& & 50 & 5.1 & 5.2 & 6.7 & 6.3 & 6.8 & 5.3 & 6.5 & 7.6 & 5.1 & 5.6 & $\mathbf{14.2}$ & 5.2 & 8.4 & $\mathbf{13.9}$ \\ 
& & 100 & 5.1 & 5.2 & 8.3 & 7.4 & 8.3 & 5.4 & 8.3 & 10.5 & 5.1 & 6.1 & 30.8 & 5.3 & 13.4 & $\mathbf{32.6}$ \\ 
& & 200 & 4.9 & 5.3 & 13.0 & 10.2 & 13.0 & 5.1 & 12.5 & 19.0 & 4.8 & 6.8 & 70.3 & 5.3 & 29.0 & $\mathbf{73.1}$ \\ 
\midrule
\multirow{12}{*}{$\substack{{\rm MvM (4)}\\\\\centering\includegraphics[width=1.75cm, clip=true, trim={1.6cm 1.5cm 1.75cm 0.5cm}]{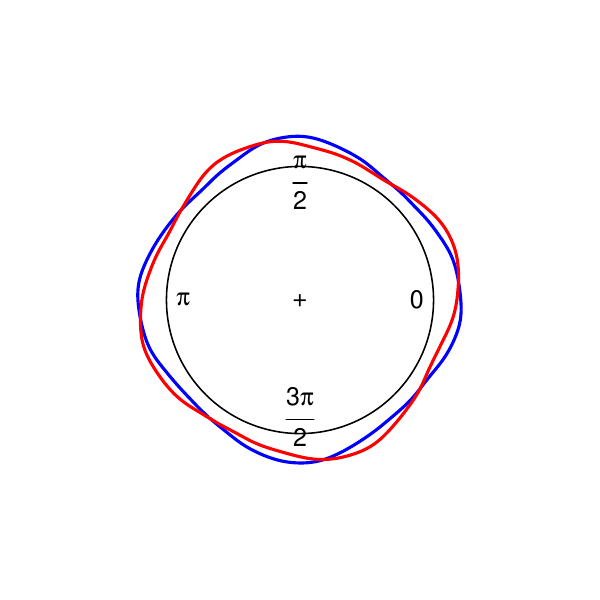}}$} & \multirow{4}{*}{$3$} & 25 & 5.4 & 5.4 & $\mathbf{5.5}$ & 5.4 & $\mathbf{5.5}$ & 5.4 & $\mathbf{5.5}$ & $\mathbf{5.6}$ & 5.4 & 5.4 & 5.4 & $\mathbf{5.5}$ & 5.4 & 2.4 \\ 
& & 50 & 5.1 & 5.4 & 5.3 & 5.2 & 5.0 & 5.2 & 5.0 & 5.4 & 5.0 & 5.0 & $\mathbf{8.2}$ & 5.1 & 5.6 & 6.8 \\ 
& & 100 & 4.9 & 5.3 & 5.8 & 5.5 & 4.9 & 5.8 & 5.0 & 6.4 & 4.9 & 5.1 & 13.1 & 5.0 & 6.8 & $\mathbf{14.9}$ \\ 
& & 200 & 5.0 & 5.3 & 7.0 & 6.4 & 5.2 & 6.7 & 5.1 & 8.6 & 5.0 & 5.1 & 29.0 & 5.0 & 9.6 & $\mathbf{35.4}$ \\ 
\cline{2-17}
& \multirow{4}{*}{$5$} & 25 & 4.8 & 4.9 & 5.1 & 4.9 & 4.9 & 5.2 & 5.0 & 5.3 & 4.8 & 4.9 & $\mathbf{5.9}$ & 4.9 & 5.1 & 4.1 \\ 
& & 50 & 5.2 & 5.5 & 5.5 & 5.3 & 5.3 & 5.4 & 5.2 & 5.8 & 5.2 & 5.1 & $\mathbf{7.6}$ & 5.2 & 5.7 & $\mathbf{7.3}$ \\ 
& & 100 & 5.3 & 5.2 & 5.9 & 5.7 & 5.2 & 5.8 & 5.3 & 6.6 & 5.2 & 5.2 & 12.4 & 5.2 & 7.0 & $\mathbf{13.6}$ \\ 
& & 200 & 4.9 & 5.5 & 6.8 & 6.3 & 5.2 & 6.8 & 5.1 & 8.2 & 4.9 & 5.1 & 24.3 & 4.9 & 9.0 & $\mathbf{30.9}$ \\ 
\cline{2-17}
& \multirow{4}{*}{$10$} & 25 & 5.3 & 4.9 & $\mathbf{5.8}$ & $\mathbf{5.7}$ & 5.2 & $\mathbf{5.6}$ & 5.5 & $\mathbf{5.9}$ & 5.3 & 5.4 & $\mathbf{5.9}$ & 5.4 & 5.5 & 5.4 \\ 
& & 50 & 4.7 & 5.1 & 5.2 & 5.1 & 5.1 & 5.3 & 4.8 & 5.6 & 4.7 & 5.0 & $\mathbf{7.3}$ & 4.7 & 5.6 & $\mathbf{7.6}$ \\ 
& & 100 & 5.1 & 5.2 & 5.9 & 5.6 & 5.2 & 5.8 & 5.2 & 6.5 & 5.1 & 5.2 & 10.4 & 5.2 & 6.5 & $\mathbf{12.0}$ \\ 
& & 200 & 4.9 & 5.5 & 6.5 & 6.1 & 5.2 & 6.4 & 5.1 & 7.6 & 4.9 & 4.9 & 19.1 & 4.9 & 8.0 & $\mathbf{23.4}$ \\ 
\bottomrule
\end{tabular}
}
\caption{\small Empirical rejection proportion (\%) of tests based on $T_{N,c}^{\phi}$ with $N = nc$ using $c$ balanced samples. In each simulation, $c-1$ samples are drawn from a distribution indicated by the first column, with location $\mu_0=\pi/2$ and concentration $\kappa=5$ (blue pdf in the charts), while the remaining sample comes from a shifted distribution with location $\mu_c=\mu_0 + \pi/10$ and the same concentration (red pdf). Results are based on $M = 10^4$ replications. Boldface indicates the row-wise maximum power across different kernels. Values within the lower one-sided $95\%$ confidence interval for the maximum power are also shown in bold.}
\label{tbl:pow-shift}
\end{table}

\begin{table}[!htbp]
\centering
\scalebox{0.67}{
\begin{tabular}{ccc|r|r|r|r|rrr|r|rrr|rrr}
\toprule
\multirow{2}{*}{Family} & \multirow{2}{*}{$c$} & \multirow{2}{*}{$n$} & \multicolumn{1}{c|}{\multirow{2}{*}{R}} & \multicolumn{1}{c|}{\multirow{2}{*}{${\rm M}_2$}} & \multicolumn{1}{c|}{\multirow{2}{*}{W}} & \multicolumn{1}{c|}{\multirow{2}{*}{A}} & \multicolumn{3}{c|}{${\rm R}_t$} & \multicolumn{1}{c|}{\multirow{2}{*}{AD}} & \multicolumn{3}{c|}{${\rm SM}_{\kappa}$} & \multicolumn{3}{c}{${\rm P}_{\rho}$} \\ 
& & & & & & & $1/4$ & $1/3$ & $1/2$ & & $0.1$ & $1$ & $10$ & $0.1$ & $0.5$ & $0.9$\\
\midrule
\multirow{12}{*}{$\substack{{\rm vM}\\\\\centering\includegraphics[width=1.75cm, clip=true, trim={1.6cm 1.5cm 1.75cm 0.5cm}]{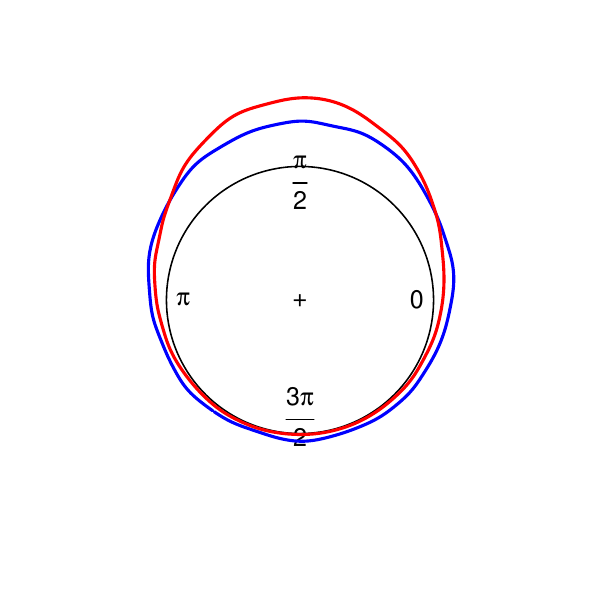}}$} & \multirow{4}{*}{$3$} & 25 & $\mathbf{23.2}$ & 6.8 & $\mathbf{22.7}$ & $\mathbf{22.7}$ & 22.1 & $\mathbf{22.6}$ & $\mathbf{23.1}$ & $\mathbf{22.9}$ & $\mathbf{23.3}$ & $\mathbf{22.8}$ & 15.0 & $\mathbf{23.2}$ & 21.2 & 5.7 \\ 
& & 50 & 45.6 & 9.0 & $\mathbf{46.6}$ & $\mathbf{46.4}$ & $\mathbf{46.3}$ & $\mathbf{46.5}$ & 45.8 & $\mathbf{46.8}$ & 45.6 & $\mathbf{46.5}$ & 35.9 & $\mathbf{46.2}$ & 45.3 & 21.1 \\ 
& & 100 & 80.4 & 14.7 & $\mathbf{81.9}$ & $\mathbf{81.8}$ & $\mathbf{81.6}$ & $\mathbf{82.0}$ & 80.5 & $\mathbf{82.0}$ & 80.6 & $\mathbf{82.0}$ & 72.8 & 81.2 & 81.3 & 53.9 \\ 
& & 200 & 98.8 & 27.7 & $\mathbf{99.2}$ & $\mathbf{99.2}$ & $\mathbf{99.2}$ & $\mathbf{99.2}$ & 98.8 & $\mathbf{99.3}$ & 98.8 & $\mathbf{99.1}$ & 98.2 & 99.0 & $\mathbf{99.2}$ & 93.2 \\ 
\cline{2-17}
& \multirow{4}{*}{$5$} & 25 & $\mathbf{20.2}$ & 5.9 & $\mathbf{19.6}$ & $\mathbf{19.6}$ & 18.5 & $\mathbf{19.7}$ & $\mathbf{20.1}$ & 19.5 & $\mathbf{20.2}$ & $\mathbf{19.6}$ & 13.4 & $\mathbf{20.0}$ & 17.9 & 7.5 \\ 
& & 50 & $\mathbf{42.2}$ & 7.2 & $\mathbf{42.3}$ & $\mathbf{42.2}$ & 40.7 & $\mathbf{42.6}$ & $\mathbf{42.3}$ & $\mathbf{42.2}$ & $\mathbf{42.4}$ & $\mathbf{42.3}$ & 30.3 & $\mathbf{42.6}$ & 40.0 & 19.2 \\ 
& & 100 & 78.8 & 9.9 & $\mathbf{79.7}$ & $\mathbf{79.6}$ & 78.9 & $\mathbf{79.6}$ & 78.8 & $\mathbf{79.4}$ & 78.9 & $\mathbf{79.5}$ & 65.8 & $\mathbf{79.3}$ & 78.1 & 45.4 \\ 
& & 200 & 99.1 & 17.5 & $\mathbf{99.3}$ & $\mathbf{99.3}$ & 99.1 & $\mathbf{99.3}$ & 99.1 & $\mathbf{99.3}$ & 99.1 & $\mathbf{99.3}$ & 97.5 & $\mathbf{99.2}$ & 99.1 & 88.9 \\ 
\cline{2-17}
& \multirow{4}{*}{$10$} & 25 & $\mathbf{15.8}$ & 5.1 & $\mathbf{16.0}$ & $\mathbf{15.9}$ & 14.9 & $\mathbf{16.0}$ & $\mathbf{16.0}$ & $\mathbf{15.8}$ & $\mathbf{15.7}$ & $\mathbf{15.7}$ & 11.7 & $\mathbf{15.9}$ & 14.6 & 7.7 \\ 
& & 50 & $\mathbf{32.6}$ & 6.0 & $\mathbf{32.3}$ & $\mathbf{32.3}$ & 30.5 & $\mathbf{32.3}$ & $\mathbf{32.7}$ & 32.0 & $\mathbf{32.8}$ & $\mathbf{32.1}$ & 21.6 & $\mathbf{32.8}$ & 29.3 & 14.1 \\ 
& & 100 & 68.7 & 7.3 & $\mathbf{69.3}$ & $\mathbf{69.2}$ & 66.9 & $\mathbf{69.7}$ & 68.7 & 68.6 & 68.8 & $\mathbf{69.2}$ & 49.9 & $\mathbf{69.2}$ & 65.6 & 32.5 \\ 
& & 200 & 97.6 & 10.3 & $\mathbf{98.0}$ & $\mathbf{97.9}$ & $\mathbf{97.7}$ & $\mathbf{98.0}$ & 97.6 & $\mathbf{97.9}$ & 97.7 & $\mathbf{98.0}$ & 91.3 & $\mathbf{97.9}$ & 97.4 & 73.8 \\ 
\midrule
\multirow{12}{*}{$\substack{{\rm W}\\\\\centering\includegraphics[width=1.75cm, clip=true, trim={1.6cm 1.2cm 1.75cm 0.5cm}]{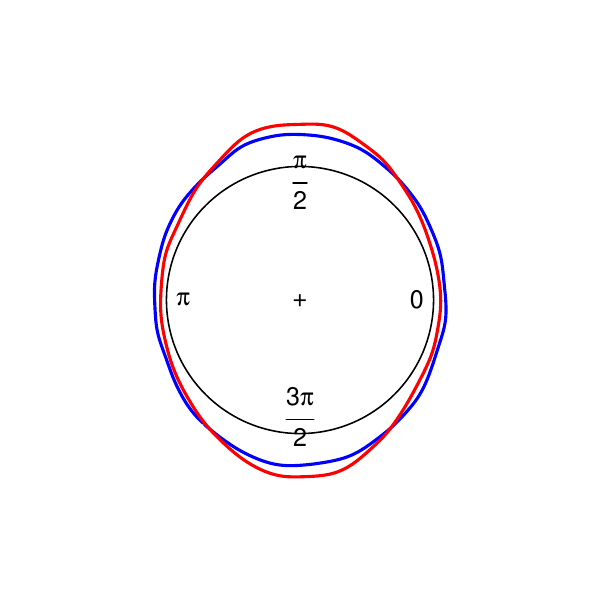}}$} & \multirow{4}{*}{$3$} & 25 & 5.4 & $\mathbf{12.5}$ & 6.3 & 5.9 & 7.7 & 6.2 & 5.6 & 6.7 & 5.4 & 6.2 & 8.0 & 5.6 & 7.4 & 2.8 \\ 
& & 50 & 5.3 & $\mathbf{22.4}$ & 7.1 & 6.6 & 10.7 & 7.1 & 5.2 & 7.9 & 5.4 & 7.0 & 13.3 & 5.8 & 10.4 & 8.0 \\ 
& & 100 & 5.4 & $\mathbf{46.4}$ & 10.6 & 9.0 & 19.9 & 10.5 & 5.5 & 13.0 & 5.8 & 10.2 & 27.7 & 6.9 & 19.9 & 19.3 \\ 
& & 200 & 5.1 & $\mathbf{80.2}$ & 20.7 & 15.8 & 45.6 & 20.6 & 5.0 & 28.0 & 5.8 & 19.4 & 59.6 & 8.8 & 46.0 & 46.8 \\ 
\cline{2-17}
& \multirow{4}{*}{$5$} & 25 & 4.8 & $\mathbf{11.0}$ & 5.6 & 5.4 & 6.5 & 5.5 & 4.9 & 5.8 & 4.9 & 5.5 & 7.4 & 5.0 & 6.4 & 4.4 \\ 
& & 50 & 5.4 & $\mathbf{19.7}$ & 7.3 & 6.8 & 10.0 & 7.3 & 5.3 & 8.1 & 5.5 & 7.1 & 12.2 & 6.0 & 9.7 & 9.7 \\ 
& & 100 & 5.2 & $\mathbf{42.4}$ & 10.0 & 8.7 & 18.0 & 9.9 & 5.2 & 12.1 & 5.4 & 9.5 & 24.6 & 6.4 & 17.9 & 17.5 \\ 
& & 200 & 4.9 & $\mathbf{78.9}$ & 17.6 & 13.6 & 40.5 & 17.7 & 5.1 & 23.5 & 5.5 & 16.6 & 53.5 & 8.1 & 39.8 & 39.2 \\ 
\cline{2-17}
& \multirow{4}{*}{$10$} & 25 & 4.6 & $\mathbf{9.1}$ & 5.5 & 5.3 & 6.3 & 5.5 & 4.8 & 5.7 & 4.8 & 5.3 & 6.7 & 4.9 & 5.9 & 5.3 \\ 
& & 50 & 5.2 & $\mathbf{16.0}$ & 6.9 & 6.5 & 8.7 & 6.9 & 5.4 & 7.6 & 5.4 & 6.8 & 10.5 & 5.9 & 8.6 & 7.9 \\ 
& & 100 & 5.1 & $\mathbf{33.4}$ & 8.7 & 7.5 & 14.0 & 8.7 & 5.2 & 10.2 & 5.3 & 8.4 & 17.8 & 6.1 & 13.6 & 13.3 \\ 
& & 200 & 5.0 & $\mathbf{69.0}$ & 14.1 & 11.4 & 29.6 & 14.0 & 5.1 & 17.8 & 5.6 & 13.3 & 40.2 & 7.7 & 29.0 & 27.9 \\ 
\midrule
\multirow{12}{*}{$\substack{{\rm SC}\\\\\centering\includegraphics[width=1.75cm, clip=true, trim={1.6cm 1.5cm 1.75cm 0.5cm}]{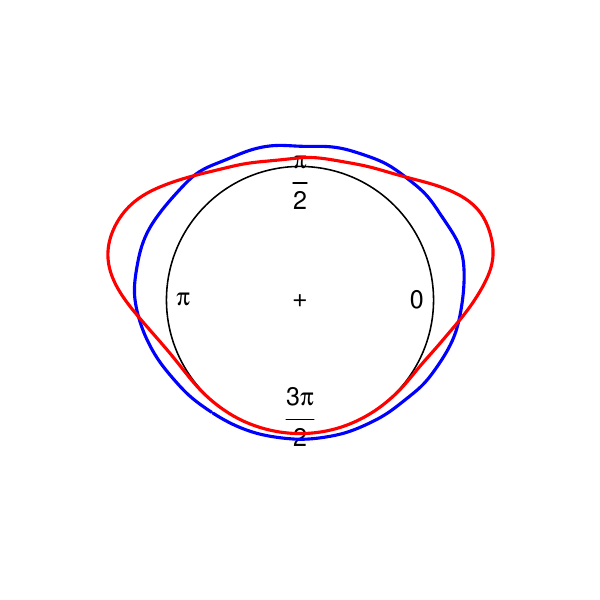}}$} & \multirow{4}{*}{$3$} & 25 & 7.7 & $\mathbf{55.6}$ & 14.6 & 12.6 & 27.3 & 14.3 & 7.8 & 18.0 & 8.2 & 14.2 & 36.2 & 9.8 & 26.7 & 16.0 \\ 
& & 50 & 9.2 & $\mathbf{91.7}$ & 33.5 & 26.2 & 66.0 & 32.8 & 9.3 & 44.3 & 10.5 & 31.6 & 80.2 & 15.5 & 66.3 & 60.0 \\ 
& & 100 & 12.4 & $\mathbf{99.9}$ & 80.8 & 67.8 & 98.1 & 79.4 & 12.6 & 91.5 & 15.7 & 77.3 & 99.7 & 32.4 & 98.6 & 98.3 \\ 
& & 200 & 19.9 & $\mathbf{100.0}$ &$\mathbf{100.0}$& 99.5 & $\mathbf{100.0}$ & $\mathbf{99.9}$ & 20.9 & $\mathbf{100.0}$ & 29.6 & $\mathbf{99.9}$ & $\mathbf{100.0}$ & 76.9 & $\mathbf{100.0}$ & $\mathbf{100.0}$ \\ 
\cline{2-17}
& \multirow{4}{*}{$5$} & 25 & 6.7 & $\mathbf{53.6}$ & 13.0 & 11.3 & 23.4 & 13.0 & 6.7 & 15.8 & 7.1 & 12.6 & 30.7 & 8.8 & 22.9 & 17.2 \\ 
& & 50 & 8.1 & $\mathbf{92.8}$ & 28.6 & 22.4 & 60.8 & 28.6 & 8.2 & 38.2 & 9.0 & 27.2 & 74.3 & 13.6 & 59.5 & 53.2 \\ 
& & 100 & 10.8 & $\mathbf{100.0}$ & 73.3 & 58.6 & 97.8 & 73.0 & 10.9 & 86.9 & 14.2 & 70.7 & 99.5 & 28.1 & 97.9 & 95.7 \\ 
& & 200 & 17.4 & $\mathbf{100.0}$ & $\mathbf{99.9}$ & 99.0 & $\mathbf{100.0}$ & 99.8 & 17.7 & $\mathbf{100.0}$ & 26.4 & 99.8 & $\mathbf{100.0}$ & 70.8 & $\mathbf{100.0}$ & $\mathbf{100.0}$ \\ 
\cline{2-17}
& \multirow{4}{*}{$10$} & 25 & 6.3 & $\mathbf{43.8}$ & 11.1 & 9.9 & 18.1 & 10.9 & 6.5 & 12.9 & 6.7 & 10.7 & 23.5 & 7.8 & 17.7 & 15.4 \\ 
& & 50 & 7.0 & $\mathbf{86.8}$ & 21.7 & 17.3 & 45.9 & 21.7 & 7.0 & 27.9 & 7.9 & 20.5 & 58.3 & 11.3 & 44.5 & 40.0 \\ 
& & 100 & 8.5 & $\mathbf{99.9}$ & 54.4 & 42.1 & 91.1 & 54.8 & 8.9 & 70.0 & 10.8 & 52.3 & 96.7 & 20.9 & 90.5 & 85.2 \\ 
& & 200 & 14.5 & $\mathbf{100.0}$ & 98.3 & 93.4 & $\mathbf{100.0}$ & 98.3 & 14.6 & $\mathbf{99.9}$ & 20.6 & 97.7 & $\mathbf{100.0}$ & 53.0 & $\mathbf{100.0}$ & $\mathbf{100.0}$ \\ 
\midrule
\multirow{12}{*}{$\substack{{\rm MvM (3)}\\\\\centering\includegraphics[width=1.75cm, clip=true, trim={1.6cm 1.5cm 1.75cm 0.5cm}]{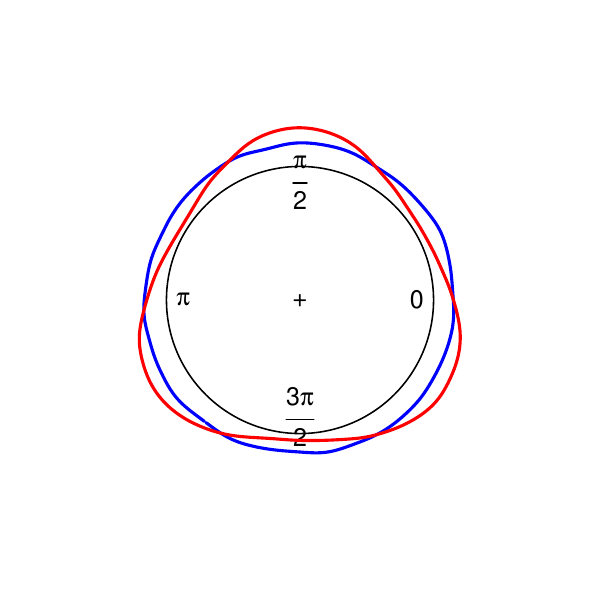}}$} & \multirow{4}{*}{$3$} & 25 & 5.8 & 6.4 & 6.6 & 6.4 & 6.7 & 5.8 & 6.5 & 7.1 & 5.8 & 6.1 & $\mathbf{11.4}$ & 5.9 & 7.8 & 5.1 \\ 
& & 50 & 4.9 & 7.4 & 7.5 & 6.9 & 7.9 & 5.4 & 7.2 & 9.2 & 5.0 & 5.9 & $\mathbf{27.5}$ & 5.3 & 11.9 & 20.5 \\ 
& & 100 & 5.4 & 7.5 & 11.7 & 9.6 & 11.9 & 5.7 & 10.8 & 16.9 & 5.5 & 7.2 & $\mathbf{64.6}$ & 5.8 & 27.7 & 56.3 \\ 
& & 200 & 5.1 & 7.5 & 24.7 & 17.6 & 23.9 & 5.5 & 23.6 & 42.8 & 5.2 & 9.2 & $\mathbf{97.0}$ & 6.0 & 70.2 & 95.2 \\ 
\cline{2-17}
& \multirow{4}{*}{$5$} & 25 & 5.2 & 5.7 & 6.2 & 5.9 & 6.2 & 5.4 & 6.0 & 6.8 & 5.2 & 5.6 & $\mathbf{11.6}$ & 5.3 & 7.4 & 7.5 \\ 
& & 50 & 5.1 & 6.0 & 7.6 & 6.8 & 7.3 & 5.2 & 7.6 & 9.2 & 5.1 & 5.8 & $\mathbf{24.2}$ & 5.3 & 11.9 & 19.8 \\ 
& & 100 & 5.0 & 6.3 & 10.9 & 9.0 & 10.8 & 5.3 & 10.4 & 15.4 & 5.0 & 6.7 & $\mathbf{58.3}$ & 5.5 & 24.0 & 50.5 \\ 
& & 200 & 5.4 & 6.3 & 22.7 & 16.5 & 21.5 & 5.9 & 21.8 & 36.9 & 5.5 & 9.1 & $\mathbf{95.7}$ & 6.3 & 62.4 & 92.0 \\ 
\cline{2-17}
& \multirow{4}{*}{$10$} & 25 & 4.8 & 5.5 & 5.7 & 5.5 & 6.0 & 5.1 & 5.7 & 6.3 & 4.9 & 5.2 & $\mathbf{9.8}$ & 5.1 & 7.1 & 8.1 \\ 
& & 50 & 5.1 & 5.3 & 7.0 & 6.4 & 7.2 & 5.2 & 6.8 & 8.0 & 5.1 & 5.8 & $\mathbf{18.6}$ & 5.3 & 9.8 & 15.9 \\ 
& & 100 & 4.8 & 5.7 & 9.8 & 8.4 & 9.5 & 5.0 & 9.7 & 12.8 & 4.9 & 6.3 & $\mathbf{44.2}$ & 5.2 & 18.6 & 36.5 \\ 
& & 200 & 5.1 & 5.3 & 17.0 & 13.2 & 16.0 & 5.3 & 17.3 & 26.4 & 5.1 & 8.1 & $\mathbf{88.1}$ & 5.7 & 44.5 & 80.8 \\ 
\midrule
\multirow{12}{*}{$\substack{{\rm MvM (4)}\\\\\centering\includegraphics[width=1.75cm, clip=true, trim={1.6cm 1.5cm 1.75cm 0.5cm}]{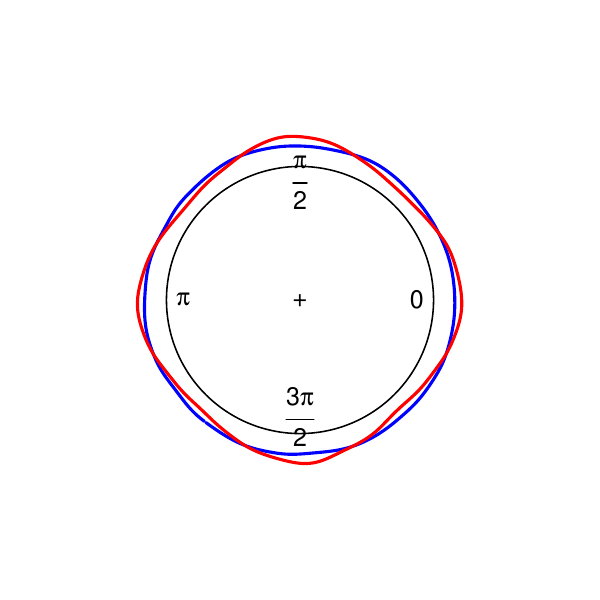}}$} & \multirow{4}{*}{$3$} & 25 & 5.0 & 4.9 & 5.1 & 5.1 & 4.8 & 4.9 & 5.1 & 5.2 & 5.1 & 5.0 & 5.0 & 5.1 & 4.8 & 2.0 \\ 
& & 50 & 4.7 & 5.1 & 5.0 & 4.9 & 5.0 & 4.9 & 4.8 & 5.3 & 4.7 & 4.7 & $\mathbf{6.8}$ & 4.7 & 5.3 & 5.2 \\ 
& & 100 & 5.1 & 5.3 & 5.6 & 5.4 & 5.2 & 5.6 & 5.2 & 6.3 & 5.0 & 5.0 & $\mathbf{11.1}$ & 5.0 & 6.6 & $\mathbf{11.4}$ \\ 
& & 200 & 5.3 & 4.7 & 6.6 & 6.4 & 5.2 & 6.5 & 5.5 & 7.5 & 5.3 & 5.4 & 20.3 & 5.3 & 8.1 & $\mathbf{24.3}$ \\ 
\cline{2-17}
& \multirow{4}{*}{$5$} & 25 & 4.9 & 4.8 & 5.2 & 5.1 & 5.1 & 5.1 & 5.1 & $\mathbf{5.5}$ & 4.9 & 5.0 & $\mathbf{5.6}$ & 5.0 & 5.2 & 4.0 \\ 
& & 50 & 5.2 & 5.2 & 5.6 & 5.5 & 5.3 & 5.7 & 5.2 & 5.8 & 5.1 & 5.4 & $\mathbf{7.0}$ & 5.2 & 5.9 & 6.4 \\ 
& & 100 & 5.2 & 5.1 & 5.7 & 5.6 & 5.2 & 5.8 & 5.2 & 6.2 & 5.2 & 5.3 & 9.8 & 5.2 & 6.2 & $\mathbf{10.3}$ \\ 
& & 200 & 5.0 & 5.0 & 6.3 & 5.9 & 5.2 & 6.3 & 5.2 & 7.2 & 5.0 & 5.2 & 17.2 & 5.1 & 7.8 & $\mathbf{20.4}$ \\ 
\cline{2-17}
& \multirow{4}{*}{$10$} & 25 & 5.0 & 4.9 & 5.3 & 5.2 & 5.2 & 5.3 & 5.1 & 5.4 & 5.1 & 5.1 & $\mathbf{5.6}$ & 5.0 & 5.3 & 5.0 \\ 
& & 50 & 5.0 & 4.9 & 5.2 & 5.0 & 5.1 & 5.1 & 5.0 & 5.5 & 5.0 & 4.9 & $\mathbf{6.4}$ & 5.0 & 5.4 & $\mathbf{6.5}$ \\ 
& & 100 & 4.9 & 5.1 & 5.5 & 5.4 & 5.1 & 5.5 & 5.0 & 6.0 & 4.9 & 5.0 & $\mathbf{8.5}$ & 4.9 & 6.0 & $\mathbf{8.9}$ \\ 
& & 200 & 5.5 & 5.2 & 6.5 & 6.3 & 5.4 & 6.6 & 5.6 & 7.4 & 5.5 & 5.6 & 14.3 & 5.5 & 7.7 & $\mathbf{15.4}$ \\ 
\bottomrule
\end{tabular}
}
\caption{\small Same description as Table~\ref{tbl:pow-shift}, but $c-1$ samples are drawn from a distribution indicated by the first column, with location $\mu=\pi/2$ and concentration $\kappa_0=1$, while the remaining sample comes from a distribution with the same location $\mu$ and concentration $\kappa_c=2$ for vM and W, and $\kappa_c=5$ for SC and MvM.}
\label{tbl:pow-conc}
\end{table}

\subsection{Power of the \texorpdfstring{${\rm LSE}$-}{LSE-}aggregation family}
\label{subsec:sim:merging-families}

In Section~\ref{sec:general-agg}, the merging ${\rm LSE}_{\kappa}$ function was introduced as an interpolation family between the average ($\kappa\to0$) and the maximum ($\kappa\to+\infty$) aggregations, controlled by a tuning parameter $\kappa$. To illustrate the behavior of this family as a function of the parameter, we obtain the power under two types of fixed von Mises alternatives, each of them designed to be especially suited for one of the ``extreme'' statistics of the family. 
First, \emph{one-different-location} alternatives are as considered in Section~\ref{subsec:sim:pow}, for which the maximum statistic is expected to detect the departure more effectively. Second, we consider \emph{equispaced-locations} alternatives, where each of the $c$ samples has a different location, equispaced along the circle, that is $\mu_{\ell}=\mu_0 + 2\pi(\ell-1)/c$, $\ell=1,\ldots,c$; hence, the information is shared across all the samples, and the average statistic is expected to collect all the information more effectively.\nowidow[3]

For each alternative scenario with fixed $n\in\{25, 50, 100\}$ and $c\in\{3,5,10\}$, $M=10^4$ Monte Carlo samples are generated. Each sample consists of a total of $N=nc$ observations, with $n$ observations in each subsample. For every Monte Carlo sample, the tests based on $\smash{T_{N,c}^{\phi,{\rm avg}}}$, $\smash{T_{N,c}^{\phi,\max}}$, and $\smash{T^{\phi, {\rm LSE}_\kappa}_{N,c}}$ with $\kappa\in\{10^{k}:k\in\{-1,0,1,2\}\}$ are computed using the $\rm AD$ kernel. The asymptotic critical values at significance level $\alpha=0.05$ are used and they are obtained using $10^4$ Monte Carlo replications of $\{\bZ_{k,r}, 1\leq k\leq K_{\rm tr}, r=1,2\}$ with $K_{\rm tr}=10^3$, thereby approximating the null distribution in Theorem~\ref{thm:asymp-H0-general}.

The empirical rejection proportions for the $\rm LSE$ family are reported in Figure~\ref{fig:emp_prp_agg_families} under the two alternative scenarios. The following conclusions can be drawn:
\begin{enumerate}[label=(\textit{\roman*}),ref=\textit{\roman*}]
    \item Under the one-different-location alternative, the maximum test shows higher power than the average, as the maximum selects the largest samplewise discrepancy.
    \item Under the equispaced-locations alternative, the average performs better than the maximum, as the averaged signal accumulates diffuse evidence whereas the maximum uses only the largest samplewise component.
    \item For a fixed $n$, the maximum statistic is less sensitive to the number of samples $c$. The average, in contrast, varies considerably with $c$, in a different way depending on the scenario: under one-different-location, the power is reduced as $c$ increases, as noted in point \eqref{conc:power-c-onediff} of Section~\ref{subsec:sim:pow}, while under equispaced-locations its power increases with $c$, as more signals are available to detect the departure from homogeneity.
    \item The power of the $\rm LSE$ class transitions between the two extreme behaviors as $\kappa$ varies.
\end{enumerate}

Additional simulations to study the behavior of different aggregations with imbalanced samples are reported in Section~\ref{subsec:sims:imbalanced} of the~SM.

\begin{figure}[!ht]
    \centering
    \begin{subfigure}[b]{0.49\linewidth}
        \includegraphics[width=\linewidth, trim={10cm 0.25cm 10.5cm 1.75cm},clip=true]{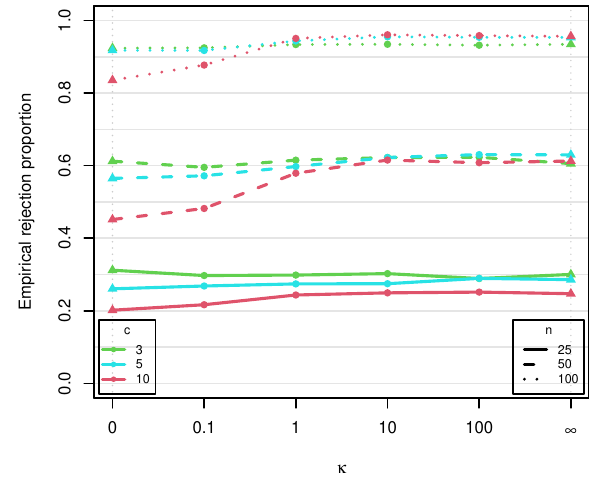}
        \caption{\small One-different-location}
    \end{subfigure}
    \hfill
    \begin{subfigure}[b]{0.49\linewidth}
        \includegraphics[width=\linewidth, trim={20.5cm 0.25cm 0cm 1.75cm},clip=true]{img/rej_prp_equispacedkappa0.25_LSE_PAD.pdf}
        \caption{\small Equispaced-locations}
    \end{subfigure}
    \caption{\small Empirical rejection proportions of $T_{N,c}^{{\rm AD},{\rm LSE}_{\kappa}}$ with $N=nc$ as a function of the parameter $\kappa$ under balanced von Mises location alternatives at significance level $0.05$. One-different-location: $c-1$ samples have location $\mu_0=\pi/2$ and concentration $\kappa_0=5$, while the remaining sample has location $\mu_c=\mu_0 + \pi/10$ and the same concentration. Equispaced-locations: The sample $\ell$ has location $\mu_{\ell}=\mu_0 + 2\pi(\ell-1)/c$, where $\mu_0=\pi/2$, and common concentration $\kappa_0=0.25$. Note that $\kappa=0$ corresponds to $\smash{T_{N,c}^{{\rm AD},{\rm avg}}}$ and $\kappa=\infty$ to $\smash{T_{N,c}^{{\rm AD},{\max}}}$. Results are based on $M=10^4$ replications.}
    \label{fig:emp_prp_agg_families}
\end{figure}

\subsection{Practical issues of the incomplete test}
\label{subsec:sim:incomplete}

As seen in Example~\ref{merging:inc_avg}, the incomplete statistic is a particular member of the weighted sum (${\rm WS}$) merging family that extends the test by \cite{Beran1969a}. This statistic is not symmetric under relabeling of the samples for $c>2$, thus posing a practical disadvantage relative to $T_{N,c}^{\phi,{\rm avg}}$, since it requires adopting a rule to decide which samplewise statistic to omit from the aggregation. The incomplete test remains consistent (see Corollary~\ref{cor:consistency-Halt-general}) after this omission, but its finite-sample power may differ from that of $T_{N,c}^{\phi,{\rm avg}}$. 

To empirically study the effect of the omission rule, we compare the finite-sample power of the incomplete test with the (complete) average test under von Mises one-different-location balanced alternatives, as explained in Section~\ref{subsec:sim:pow}. For fixed $n\in\{50, 200\}$ and $c\in\{3,5,10\}$, $M=10^4$ Monte Carlo samples are generated. Each sample consists of a total of $N=nc$ observations with $n$ observations in each subsample. For every Monte Carlo sample, the tests based on $\smash{T^{\phi,{\rm avg}}_{N,c}}$ and $\smash{T^{\phi,{\rm inc}}_{N,c}}$ are computed using the corresponding asymptotic critical value at significance level $\alpha=0.05$. For the average test, the critical value is computed as in Section~\ref{subsec:sim:asymp0}, while for the incomplete test, it is obtained using \cite{Imhof1961} to approximate its asymptotic distribution, which coincides with the distribution of $\sum_{k=1}^{\infty}2^{-1}b_k(\phi)[(1-\delta_{c2})U_k+\pi_{\rm omit} V_k]$ where $(V_k)_{k\geq 1}$ and $(U_k)_{k\geq 1}$ are independent sequences of mutually independent $\chi^2_2$ and $\chi^2_{2(c-2)}$ random variables, respectively, and where $\pi_{\rm omit}$ denotes the asymptotic proportion of the omitted sample. This null limit follows from Theorem~\ref{thm:asymp-H0-general}.

A natural choice is to omit one of the samples at random. Nevertheless, this may result in a loss of power. In particular, if the omitted sample is the one that differs from the others, the information carried by its associated uniformity statistic may be lost, even though the retained samplewise statistics still contain enough signal, through the pooled uniform scores, for consistent rejection. To quantify this effect, we evaluate the empirical power of the incomplete test under random omission and two oracle benchmark rules: (\emph{rnd}) a sample is omitted at random; (\emph{bst}) best-case scenario: omit one of the $c-1$ samples generated from the common distribution; and (\emph{wst}) worst-case scenario: omit the last sample, which is the only one that differs.

Table~\ref{tbl:pow-shift-inc-comparison} reports the relative difference in power between the incomplete and the average test. When the sample is omitted at random, the incomplete test exhibits a power reduction of approximately $1$--$20\%$ relative to the average test. This is consistent with the $1/c$ probability that random omission removes the shifted sample's component. In the worst-case scenario, this loss is substantially higher, reaching $60$--$90\%$ in some cases. In the best-case scenario, the power is comparable, and even slightly higher than the average test, since the rule omitted a less informative component.

\begin{table}[!htbp]
    \centering
    \scalebox{0.8}{
    \begin{tabular}{cc|rrr|rrr|rrr|rrr|rrr}
    \toprule
    \multirow{2}{*}{$c$} & \multirow{2}{*}{$n$} & \multicolumn{3}{c|}{R} & \multicolumn{3}{c|}{${\rm M}_2$} & \multicolumn{3}{c|}{W} & \multicolumn{3}{c|}{AD} & \multicolumn{3}{c}{${\rm SM}_{10}$}\\
    & & rnd & bst & wst & rnd & bst & wst & rnd & bst & wst & rnd & bst & wst & rnd & bst & wst \\
    \midrule
    \multirow{2}{*}{3} & 50 & $-$16.8 & 4.6 & $-$60.5 & $-$13.6 & 9.5 & $-$57.0 & $-$16.9 & 4.3 & $-$61.0 & $-$17.5 & 4.2 & $-$61.0 & $-$16.8 & 6.7 & $-$63.6 \\ 
    & 200 & $-$2.4 & 0.0 & $-$7.3 & $-$18.3 & 2.9 & $-$57.0 & $-$1.5 & 0.0 & $-$4.5 & $-$1.3 & 0.0 & $-$4.0 & $-$1.9 & 0.0 & $-$6.2 \\ 
    \cline{1-17}
    \multirow{2}{*}{5} & 50 & $-$11.2 & 6.2 & $-$81.3 & $-$3.3 & 5.9 & $-$59.9 & $-$11.6 & 5.5 & $-$81.8 & $-$11.4 & 6.0 & $-$81.7 & $-$10.4 & 6.2 & $-$81.4 \\ 
    & 200 & $-$10.4 & 0.2 & $-$54.3 & $-$13.2 & 4.2 & $-$80.8 & $-$9.3 & 0.0 & $-$48.0 & $-$9.1 & 0.0 & $-$46.1 & $-$10.2 & 0.1 & $-$52.0 \\
    \cline{1-17}
    \multirow{2}{*}{10} & 50 & $-$4.8 & 5.2 & $-$85.2 & $-$3.7 & 1.8 & $-$60.1 & $-$4.8 & 4.5 & $-$86.4 & $-$4.9 & 4.3 & $-$86.6 & $-$4.6 & 5.3 & $-$85.2 \\ 
    & 200 & $-$8.6 & 0.2 & $-$87.0 & $-$4.2 & 4.1 & $-$87.7 & $-$8.5 & 0.0 & $-$85.6 & $-$8.4 & 0.1 & $-$85.3 & $-$8.3 & 0.1 & $-$86.2 \\ 
    \bottomrule
    \end{tabular}
    }
    \caption{\small Relative difference in power ($\%$) of incomplete vs. (complete) average tests, computed as $100(\widehat{\text{power}}^{{\rm inc}}-\widehat{\text{power}}^{\rm avg})/\widehat{\text{power}}^{\rm avg}$, under the location-shift von Mises alternatives described in Table~\ref{tbl:pow-shift}. The sample omitted in the incomplete statistic is chosen according to one of the omission rules indicated in the column headers.}
    \label{tbl:pow-shift-inc-comparison}
\end{table}

Although the incomplete test is theoretically valid, the average test is invariant under sample relabeling and avoids the sensitivity of finite-sample power to an arbitrary omission rule, and is therefore preferable in practice.

\section{Application: Polar bears nursing bout initiation}
\label{sec:app_bear}

Just after leaving the maternity den at 3--4 months of age, polar bear cubs remain dependent on their mothers until they reach 2.5 years of age. During this period, their survival depends strongly on the mother's hunting success. Maternal behavior thus plays a crucial role in population recruitment dynamics. To understand these processes, a team of researchers carried out a long-term observational study of wild polar bears on Devon Island in the Canadian Arctic between 1973 and 1999 \citep{Stirling2022}.

The data collected during this period were subsequently analyzed in detail by \cite{Stirling2024}. Among other biological questions, the authors investigated whether the timing of nursing bouts was uniformly distributed throughout the day, a question also addressed by \cite{Fernandez-de-Marcos2023b}. In the latter study, this question was further examined separately for spring and summer. By contrast, \cite{Stirling2024} did not pursue a seasonal analysis, since they had previously applied the two-sample \cite{Wheeler1964} test and found no significant differences in the distribution of nursing onset times between seasons. Another difference between the two analyses is that \cite{Stirling2024} carried out the analysis using solar times, after applying a suitable transformation to account for solar light, whereas \cite{Fernandez-de-Marcos2023b} analyzed the original clock times. Here, we revisit the seasonal comparison to illustrate the use of the newly proposed tests, and we further investigate potential differences in nursing patterns across cub-age categories. We perform the analysis using the original clock times, since there is continuous daylight on Devon Island during the observation period (April~20--August~8).

The dataset consists of the date and time at which a female polar bear was observed initiating a nursing bout, yielding 151 observations in summer and 69 in spring. The age of the litter is classified into three categories: cubs-of-the-year ($n_1=128$), yearlings ($n_2 = 80$), and two-year-olds ($n_3 = 12$). Since the data were collected manually, the recorded times have a precision of one minute. To remove this data collection artifact, each observation was perturbed by an independent uniform random jitter over its corresponding centered one-minute interval, thus accounting for the discretization. Figures~\ref{fig:bears-season} and~\ref{fig:bears-group} show the observations by season and by cub-age category, respectively. The circular histograms suggest that, in spring, nursing onset times are concentrated around certain time windows. Visual inspection reveals apparent differences between the seasonal patterns. With respect to cub age, the two-year-old group appears to depart visually from uniformity, and its pattern seems to differ from those of the other two age categories. Nonetheless, the limited number of observations in this group may make this impression misleading.

\begin{figure}[!ht]
    \centering
    \includegraphics[clip=TRUE, trim={2.75cm 2.5cm 2cm 1.3cm}, width=0.53\linewidth]{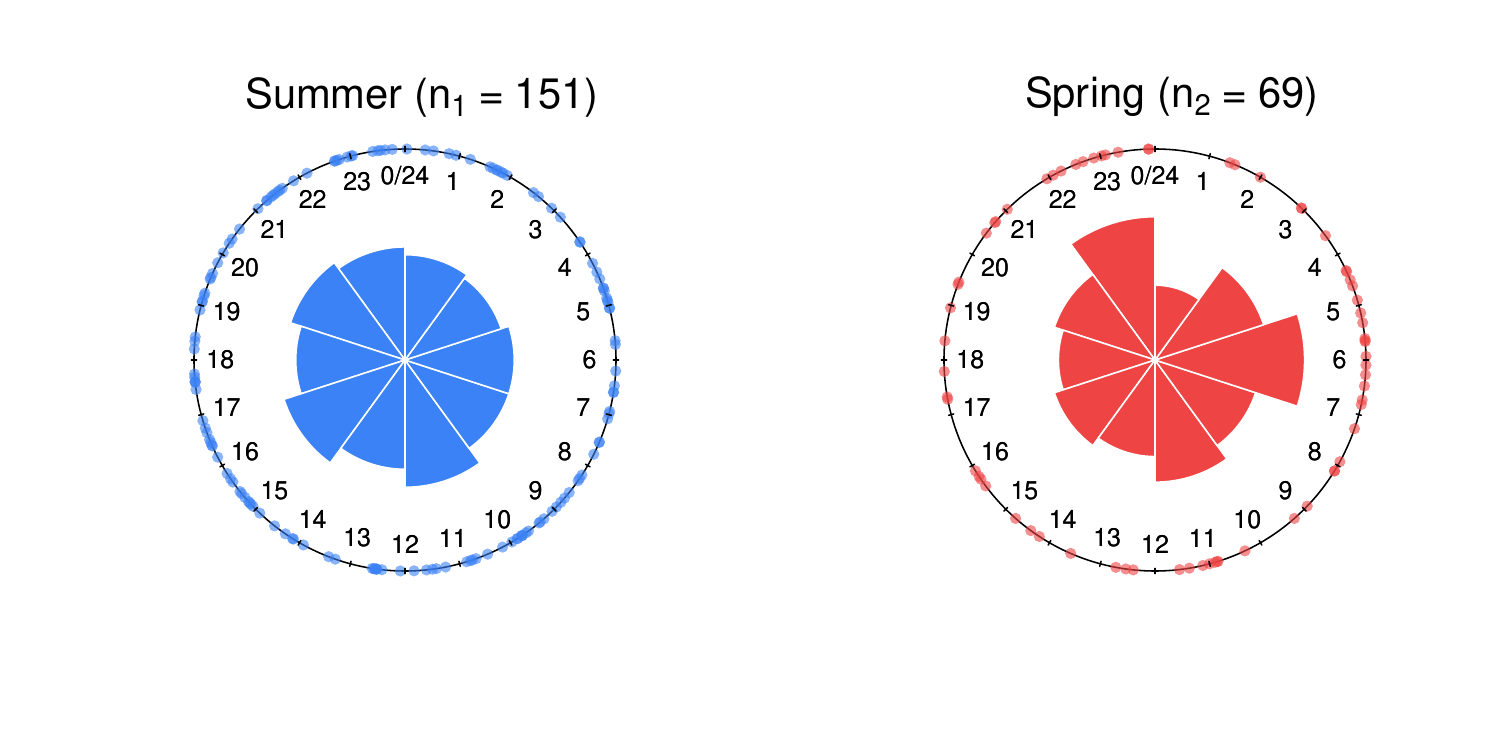}
    \caption{\small Nursing bout onset times (with added random jitter) recorded for wild polar bears in the Canadian Arctic during summer (left) and spring (right) from 1973--1999.}
    \label{fig:bears-season}
\end{figure}
\begin{figure}[!ht]
    \centering
    \includegraphics[clip=TRUE, trim={1.5cm 1.75cm 1.cm 0.75cm}, width=0.855\linewidth]{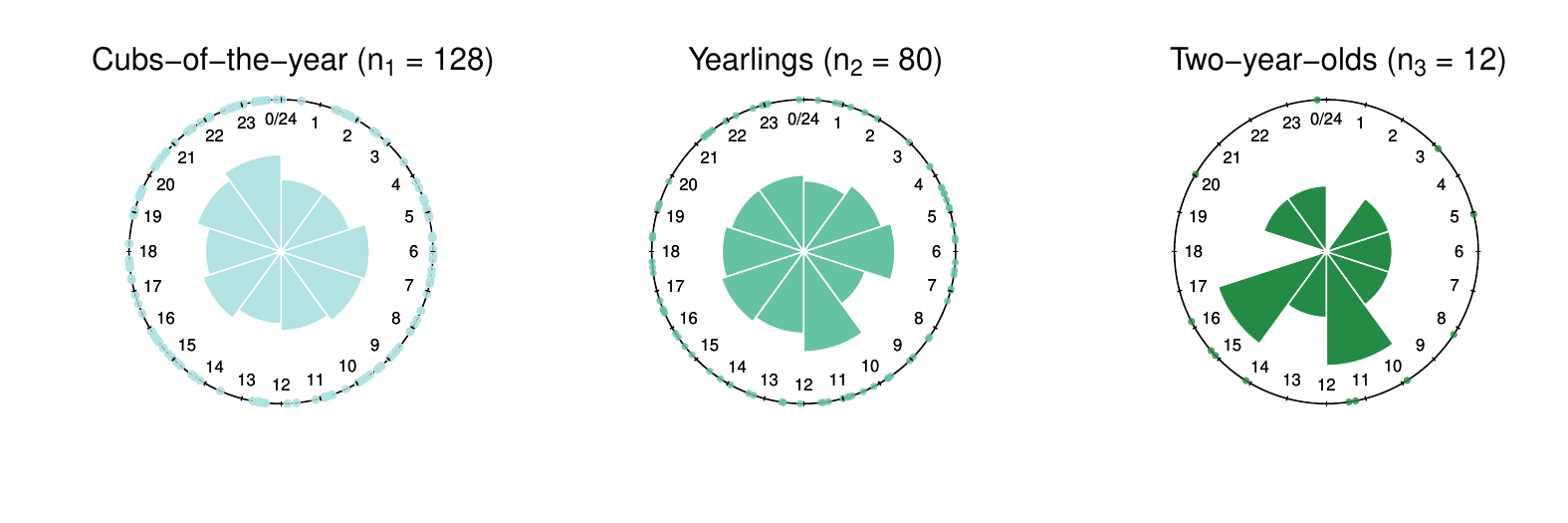}
    \caption{\small Same as Figure~\ref{fig:bears-season}, but stratified by cub age: cubs-of-the-year, yearlings, and two-year-olds.}
    \label{fig:bears-group}
\end{figure}

The first block of Table~\ref{tbl:bears-hom} reports the asymptotic and permutation-based ($M=10^4$) $p$-values for assessing the homogeneity of spring and summer nursing patterns, using the kernels considered in Section~\ref{sec:sim}. Since, when $c=2$, the aggregation rules considered yield equivalent tests (Corollary~\ref{cor:agg-equiv-c2}), only the average test is shown for the seasonal comparison. None of the tests reject the null hypothesis of equal distributions at any usual significance level. The second block presents the corresponding $p$-values for testing the homogeneity of the three cub-age categories, using both the average, $\smash{T_{N,c}^{\phi,{\rm avg}}}$, and maximum, $\smash{T_{N,c}^{\phi,\max}}$, tests. Again, the null hypothesis of equal distributions is not rejected at any usual significance level. Note that the permutation-based and asymptotic $p$-values are generally consistent, even though the two-year-old group has only $n_3=12$ observations. We also performed a two-sample test between cubs-of-the-year ($n_1=128$) and the combined yearling/two-year-old group ($n_2=92$), but these tests did not reject homogeneity either. Overall, these findings do not provide statistical evidence of differences in nursing timing across seasons or cub-age categories.

\begin{table}[ht!]
    \centering
    \scalebox{0.7}{
    \begin{tabular}{ccc|c|c|c|c|ccc|c|ccc|ccc}
        \toprule
        & \multirow{2}{*}{Test} & \multirow{2}{*}{Calibration} & \multirow{2}{*}{$\rm R$} & \multirow{2}{*}{${\rm M}_2$} & \multirow{2}{*}{$\rm W$} & \multirow{2}{*}{$\rm A$} & \multicolumn{3}{c|}{${\rm R}_{t}$} & \multirow{2}{*}{$\rm AD$} & \multicolumn{3}{c|}{${\rm SM}_{\kappa}$} & \multicolumn{3}{c}{${\rm P}_{\rho}$} \\
        & & & & & & & $1/4$ & $1/3$ & $1/2$ & & $0.1$ & $1$ & $10$ & $0.1$ & $0.5$ & $0.9$\\
        \midrule
        \multirow{2}{*}{Seasonal} & \multirow{2}{*}{Avg} & Asymp & 0.410 & 0.884 & 0.490 & 0.474 & 0.603 & 0.534 & 0.393 & 0.503 & 0.419 & 0.503 & 0.581 & 0.446 & 0.556 & 0.539 \\
        & & Perm & 0.411 & 0.884 & 0.497 & 0.473 & 0.607 & 0.543 & 0.391 & 0.523 & 0.419 & 0.505 & 0.587 & 0.446 & 0.559 & 0.554 \\
        \midrule
        \multirow{4}{*}{Cub-age} & \multirow{2}{*}{Avg} & Asymp & 0.581 & 0.794 & 0.776 & 0.741 & 0.814 & 0.716 & 0.658 & 0.830 & 0.591 & 0.699 & 0.967 & 0.626 & 0.880 & 0.963 \\
        & & Perm & 0.583 & 0.801 & 0.790 & 0.758 & 0.822 & 0.727 & 0.663 & 0.853 & 0.594 & 0.705 & 0.972 & 0.632 & 0.889 & 0.983 \\
        \cline{2-17}
        & \multirow{2}{*}{Max} & Asymp & 0.458 & 0.691 & 0.552 & 0.527 & 0.566 & 0.507 & 0.513 & 0.588 & 0.461 & 0.511 & 0.865 & 0.474 & 0.687 & 0.734 \\
        & & Perm & 0.473 & 0.701 & 0.566 & 0.535 & 0.582 & 0.519 & 0.533 & 0.602 & 0.476 & 0.523 & 0.883 & 0.488 & 0.701 & 0.768 \\
        \bottomrule
    \end{tabular}}
    \caption{\small Asymptotic and permutation-based $p$-values for the $c$-sample Sobolev tests based on $\smash{T_{N,c}^{\phi,{\rm avg}}}$ (Avg) and $\smash{T_{N,c}^{\phi,\max}}$ (Max) with $N=220$ applied to nursing onset times. The kernels used are indicated in the header.}
    \label{tbl:bears-hom}
\end{table}

\section{Discussion}
\label{sec:discussion}

This work introduces a general class of nonparametric homogeneity tests for circular data based on Sobolev statistics and uniform scores. The proposed framework unifies several previously developed tests for the multisample problem and yields new practically useful tests. In particular, the circular Anderson--Darling test provides a natural omnibus improvement over the Watson-type test, while the softmax and Poisson kernels offer targeted sensitivity to multimodal departures from homogeneity.

A central advantage of the class is that the null asymptotic distribution is distribution-free. Therefore, the tests can be calibrated without bootstrap or permutation resampling, which makes it especially convenient. The aggregation framework adds practical flexibility: average and maximum aggregations target diffuse and concentrated departures, respectively, while the $M$- and $\mathrm{LSE}$-families interpolate between these behaviors. A further theoretical contribution is identifying the decay condition on the Fourier coefficients required by the discreteness of the uniform scores, which had previously gone unnoticed.

The numerical experiments show that no single kernel is uniformly best across all alternatives, underscoring the value of the proposed framework as a flexible family of tests. Still, as a practical recommendation, the Anderson--Darling kernel appears to be a strong default omnibus choice, consistently improving on the Watson-type competitor. Regarding aggregation, the average test seems to offer the best compromise between symmetry, tractability, and power. The maximum test can be advantageous when the deviation from homogeneity is dominated by one or a few components, but this comes at the price of Monte Carlo approximation of the asymptotic critical values and a potential loss of efficiency when the signal is spread across samples. However, performance can vary depending on the sample-size proportions and alternative type. The incomplete test is theoretically appealing because of its closer connection with the classical two-sample construction, yet its dependence on the omitted sample and its lack of relabeling invariance for $c>2$ limit its practical attractiveness.

This work provides a unified theoretical and practical framework for circular homogeneity testing. Alternative nonparametric approaches could be based on estimating samplewise densities by kernel smoothing and measuring discrepancies between them, following the framework initiated by \cite{Anderson1994a} for the two-sample Euclidean problem. Another possibility is to use empirical characteristic functions for the samples and measure their distance, as done in the circular goodness-of-fit test in \cite{Jammalamadaka2019a}. However, their main disadvantage is that they would generally not be distribution-free, therefore requiring resampling.

Several future directions naturally follow. First, one could develop homogeneity tests on the sphere of arbitrary dimension. The obvious obstacle is the absence of a natural ordering in higher dimensions, which may be addressed using projection-based or transport-based tests. A second direction is the study of high-dimensional settings, where different aggregation rules could lead to useful tests. Finally, the choice of the tuning parameter in the aggregation families in a data-driven manner remains an interesting open problem.

\section*{Supplementary materials}

Supplementary materials provide the proofs of the results presented in the paper and additional simulations with imbalanced samples.

\section*{Acknowledgments}

The first author is supported by grant PIPF-2023/COM-30702, funded by Comunidad de Madrid. Both authors acknowledge support from PCI2024-155058-2, funded by MICIU/AEI/10.13039/\-501100011033/UE.



\fi

\ifsupplement

\newpage
\title{Supplementary materials for ``A class of nonparametric homogeneity tests on the circle''}
\setlength{\droptitle}{-1cm}
\predate{}%
\postdate{}%
\date{}

\author{Alberto Fern\'andez-de-Marcos$^{1,2}$ and Eduardo Garc\'ia-Portugu\'es$^{1}$}
\footnotetext[1]{Department of Statistics, Universidad Carlos III de Madrid (Spain).}
\footnotetext[2]{Corresponding author. e-mail: \href{mailto:albertfe@est-econ.uc3m.es}{albertfe@est-econ.uc3m.es}.}
\maketitle

\begin{abstract}
	These supplementary materials are divided into three parts. Section~\ref{sec:proofs} contains the proofs of the main results, Section~\ref{sec:lemma} contains some technical lemmas, and Section~\ref{sec:sims} includes additional numerical experiments.
\end{abstract}
\begin{flushleft}
	\small\textbf{Keywords:} Circular data; Homogeneity tests; Multisample tests; Sobolev tests; Uniform scores.
\end{flushleft}

\appendix

\section{Proofs of the main results}\label{sec:proofs}

\subsection{Proofs of Section~\ref{sec:unif-scores}}\label{subsec:proofs:unif-scores}

\begin{proof}[Proof of Proposition~\ref{prp:unif-scores-symm}]
    To prove \eqref{cond:symm-1}, we use~\eqref{eq:sob-def} and the expansion \eqref{eq:sob-exp}, which yield
    \begin{align}
        n_1 S_{n_1}^{\phi}(\bc^{(1)})-n_2 S_{n_2}^{\phi}(\bc^{(2)}) =&\; \sum_{i,j=1}^{n_1}\phi(\cos(c_i^{(1)} - c_j^{(1)})) - \sum_{i,j=1}^{n_2}\phi(\cos(c_i^{(2)} - c_j^{(2)}))\nonumber\\
        =&\;\sum_{k=1}^{\infty}b_k(\phi)\bigg(\Big(\sum_{i=1}^{n_1}\cos(k c_i^{(1)})\Big)^2 - \Big(\sum_{i=1}^{n_2}\cos(k c_i^{(2)})\Big)^2 \nonumber\\
        &\qquad\qquad\quad+ \Big(\sum_{i=1}^{n_1}\sin(k c_i^{(1)})\Big)^2 - \Big(\sum_{i=1}^{n_2}\sin(k c_i^{(2)})\Big)^2\bigg).\label{eq:sym-1}
    \end{align}

    Let $k\geq 1$, $\bc=(c_1,\ldots,c_N)$ with $c_i=2\pi r_i/N$ being the uniform score associated to the $i$th observation, $i=1,\ldots,N$, and $c_{(i)}$ be the $i$th order statistic of $\bc$. We show that $\sum_{i=1}^{N}\sin(k c_i) = 0$, so that the sine terms cancel each other.
    For even $N$, we have
    \begin{align*}
        \sum_{i=1}^{N}\sin(k c_i)&=\sum_{i=1}^{N/2-1}\sin(k c_{(i)}) + \sum_{i=1}^{N/2-1}\sin(k c_{(N-i)}) + \sin(k c_{(N/2)}) + \sin(k c_{(N)})\\
        &=\sum_{i=1}^{N/2-1}\Big(\sin(2\pi k i / N) + \sin(- 2\pi k i/N)\Big) + \sin(\pi k) + \sin(2\pi k)= 0,
    \end{align*}
    while for odd $N$, we have
    \begin{align*}
        \sum_{i=1}^{N}\sin(k c_i)&=\sum_{i=1}^{\floor{N/2}}\sin(k c_{(i)}) + \sum_{i=1}^{\floor{N/2}}\sin(k c_{(N-i)}) + \sin(k c_{(N)})\\
        &=\sum_{i=1}^{\floor{N/2}}\Big(\sin(2\pi k i / N) + \sin(- 2\pi k i/N)\Big) + \sin(2\pi k)= 0.
    \end{align*}
    Finally, we show that $\sum_{i=1}^{N}\cos(k c_i) = N$ only when $k\in N\Z$; otherwise, it is zero.
    First, consider $k = k'N$, for $k'\geq1$. Then,
    \begin{align*}
        \sum_{i=1}^{N}\cos(k' N c_i)&=\sum_{i=1}^{N}\cos(2\pi k' i)=N.
    \end{align*}
    Otherwise,
    \begin{align*}
        \sum_{j=1}^{N}\cos(k c_j) = \sum_{j=1}^{N}\cos(2\pi k j/N)=\Re\Big(\sum_{j=1}^{N}e^{i 2\pi k j/N}\Big)=\Re\Big(\dfrac{e^{i 2\pi k / N}(1 - e^{i 2\pi k})}{1 - e^{i 2\pi k/N}}\Big)=0, \quad\text{for }k\notin N\Z.
    \end{align*}
    Therefore, 
    \begin{align*}
        \eqref{eq:sym-1} =&\;\sum_{k=1}^{\infty}b_k(\phi)\bigg(\Big(\sum_{i=1}^{n_1}\cos(k c_i^{(1)})\Big)^2 - \Big(\sum_{i=1}^{n_2}\cos(k c_i^{(2)})\Big)^2\bigg)\\
        =&\;\sum_{k=1}^{\infty}b_{kN}(\phi)\bigg(\Big(\sum_{i=1}^{n_1}\cos(kN c_i^{(1)})\Big)^2 - \Big(\sum_{i=1}^{n_2}\cos(kN c_i^{(2)})\Big)^2\bigg)\\
        =&\;N(n_1-n_2)\sum_{k=1}^{\infty}b_{kN}(\phi),
    \end{align*}
    and the result is proven.

    Finally, \eqref{cond:symm-2} follows immediately from~\eqref{eq:unif-sc-k}, after using~\eqref{cond:symm-1} and~\eqref{eq:beran},
    \begin{align*}
        T_{N,2}^{\phi}&=S^{\phi}_{n_1}(\bc^{(1)})+S^{\phi}_{n_2}(\bc^{(2)})\\
        &=\left(\frac{n_1}{n_2} + 1\right) S_{n_1}^{\phi}(\bc^{(1)})-\left(\frac{n_1}{n_2}-1\right)N \sum_{k=1}^{\infty}b_{kN}(\phi)\\
        &=\frac{N}{n_1n_2} B_{n_1,n_2}^{\psi}-\left(\frac{n_1}{n_2}-1\right)N \sum_{k=1}^{\infty}b_{kN}(\phi).
    \end{align*} 
\end{proof}

\subsection{Proofs of Section~\ref{sec:specific}}\label{subsec:proofs:specific}

\begin{proof}[Proof of Proposition~\ref{prp:Maag-unifsc}]
    The computational formula for the Watson uniformity test statistic, see, e.g., expression (6.3.33) in \cite{Mardia1999a}, yields
    \begin{align}\label{eq:W-sc-c}
        T_{N,c}^{{\rm W}}=\sum_{\ell=1}^{c}\left[\sum_{i=1}^{n_\ell}\bigg(\dfrac{r_{(i)}^{(\ell)}}{N}-\dfrac{2i-1}{2n_\ell}-\dfrac{\bar{r}^{(\ell)}}{N}+\dfrac{1}{2}\bigg)^2+\dfrac{1}{12n_\ell}\right],
    \end{align}
    where $\bar{r}^{(\ell)}:=n_\ell^{-1}\sum_{i=1}^{n_\ell}r_{(i)}^{(\ell)}$, and $r_{(i)}^{(\ell)}$ corresponds to the $i$th ordered pooled rank of the $\ell$th sample.
    
    We derive a computational formula based on ranks for $U^2_{N,c}$. Let $\Theta_{(1)}<\cdots<\Theta_{(N)}$ denote the pooled order statistics, with ties occurring with probability zero. Note that
    \begin{align}
        U^2_{N,c}&=\sum_{\ell=1}^{c} \frac{n_{\ell}}{N}\sum_{i=1}^N \left[\widehat{F}_{\ell,n_{\ell}}(\Theta_{(i)})-\widehat{H}_N(\Theta_{(i)})-\frac{1}{N}\sum_{j=1}^N (\widehat{F}_{\ell,n_{\ell}}(\Theta_{(j)})-\widehat{H}_N(\Theta_{(j)}))\right]^2\nonumber\\
        &=\sum_{\ell=1}^{c} \frac{n_{\ell}}{N}\left\{\sum_{i=1}^N (\widehat{F}_{\ell,n_{\ell}}(\Theta_{(i)})-\widehat{H}_N(\Theta_{(i)}))^2-\frac{1}{N}\left[\sum_{j=1}^N (\widehat{F}_{\ell,n_{\ell}}(\Theta_{(j)})-\widehat{H}_N(\Theta_{(j)}))\right]^2\right\},\label{eq:Umaag1}
    \end{align}
    as specified in~\cite{Maag1966|SM}.
    
    Fix $\ell=1,\ldots, c$, and define for any $j=1,\ldots,N$, $a_{\ell, j}:=\#\{\Theta_{i}^{(\ell)}\leq \Theta_{(j)}:1\leq i\leq n_{\ell}\}$. Then,  we have $\widehat{F}_{\ell,n_{\ell}}(\Theta_{(j)})=a_{\ell,j}/n_{\ell}$, and $\widehat{H}_N(\Theta_{(j)})=~j/N$.
    Also note that 
    $$
    a_{\ell,j}=\begin{cases}
        0,&\text{ if } j< r_{(1)}^{(\ell)},\\
        i,&\text{ if } r_{(i)}^{(\ell)}\leq j< r_{(i+1)}^{(\ell)},\\
        n_\ell,&\text{ if } j\geq r_{(n_\ell)}^{(\ell)}.
    \end{cases}
    $$
    Let $r_{(0)}^{(\ell)}:=1$ and $r_{(n_\ell+1)}^{(\ell)}:=N+1$. Then,
    \begin{align}
        \sum_{j=1}^N (\widehat{F}_{\ell,n_{\ell}}(\Theta_{(j)})-\widehat{H}_N(\Theta_{(j)}))&=\sum_{j=1}^N \Big(\frac{a_{\ell,j}}{n_{\ell}}-\frac{j}{N}\Big)\nonumber
        =\sum_{i=0}^{n_{\ell}}\sum_{j=r_{(i)}^{(\ell)}}^{r_{(i+1)}^{(\ell)} - 1} \Big(\frac{a_{\ell,j}}{n_{\ell}}-\frac{j}{N}\Big)\nonumber\\
        &=\frac{N+1}{2}-\frac{1}{n_\ell}\sum_{i=1}^{n_\ell}r_{(i)}^{(\ell)}=\frac{N+1}{2}-\bar{r}^{(\ell)},\label{eq:Umaag2}
    \end{align}
    where the last equality comes from 
    \begin{align*}
        \sum_{j=r_{(i)}^{(\ell)}}^{r_{(i+1)}^{(\ell)} - 1} \Big(\frac{a_{\ell,j}}{n_{\ell}}-\frac{j}{N}\Big)&=\sum_{j=r_{(i)}^{(\ell)}}^{r_{(i+1)}^{(\ell)} - 1} \Big(\frac{i}{n_{\ell}}-\frac{j}{N}\Big)=\frac{i}{n_{\ell}}(r_{(i+1)}^{(\ell)}  - r_{(i)}^{(\ell)})-\frac{1}{2N}(r_{(i+1)}^{(\ell)}(r_{(i+1)}^{(\ell)}-1) - r_{(i)}^{(\ell)}(r_{(i)}^{(\ell)}-1)),
    \end{align*}
    for a fixed $i=0,\ldots, n_\ell$, and from
    \begin{align*}
        \sum_{i=0}^{n_{\ell}}i(r_{(i+1)}^{(\ell)}- r_{(i)}^{(\ell)})&=n_\ell(N+1)-\sum_{i=1}^{n_\ell}r_{(i)}^{(\ell)}.
    \end{align*}
    
    Noting that
    \begin{align*}
        \sum_{j=r_{(i)}^{(\ell)}}^{r_{(i+1)}^{(\ell)} - 1} \Big(\frac{a_{\ell,j}}{n_{\ell}}-\frac{j}{N}\Big)^2&=\sum_{j=r_{(i)}^{(\ell)}}^{r_{(i+1)}^{(\ell)} - 1} \Big(\frac{i}{n_{\ell}}-\frac{j}{N}\Big)^2\\
        &=\frac{1}{n_{\ell}^2}i^2(r_{(i+1)}^{(\ell)} - r_{(i)}^{(\ell)})-\frac{i}{N n_{\ell}}\big(r_{(i+1)}^{(\ell)}(r_{(i+1)}^{(\ell)} -1) - r_{(i)}^{(\ell)}(r_{(i)}^{(\ell)}-1)\big)\\
        &\quad+ \frac{1}{6N^2}\Big(r_{(i+1)}^{(\ell)}(r_{(i+1)}^{(\ell)}-1)(2r_{(i+1)}^{(\ell)}-1) - r_{(i)}^{(\ell)}(r_{(i)}^{(\ell)}-1)(2r_{(i)}^{(\ell)}-1)\Big),
    \end{align*}
    we have
    \begin{align}
        \sum_{i=1}^{N}(\widehat{F}_{\ell,n_{\ell}}(\Theta_{(i)})-\widehat{H}_N(\Theta_{(i)}))^2
        =\;&\frac{(N+1)(2N+1)}{6N}-\frac{1}{n_{\ell}^2}\sum_{i=1}^{n_\ell}(2i-1)r_{(i)}^{(\ell)}\nonumber\\
        &+\frac{1}{N n_{\ell}}\sum_{i=1}^{n_\ell}(r_{(i)}^{(\ell)})^2-\frac{\bar{r}^{(\ell)}}{N},\label{eq:Umaag3}
    \end{align}
    using
    \begin{align*}
        \sum_{i=0}^{n_{\ell}}i^2(r_{(i+1)}^{(\ell)}- r_{(i)}^{(\ell)})&=n_\ell^2(N+1)-\sum_{i=1}^{n_\ell}(2i-1)r_{(i)}^{(\ell)}
    \end{align*}
    and
    \begin{align*}
        \sum_{i=0}^{n_{\ell}}i\Big(r_{(i+1)}^{(\ell)}(r_{(i+1)}^{(\ell)}-1) - r_{(i)}^{(\ell)}(r_{(i)}^{(\ell)}-1)\Big)&=n_\ell N (N+1)-\sum_{i=1}^{n_\ell}r_{(i)}^{(\ell)}(r_{(i)}^{(\ell)}-1).
    \end{align*}
    
    Therefore,~\eqref{eq:Umaag2} and~\eqref{eq:Umaag3} yield
    \begin{align*}
        \eqref{eq:Umaag1}&=\sum_{\ell=1}^{c} \frac{n_{\ell}}{N}\left\{\frac{(N+1)(2N+1)}{6N}-\frac{1}{n_{\ell}^2}\sum_{i=1}^{n_\ell}(2i-1)r_{(i)}^{(\ell)}+\frac{1}{N n_{\ell}}\sum_{i=1}^{n_\ell}(r_{(i)}^{(\ell)})^2-\frac{\bar{r}^{(\ell)}}{N}-\frac{1}{N}\Big(\frac{N+1}{2}-\bar{r}^{(\ell)}\Big)^2\right\}\\
        &=\sum_{\ell=1}^{c} \frac{n_{\ell}}{N}\bigg\{\frac{(N+1)(2N+1)}{6N}-\frac{N(2n_\ell-1)(2n_\ell+1)}{12n_\ell^2}+\frac{N}{n_\ell}\sum_{i=1}^{n_\ell}\bigg(\frac{r_{(i)}^{(\ell)}}{N} -\frac{2i-1}{2n_\ell}\bigg)^2\\
        &\qquad\qquad\quad-\frac{\bar{r}^{(\ell)}}{N}-\frac{1}{N}\Big(\frac{N+1}{2}-\bar{r}^{(\ell)}\Big)^2\bigg\}\\
        &=\sum_{\ell=1}^{c} \frac{n_{\ell}}{N}\left\{\frac{(N+1)(2N+1)}{6N}-\frac{N(2n_\ell-1)(2n_\ell+1)}{12n_\ell^2}-\frac{2N+1}{4N}-\frac{N}{12n_\ell^2}\right\}\\
        &\quad+\sum_{\ell=1}^{c}\left\{\sum_{i=1}^{n_\ell}\Big(\frac{r_{(i)}^{(\ell)}}{N} - \frac{2i-1}{2n_\ell} - \frac{\bar{r}^{(\ell)}}{N} +\frac{1}{2}\Big)^2+\frac{1}{12n_\ell}\right\}\\
        &=\sum_{\ell=1}^{c}\left\{\sum_{i=1}^{n_\ell}\Big(\frac{r_{(i)}^{(\ell)}}{N} - \frac{2i-1}{2n_\ell} - \frac{\bar{r}^{(\ell)}}{N} +\frac{1}{2}\Big)^2+\frac{1}{12n_\ell}\right\}-\dfrac{1}{12N},
    \end{align*}
    and the result follows from \eqref{eq:W-sc-c}.
\end{proof}

\subsection{Proofs of Section~\ref{sec:general-agg}}\label{subsec:proofs:general-agg}

\begin{proof}[Proof of Corollary~\ref{cor:agg-equiv-c2}]
    Write $\rho_{12}=n_1/n_2$ and $\Delta_{n_1,n_2}^{\phi}:=(\rho_{12}-1)N\sum_{k=1}^{\infty} b_{kN}(\phi)$. By Proposition~\ref{prp:unif-scores-symm}\eqref{cond:symm-1}, $S_{n_2}^{\phi}(\bc^{(2)})=\rho_{12}S_{n_1}^{\phi}(\bc^{(1)})-\Delta_{n_1,n_2}^{\phi}$. Thus,
    $
    T_{N,2}^{\phi,m}=g_m\big(S_{n_1}^{\phi}(\bc^{(1)})\big)
    $
    with 
    $g_m(x):=m(x, \rho_{12} x - \Delta_{n_1,n_2}^{\phi})$.
    The function $g_m$ is strictly increasing because $m$ satisfies \eqref{m:str-inc} and because $\rho_{12}>0$.
    Then, $T_{N,2}^{\phi,m}$ is a strictly increasing transformation of the same statistic $S_{n_1}^{\phi}(\bc^{(1)})$, and the result follows.
\end{proof}

\subsection{Proofs of Section~\ref{sec:null-asymp}}\label{subsec:proofs:null-asymp}

\begin{proof}[Proof of Proposition~\ref{prp:gkr-asymp-H0}]
    The samples follow the triangular-array framework of Lemma~\ref{lemma:gkr-clt-moving}, which we apply. For each $N\geq 2c$, let $F_{\ell,N}$ denote the angular cdf of the $\ell$th sample with starting point $\theta=0$, and set $H_N:=\sum_{\ell=1}^{c}\pi_{\ell,N}F_{\ell,N}$. Under $\Hcal_0$, we have $F_{1,N}=\cdots=F_{c,N}=H_N$ for all $N\geq1$.
    
    We begin by expressing $\bG_N$ in terms of the ecdf's,
    $$
    G_{k,r,\ell,N}=\dfrac{1}{n_{\ell}}\sum_{i=1}^{n_\ell}g_{k,r}(c_i^{(\ell)})=\dfrac{1}{n_{\ell}}\sum_{i=1}^{n_\ell}g_{k,r}(2\pi \widehat{H}_N(\Theta_i^{(\ell)}))=\int_{0}^{2\pi}g_{k,r}(2\pi \widehat{H}_N(\theta))\,\rd \widehat{F}_{\ell,n_{\ell}}(\theta).
    $$

    Further, let $\bmu_N:=(\mu_{k,r,\ell,N})$ where 
    \begin{align*}
        \mu_{k,r,\ell,N}:=\int_{0}^{2\pi}g_{k,r}(2\pi H_N(\theta))\,\rd F_{\ell,N}(\theta).
    \end{align*}

    The proof proceeds in three steps. First, we show that $N^{1/2}(\bG_N-\bmu_N)$ is asymptotically normal using Lemma~\ref{lemma:gkr-clt-moving}. Second, we show that $\bmu_N=\mathbf{0}$ under $\Hcal_0$. Third, we identify the limiting covariance matrix under $\Hcal_0$.

    Under $\Hcal_0$, it follows that \eqref{eq:Ilim} becomes
    \begin{align}
        I_{k,k',r,r'}^{(i,j,\ell)}
        =(2\pi)^2\bigg(&\iint_{0<u<v<1} u(1-v) g'_{k,r}(2\pi u)g'_{k',r'}(2\pi v)\,\rd u\,\rd v\nonumber\\
        &+\iint_{0<v<u<1} v(1-u) g'_{k,r}(2\pi u) g'_{k',r'}(2\pi v)\,\rd u\,\rd v\bigg)
        =\delta_{kk'}\delta_{rr'}.\label{eq:Ilim-H0}
    \end{align}
    By Lemma~\ref{lemma:gkr-clt-moving}, there exists a unique symmetric matrix $\bSigma_0$ such that
    \begin{align*}
        N^{1/2}(\bG_N-\bmu_N)\inlaw \mathcal{N}_D(\mathbf{0},\bSigma_0).
    \end{align*}

    We show that $\bmu_N=\mathbf{0}$ for all $N$ under $\Hcal_0$. Noting that $H_N=F_{\ell,N}$ for all $\ell=1,\ldots,c$ and $N\geq 1$ under $\Hcal_0$, we obtain
    \begin{align*}
        \mu_{k,r,\ell,N}=\int_0^1 g_{k,r}(2\pi u)\,\rd u=0,
    \end{align*}
    for all $k=1,\ldots, K$, $r=1,2$, $\ell=1,\ldots,c$, and $N\geq 1$.
    
    We obtain the entries of the covariance matrix under $\Hcal_0$, $\bSigma_0=(\bSigma_{(k,r),(k',r')})=(\sigma_{(k,r,\ell),(k',r',\ell')})$. 
    From the quadratic-form representation established in the proof of Lemma~\ref{lemma:gkr-clt-moving}, the variance $\sigma_\bba^2=\bba^\top\bSigma_0\bba$ is defined for every $\bba\in\R^D$. Under $\Hcal_0$, \eqref{eq:Ilim-H0} gives $A_{ij}^{(\ell)}(J_{\ell_1,\bba},J_{\ell_2,\bba})=\allowbreak\sum_{k=1}^{K}\sum_{r=1}^{2}a_{k,r,\ell_1}a_{k,r,\ell_2}$ for all $i,j,\ell,\ell_1,\ell_2=1,\ldots,c$, and
    $
    Q_0(\bba):=\sigma_\bba^2
    =\sum_{i=1}^{c}\sigma_{ii,\bba}
    +\sum_{\substack{i,j=1\\i\neq j}}^{c}\sigma_{ij,\bba},
    $
    with
    \begin{align*}
    \sigma_{ii,\bba}&=\Big(\dfrac{1}{\pi_i} - 1\Big)\sum_{k=1}^{K}\sum_{r=1}^{2}a_{k,r,i}^2,
    \quad\text{and}\quad
    \sigma_{ij,\bba}=-\sum_{k=1}^{K}\sum_{r=1}^{2}a_{k,r,i}a_{k,r,j},\quad i\neq j.
    \end{align*}

    To obtain the entries of $\bSigma_0$, we use the polarization identity of the quadratic form $Q_0$. For instance, setting $\bba=\be_{k,r,i}$, where $\be_{k,r,i}$ is the vector whose $(k,r,i)$th coordinate equals 1 and all other coordinates are zero, yields
    \begin{align}
        \sigma_{(k,r,i),(k,r,i)}=\sigma_{ii,\bba}=1/\pi_{i} - 1,
        \label{eq:sigma0_diag}
    \end{align}
    and
    \begin{align}
        \sigma_{(k,r,i),(k,r,j)}
        &=\frac12(Q_0(\be_{k,r,i}+\be_{k,r,j})
        -Q_0(\be_{k,r,i})-Q_0(\be_{k,r,j}))
        =-1 \label{eq:sigma0_diag2}
    \end{align}
    for all $i,j=1,\ldots,c$ such that $i\neq j$.
    
    For off-diagonal blocks $\bSigma_{(k,r),(k',r')}$ with $k\neq k'$ or $r\neq r'$, we obtain
    \begin{align}
        \sigma_{(k,r,i),(k',r',j)}
        &=\frac12(Q_0(\be_{k,r,i}+\be_{k',r',j})
        -Q_0(\be_{k,r,i})-Q_0(\be_{k',r',j}))\nonumber\\
        &=\frac{1}{2}(1/\pi_{i} + 1/\pi_{j} - 2-(1/\pi_{i} - 1)-(1/\pi_{j} - 1))
        =0,\label{eq:sigma0_off}
    \end{align}
    for all $i,j=1,\ldots,c$.
    
    Combining~\eqref{eq:sigma0_diag},~\eqref{eq:sigma0_diag2}, and~\eqref{eq:sigma0_off}, it follows that $\bSigma_{(k,r),(k',r')}=\mathbf{0}$ when $k\neq k'$ or $r\neq r'$, and $\bSigma_{(k,r),(k,r)}=(\sigma_{(k,r,i),(k,r,j)})_{i,j=1}^{c}$, which proves the result.
\end{proof}

\begin{proof}[Proof of Theorem~\ref{thm:asymp-H0-general}]
    We denote by $\varphi^{(N)}_K$, $\varphi^{(N)}$, $\varphi_K^{\infty}$, and $\varphi^{\infty}$ the characteristic functions of the $K$-truncated statistic, 
    $$
    T^{\phi,m}_{N,c,K}:=m\bigg(\Big(\sum_{k=1}^{K}\sum_{r=1}^{2} 2^{-1}b_k(\phi) n_{\ell}^{-1}\big(\sum_{i=1}^{n_{\ell}}g_{k,r}(c_i^{(\ell)})\big)^2\Big)_{\ell=1}^{c}\bigg),
    $$
    the statistic $T^{\phi,m}_{N,c}$, the random variable 
    $$
    T^{\phi,m}_{\infty,c,K}:=m\bigg(\Big(\pi_\ell \sum_{k=1}^{K}\sum_{r=1}^{2} 2^{-1}b_k(\phi) Z_{k,r,\ell}^2\Big)_{\ell=1}^{c}\bigg),
    $$
    and $T^{\phi,m}_{\infty,c}$, respectively. Note that $T^{\phi,m}_{\infty,c}$ is well defined. Indeed, by the Monotone Convergence Theorem (MCT), for each $\ell=1,\ldots,c$ we have
    $$
    \E{\sum_{k=1}^{\infty}\sum_{r=1}^{2}
    \pi_\ell 2^{-1}|b_k(\phi)|Z_{k,r,\ell}^2}
    =
    \Big(\pi_\ell\E{Z_{1,1,\ell}^2}\Big)
    \sum_{k=1}^{\infty}|b_k(\phi)|<\infty;
    $$
    hence, $\pi_\ell \sum_{k=1}^{\infty}\sum_{r=1}^{2} 2^{-1}b_k(\phi) Z_{k,r,\ell}^2<+\infty$ almost surely. Since $m$ is continuous, thus measurable,
    $T^{\phi,m}_{\infty,c}$ is a well defined random variable.
    
    We prove that for all $t\in\R$,
    $|\varphi^{(N)}(t)-\varphi^{\infty}(t)|\to 0$,
    as $N\to\infty$. We proceed in two steps: first, we derive the asymptotic distribution of $T_{N,c,K}^{\phi,m}$ for fixed $K\geq 1$, and then we control the truncation contributions.
    
    Fix $K\geq 1$. Expanding in terms of the spherical harmonics, we have
    \begin{align*}
        T_{N, c, K}^{\phi, m}=\,&m\bigg(\Big(\sum_{k=1}^{K}\sum_{r=1}^{2} 2^{-1}b_k(\phi) n_{\ell}G_{k,r,\ell,N}^2\Big)_{\ell=1}^{c}\bigg)\nonumber\\
        =\,&m\bigg(\Big(\sum_{k=1}^{K}\sum_{r=1}^{2} 2^{-1}b_k(\phi) \pi_{\ell,N} (N^{1/2}G_{k,r,\ell,N})^2\Big)_{\ell=1}^{c}\bigg).
    \end{align*}
    Note that $m$ is continuous due to \eqref{m:lip}. Then, by Proposition~\ref{prp:gkr-asymp-H0} and the convergence $\bpi_N\to\bpi$, using Slutsky's theorem and the continuous mapping theorem yields $T_{N,c,K}^{\phi,m}~\inlaw~T_{\infty,c,K}^{\phi,m}$, and, thus, 
    \begin{align}
        |\varphi_K^{(N)}(t)-\varphi_K^{\infty}(t)|\to 0,\label{eq:T_K_asymp_general}
    \end{align}
    for all $K\geq 1$ and all $t\in\R$, as $N\to\infty$. Given the block-diagonal structure of $\bSigma_0$, the vectors $\{\bZ_{k,r}:1\leq k\leq K, r=1,2\}$ are mutually independent.
    
    We now prove that
    \begin{align}
        \sup_{N\geq 2c} |\varphi^{(N)}(t)-\varphi_K^{(N)}(t)| \to 0,\label{eq:TN_asymp_general}
    \end{align}
    for all $t\in\R$, as $K\to\infty$. Using condition \eqref{m:lip} and Lemma~\ref{lemma:exp-G2}, we have for all $N\geq 2c$, all $K\geq 1$, and $t\in\R$,
    \begin{align*}
        |\varphi^{(N)}&(t)-\varphi_K^{(N)}(t)|\leq |t|\Es{|T_{N,c}^{\phi,m}-T_{N,c,K}^{\phi,m}|}{0}\nonumber\\
        &=|t|\Es{\Big|m\bigg(\Big(\sum_{k=1}^{\infty}\sum_{r=1}^{2} 2^{-1}b_k(\phi) n_{\ell}G_{k,r,\ell,N}^2\Big)_{\ell=1}^{c}\bigg)-m\bigg(\Big(\sum_{k=1}^{K}\sum_{r=1}^{2} 2^{-1}b_k(\phi) n_{\ell}G_{k,r,\ell,N}^2\Big)_{\ell=1}^{c}\bigg)\Big|}{0}\nonumber\\
        &\leq L|t|\Es{\Big\lVert\Big(\sum_{k=K+1}^{\infty}\sum_{r=1}^{2} 2^{-1}b_k(\phi) n_{\ell}G_{k,r,\ell,N}^2\Big)_{\ell=1}^{c}\Big\rVert_{\infty}}{0}\nonumber\\
        &\leq L|t|\Es{\sum_{\ell=1}^{c}\Big|\sum_{k=K+1}^{\infty}\sum_{r=1}^{2}2^{-1}b_k(\phi)n_{\ell}G^2_{k,r,\ell,N}\Big|}{0}\nonumber\\
        &\leq L|t|\Es{\sum_{\ell=1}^{c}\sum_{k=K+1}^{\infty}\sum_{r=1}^{2}2^{-1}|b_k(\phi)|n_{\ell}G^2_{k,r,\ell,N}}{0}\nonumber\\
        &= L|t|\sum_{\ell=1}^{c}\sum_{k=K+1}^{\infty}\sum_{r=1}^{2}2^{-1}|b_k(\phi)|\Es{n_{\ell}G^2_{k,r,\ell,N}}{0}\nonumber\\
        &= L|t|\sum_{\ell=1}^{c}\sum_{k=K+1}^{\infty}|b_k(\phi)|\bigg(\Big(1 - \dfrac{n_{\ell}-1}{N-1}\Big)\mathds{1}_{\{k\notin N\Z_{+}\}} + n_{\ell}\mathds{1}_{\{k\in N\Z_{+}\}}\bigg)\nonumber\\
        &= L|t|\Big((c-1)\dfrac{N}{N-1}\sum_{\substack{k>K\\k\notin N\Z_{+}}}|b_k(\phi)| + N\sum_{\substack{k>K\\k\in N\Z_{+}}}|b_k(\phi)|\Big)\nonumber\\
        &\leq L|t|\Big(2(c-1)\sum_{\substack{k>K\\k\notin N\Z_{+}}}|b_k(\phi)| + N\sum_{\substack{k>K\\k\in N\Z_{+}}}|b_k(\phi)|\Big),
    \end{align*}
    where the second equality holds by the MCT.
    Since $|b_k(\phi)|=O(k^{-(1+\delta)})$ for some $\delta>0$, there exists some $K_1\geq 1$ and $C>0$, such that $|b_k(\phi)|\leq Ck^{-(1+\delta)}$ for all $k\geq K_1$. Hence, for all $K\geq K_1$,
    \begin{align*}
        N\sum_{\substack{k>K\\k\in N\Z_{+}}}|b_k(\phi)|&=N\sum_{\substack{k>K\\k=Nk',k'\in \Z_{+}}}|b_k(\phi)|
        \leq N\sum_{k'\geq \lceil K/N\rceil}|b_{Nk'}(\phi)|
        \leq \frac{C}{N^{\delta}}\sum_{k'\geq \lceil K/N\rceil} \frac{1}{(k')^{1+\delta}}\nonumber\\
        &\leq \frac{C}{N^{\delta}}\lrb{\frac{1}{\lceil K/N\rceil^{1+\delta}}+\int_{\lceil K/N\rceil}^\infty \frac{1}{x^{1+\delta}}\,\rd x}
        = \frac{C}{N^{\delta}}\lrb{\frac{1}{\lceil K/N\rceil^{1+\delta}}+\frac{1}{\delta\lceil K/N\rceil^\delta}}\nonumber\\
        &\leq C\Big(1+\frac{1}{\delta}\Big)\lrb{\frac{1}{N^{\delta}}\mathds{1}_{\{N>K\}} + \dfrac{1}{K^\delta}\mathds{1}_{\{N\leq K\}}}
        \leq \Big(1+\frac{1}{\delta}\Big)\dfrac{C}{K^\delta}.
    \end{align*}
    Therefore, as $K\to\infty$ and due to the summability condition~\eqref{eq:summ-cond},
    \begin{align*}
        |\varphi^{(N)}(t)-\varphi_{K}^{(N)}(t)|&\leq L|t|\bigg(2(c-1)\sum_{\substack{k>K\\k\notin N\Z_{+}}}|b_k(\phi)| + \Big(1+\frac{1}{\delta}\Big)\dfrac{C}{K^\delta}\bigg)\\
        &\leq L|t|\bigg(2(c-1)\sum_{k>K}|b_k(\phi)| + \Big(1+\frac{1}{\delta}\Big)\dfrac{C}{K^\delta}\bigg)=o(1),
    \end{align*}
    for all $N\geq 2c$, and~\eqref{eq:TN_asymp_general} follows.
    
    Finally, by condition \eqref{m:lip}, the MCT, and the summability condition~\eqref{eq:summ-cond}, 
    \begin{align*}
        |\varphi_K^{\infty}(t)-\varphi^{\infty}(t)|&\leq|t|\E{|T_{\infty,c,K}^{\phi,m} - T_{\infty,c}^{\phi,m}|}\\
        &=|t|\E{\Big|m\bigg(\Big(\pi_\ell \sum_{k=1}^{K}\sum_{r=1}^{2} 2^{-1}b_k(\phi) Z_{k,r,\ell}^2\Big)_{\ell=1}^{c}\bigg) - m\bigg(\Big(\pi_\ell \sum_{k=1}^{\infty}\sum_{r=1}^{2} 2^{-1}b_k(\phi) Z_{k,r,\ell}^2\Big)_{\ell=1}^{c}\bigg)\Big|}\\
        &\leq L|t|\E{\Big\lVert \Big(\sum_{k=K+1}^{\infty}\sum_{r=1}^{2} \pi_\ell 2^{-1}b_k(\phi) Z_{k,r,\ell}^2\Big)_{\ell=1}^{c}\Big\rVert_{\infty}}\\
        &\leq L|t|\E{\sum_{\ell=1}^{c} \Big|\sum_{k=K+1}^{\infty}\sum_{r=1}^{2} \pi_\ell 2^{-1}b_k(\phi) Z_{k,r,\ell}^2\Big|}\\
        &\leq L|t|\E{\sum_{\ell=1}^{c} \sum_{k=K+1}^{\infty}\sum_{r=1}^{2} \pi_\ell 2^{-1}|b_k(\phi)| Z_{k,r,\ell}^2}\\
        &\leq L|t|\Big(\sum_{\ell=1}^{c}\pi_\ell\E{Z_{1,1,\ell}^2}\Big) \sum_{k=K+1}^{\infty} |b_k(\phi)|\\
        &=L|t|(c-1) \sum_{k=K+1}^{\infty} |b_k(\phi)|
        \to 0\quad\text{ as }K\to\infty.
    \end{align*}

    Let $\epsilon>0$. Choose $K_0\geq 1$ such that
    \begin{align*}
        \sup_{N\geq 2c} |\varphi^{(N)}(t)-\varphi_{K_0}^{(N)}(t)|<\epsilon/3, \quad\text{ and}\quad|\varphi_{K_0}^{\infty}(t)-\varphi^{\infty}(t)|<\epsilon/3.
    \end{align*}
    Then, by the triangle inequality and~\eqref{eq:T_K_asymp_general}, it follows that for each $t\in\R$,
    $$
    |\varphi^{(N)}(t)-\varphi^{\infty}(t)| \leq |\varphi^{(N)}(t)-\varphi_{K_0}^{(N)}(t)| + |\varphi_{K_0}^{(N)}(t)-\varphi_{K_0}^{\infty}(t)| + |\varphi_{K_0}^{\infty}(t)-\varphi^{\infty}(t)|<\epsilon,
    $$
    for all sufficiently large $N$.
\end{proof}

\begin{proof}[Proof of Corollary~\ref{cor:asymp-H0-sum}]
    Considering $m_{\rm sum}(\bx):=\sum_{\ell=1}^{c}x_\ell$, we only need to show that
    \begin{align}\label{eq:weak-limit-eq}
        T_{\infty, c}^{\phi, m_{\rm sum}}\equald T_{\infty, c}^{\phi}.
    \end{align}
    Fix $K\geq 1$. 
    We can rewrite
    $$
    T^{\phi,m_{\rm sum}}_{\infty,c,K}=
    2^{-1} \bZ^{\top}\big(\operatorname{diag}(b_1(\phi),\ldots,b_K(\phi))\otimes\bI_2\otimes\bPi\big)\bZ,
    $$
    where $\bZ\sim \mathcal{N}_{D}(\mathbf{0},\bSigma_0)$ and $\bPi:=\operatorname{diag}(\bpi)$. 
    
    Due to the diagonal block structure, we work with each block $\bZ_{k,r}\sim\mathcal{N}_{c}(\mathbf{0},\bS)$, where $\bS=\diag{1/\bpi}-\mathbf{1}_{c}\mathbf{1}_{c}^{\top}$, so that
    \begin{align*}
        \bZ^{\top}\big(\operatorname{diag}(b_1(\phi),\ldots,b_K(\phi))\otimes\bI_2\otimes\bPi\big)\bZ=\sum_{k=1}^{K}\sum_{r=1}^{2}b_k(\phi)\bZ_{k,r}^{\top}\bPi\bZ_{k,r}.
    \end{align*}
    Since $\bSigma_0=\bI_{2K}\otimes \bS$, the cross-covariance between $\bZ_{k,r}$ and $\bZ_{k',r'}$ is zero whenever $(k,r)\neq (k',r')$, and the vectors $\bZ_{k,r}$ are mutually independent.
    
    For each $(k,r)$ note that $\bZ_{k,r}\equald\bS^{1/2}\bW_{k,r}$ where $\{\bW_{k,r}:1\leq k\leq K,\ r=1,2\}$ are mutually independent and $\bW_{k,r}\sim\mathcal{N}_{c}(\mathbf{0},\bI_c)$. Then,
    $$
    \bZ_{k,r}^{\top}\bPi\bZ_{k,r}\equald\bW_{k,r}^{\top}(\bS^{1/2})^{\top}\bPi\bS^{1/2}\bW_{k,r}.
    $$
    We now study the spectrum of $\bD:=(\bS^{1/2})^{\top}\bPi\bS^{1/2}$. For square matrices $\bX,\bY$, we have that $\bX\bY$ and $\bY\bX$ have the same spectrum, see, e.g., Theorem 1.3.22 in \cite{Horn2013|SM}. Noting that 
    $$
    \bD=(\bS^{1/2})^{\top}(\bPi^{1/2})^{\top}\bPi^{1/2}\bS^{1/2}=(\bPi^{1/2}\bS^{1/2})^{\top}(\bPi^{1/2}\bS^{1/2}),
    $$
    it has the same spectrum as 
    $
        \bC :=\; \bPi^{1/2}\bS(\bPi^{1/2})^{\top}
        =\;\bI_c - \sqrt{\bpi}\sqrt{\bpi}^{\top}.
    $
    Note that 
    $\bC$ is idempotent. Also, $\operatorname{rank}(\bC)=\operatorname{tr}(\bC)=c-1$. Consequently, $\bC$ has the eigenvalue $1$ with multiplicity $c-1$ and the eigenvalue $0$ with multiplicity $1$.
    Since $\bD$ is symmetric, it is orthogonally diagonalizable as
    $
        \bD=\bU^{\top}\bLambda\bU, 
    $
    with
    $
    \bLambda=\diag{\bI_{c-1},0},
    $
    for some orthogonal matrix~$\bU$.

    Let $\bY_{k,r}:=\bU\bW_{k,r}$. Then, $\bY_{k,r}\sim\mathcal{N}_c(\mathbf{0}, \bI_c)$, and
    \begin{align*}
        \bZ_{k,r}^{\top}\bPi\bZ_{k,r}
        \equald\bY_{k,r}^{\top}\bLambda\bY_{k,r}
        =\sum_{j=1}^{c-1}Y_{k,r,j}^2
        \sim \chi^2_{c-1}.
    \end{align*}
    By independence across the pairs $(k,r)$,
    $
        T^{\phi,m_{\rm sum}}_{\infty,c,K}\equald\sum_{k=1}^{K} 2^{-1}b_k(\phi)Y_k=:
        T_{\infty,c,K}^{\phi}
    $
    for all $K\geq 1$.

    Now, since $T_{\infty,c,K}^{\phi,m_{\rm sum}}\to T_{\infty,c}^{\phi,m_{\rm sum}}$ in mean, by the $L^1$ tail argument used in the proof of Theorem~\ref{thm:asymp-H0-general}, we have that it converges also in distribution. Noting that $T_{\infty,c,K}^{\phi}\to T_{\infty,c}^{\phi}$ in mean, and therefore in law, the uniqueness of weak limits proves~\eqref{eq:weak-limit-eq}.
\end{proof}

\begin{proof}[Proof of Corollary~\ref{cor:centered_asymp}]
    We denote by $\varphi^{(N)}_K$, $\varphi^{(N)}$, $\varphi_K^{\infty}$, and $\varphi^{\infty}$ the characteristic functions of the $K$-truncated centered statistic, $\widetilde{T}^{\phi,m}_{N,c,K} :=m\big(\big(\sum_{\substack{1\le k\le K\\ k\notin N\Z_+}}\sum_{r=1}^2 2^{-1}b_k(\phi)n_\ell G_{k,r,\ell,N}^2\big)_{\ell=1}^c\big)$, $\widetilde{T}^{\phi,m}_{N,c}$, $T^{\phi,m}_{\infty,c,K}:=m\big(\big(\pi_\ell\sum_{k=1}^K\sum_{r=1}^2 2^{-1}b_k(\phi)Z_{k,r,\ell}^2\big)_{\ell=1}^c\big)$, and $T^{\phi,m}_{\infty,c}$, respectively. The proof follows the same lines as the proof of Theorem~\ref{thm:asymp-H0-general}.
    
    The only differences are in showing $|\varphi_K^{(N)}(t)-\varphi^{\infty}_K(t)|=o(1)$ as $N\to \infty$, which follows from the fixed-$K$ argument in Theorem~\ref{thm:asymp-H0-general}, since $\{1,\ldots,K\}\cap N\Z_{+}$ is empty for all $N> K$ and, hence, $\widetilde{T}_{N,c,K}^{\phi,m}=T_{N,c,K}^{\phi,m}$ for all such $N$, and in showing $\sup_{N\geq 2c}|\varphi^{(N)}(t)-\varphi^{(N)}_K(t)|=o(1)$ as $K\to \infty$ for the centered statistic. It follows that for all $N\geq 2c$ and all $t\in\R$,
    \begin{align*}
        |\varphi^{(N)}&(t)-\varphi_K^{(N)}(t)|\\
        &\leq |t|\Es{\Big|\widetilde{T}_{N,c}^{\phi,m}-\widetilde{T}_{N,c,K}^{\phi,m}\Big|}{0}\\
        &= |t|\EsBigg{\bigg|m\bigg(\Big(\sum_{k\notin N\Z_{+}}\sum_{r=1}^{2} 2^{-1}b_k(\phi) n_{\ell}G_{k,r,\ell,N}^2\Big)_{\ell=1}^{c}\bigg)-m\bigg(\Big(\sum_{\substack{k\leq K\\k\notin N\Z_{+}}}\sum_{r=1}^{2} 2^{-1}b_k(\phi) n_{\ell}G_{k,r,\ell,N}^2\Big)_{\ell=1}^{c}\bigg)\bigg|}{0}\\
        &\leq L|t|\EsBigg{\bigg\lVert\Big(\sum_{\substack{k>K\\k\notin N\Z_{+}}}\sum_{r=1}^{2} 2^{-1}b_k(\phi) n_{\ell}G_{k,r,\ell,N}^2\Big)_{\ell=1}^{c}\bigg\rVert_{\infty}}{0}\\
        &\leq L|t|\EsBigg{\sum_{\ell=1}^{c}\sum_{\substack{k>K\\k\notin N\Z_{+}}}\sum_{r=1}^{2}2^{-1}|b_k(\phi)| n_{\ell}G^2_{k,r,\ell,N}}{0}\nonumber\\
        &= L|t|\sum_{\ell=1}^{c}\sum_{\substack{k>K\\k\notin N\Z_{+}}}|b_k(\phi)|\Big(1 - \dfrac{n_{\ell}-1}{N-1}\Big)\nonumber\\
        &= L|t|\Big((c-1)\dfrac{N}{N-1}\sum_{\substack{k>K\\k\notin N\Z_{+}}}|b_k(\phi)|\Big)\nonumber\\
        &\leq 2L|t|(c-1)\sum_{k>K}|b_k(\phi)|,\nonumber
    \end{align*}
    and thus, $\sup_{N\geq 2c}|\varphi^{(N)}(t)-\varphi^{(N)}_K(t)|=o(1)$ as $K\to\infty$.
\end{proof}

\subsection{Proofs of Section~\ref{sec:nonnull-asymp}}\label{subsec:proofs:nonnull-asymp}

\begin{proof}[Proof of Theorem~\ref{thm:asymp-Halt-general}]
    Let $\psi:[-1,1]\to\R$ be defined as $\psi(x):=\phi(\cos(2\pi x))$ for all $x\in[-1,1]$. Let $X_{\ell,N}:=\pi_{\ell,N}\int_{0}^{2\pi}\int_{0}^{2\pi} \psi(\widehat{H}_N(\alpha) - \widehat{H}_N(\beta))\,\rd \widehat{F}_{\ell,n_{\ell}}(\alpha)\,\rd \widehat{F}_{\ell,n_{\ell}}(\beta)$, $\widetilde{X}_{\ell,N}:=\pi_{\ell}\int_{0}^{2\pi}\int_{0}^{2\pi}\psi(H(\alpha)-H(\beta))\,\rd \widehat{F}_{\ell,n_{\ell}}(\alpha)\,\rd \widehat{F}_{\ell,n_{\ell}}(\beta)$, and $x_\ell:=\pi_{\ell}\sum_{k=1}^{\infty}b_k(\phi)\big(\gamma_{\ell,k,1}^2 + \gamma_{\ell,k,2}^2\big)$ for each $\ell=1,\ldots,c$, and define $\bX_N:=(X_{1,N},\ldots,X_{c,N})$, $\widetilde{\bX}_{N}:=(\widetilde{X}_{1,N},\ldots,\widetilde{X}_{c,N})$, and $\bx:=(x_{1},\ldots,x_{c})$. We begin by showing that
    \begin{align}\label{eq:fixed-alt-general-x}
        \|\bX_N-\widetilde{\bX}_N\|\to 0\quad\text{almost surely}.
    \end{align}

    Fix $\ell=1,\ldots,c$. Then,
    \begin{align*}
        |X_{\ell,N} - \widetilde{X}_{\ell,N}|&=\left|\int_{0}^{2\pi}\int_{0}^{2\pi} \left[\pi_{\ell,N}\psi (\widehat{H}_N(\alpha) - \widehat{H}_N(\beta)) - \pi_{\ell}\psi(H(\alpha)-H(\beta))\right]\,\rd \widehat{F}_{\ell,n_{\ell}}(\alpha)\,\rd \widehat{F}_{\ell,n_{\ell}}(\beta)\right|\\
        &\leq \bigg(\pi_{\ell,N}\bigg|\int_{0}^{2\pi}\int_{0}^{2\pi} \left[\psi (\widehat{H}_N(\alpha) - \widehat{H}_N(\beta)) - \psi(H(\alpha)-H(\beta))\right]\,\rd \widehat{F}_{\ell,n_{\ell}}(\alpha)\,\rd \widehat{F}_{\ell,n_{\ell}}(\beta)\bigg|\nonumber\\
        &\quad+ \bigg|\int_{0}^{2\pi}\int_{0}^{2\pi}(\pi_{\ell,N} - \pi_{\ell})\;\psi(H(\alpha)-H(\beta))\,\rd \widehat{F}_{\ell,n_{\ell}}(\alpha)\,\rd \widehat{F}_{\ell,n_{\ell}}(\beta)\bigg|\bigg)\nonumber\\
        &\leq \bigg(\pi_{\ell,N}\dfrac{1}{n_\ell^2}\sum_{i,j=1}^{n_\ell} \bigg|\psi (\widehat{H}_N(\Theta_i^{(\ell)}) - \widehat{H}_N(\Theta_j^{(\ell)})) - \psi(H(\Theta_i^{(\ell)})-H(\Theta_j^{(\ell)}))\bigg|\nonumber\\
        &\quad+ \dfrac{1}{n_\ell^2}|\pi_{\ell,N} - \pi_{\ell}|\sum_{i,j=1}^{n_\ell}\;\left|\psi(H(\Theta_i^{(\ell)})-H(\Theta_j^{(\ell)}))\right|\bigg)\nonumber\\
        &\leq \bigg(\sup_{i,j=1,\ldots,n_\ell}\bigg|\psi (\widehat{H}_N(\Theta_i^{(\ell)}) - \widehat{H}_N(\Theta_j^{(\ell)})) - \psi(H(\Theta_i^{(\ell)})-H(\Theta_j^{(\ell)}))\bigg|\nonumber\\
        &\quad+ |\pi_{\ell,N} - \pi_{\ell}|\sup_{i,j=1,\ldots,n_\ell}\left|\psi(H(\Theta_i^{(\ell)})-H(\Theta_j^{(\ell)}))\right|\bigg).
    \end{align*}
    The first term in the last inequality converges almost surely to zero by Lemma~\ref{lemma:consistency}. For the second term, the function $\psi$ is continuous on a compact set; therefore there exists some $C>0$ such that $|\psi(x)|\leq C$ for all $x\in[-1,1]$. Thus, as $N\to\infty$,
    \begin{align*}
        |\pi_{\ell,N} - \pi_{\ell}|\sup_{i,j=1,\ldots,n_\ell}\left|\psi(H(\Theta_i^{(\ell)})-H(\Theta_j^{(\ell)}))\right|\leq C|\pi_{\ell,N} - \pi_{\ell}|\to 0.
    \end{align*}
    Since this holds for all $\ell=1,\ldots,c$, \eqref{eq:fixed-alt-general-x} follows.

    Note that for each $\ell=1,\ldots,c$,
    $$
        \dfrac{\widetilde{X}_{\ell,N}}{\pi_\ell}=\int_{0}^{2\pi}\int_{0}^{2\pi}\psi(H(\alpha)-H(\beta))\,\rd \widehat{F}_{\ell,n_{\ell}}(\alpha)\,\rd \widehat{F}_{\ell,n_{\ell}}(\beta)
        =\dfrac{1}{n_\ell^2}\sum_{i,j=1}^{n_\ell}\psi(H(\Theta_i^{(\ell)})-H(\Theta_j^{(\ell)}))
    $$
    is a $V$-statistic. Thus, by the SLLN of $V$-statistics (see, e.g., Theorem 3.3.1 in \cite{Koroljuk1994|SM}), the following convergence holds almost surely,
    \begin{align}\label{eq:SLLN-Vstat-general}
        \widetilde{X}_{\ell,N}=\dfrac{\pi_{\ell}}{n_\ell^2}\sum_{i,j=1}^{n_\ell}\psi(H(\Theta_i^{(\ell)})-H(\Theta_j^{(\ell)}))&\to \pi_{\ell}\int_{0}^{2\pi}\int_{0}^{2\pi}\psi(H(\alpha)-H(\beta))\,\rd F_{\ell}(\alpha)\,\rd F_{\ell}(\beta).
    \end{align}
    Also, the RHS of~\eqref{eq:SLLN-Vstat-general} is equivalent to
    \begin{align*}
        \pi_{\ell}\int_{0}^{2\pi}\int_{0}^{2\pi}&\psi(H(\alpha)-H(\beta))\,\rd F_{\ell}(\alpha)\,\rd F_{\ell}(\beta)\nonumber\\
        &=\pi_{\ell}\int_{0}^{2\pi}\int_{0}^{2\pi}\sum_{k=1}^{\infty}b_k(\phi)\cos(2\pi k(H(\alpha)-H(\beta)))\,\rd F_{\ell}(\alpha)\,\rd F_{\ell}(\beta)\nonumber\\
        &=\pi_{\ell}\sum_{k=1}^{\infty}b_k(\phi)\int_{0}^{2\pi}\int_{0}^{2\pi}\cos(2\pi k(H(\alpha)-H(\beta)))\,\rd F_{\ell}(\alpha)\,\rd F_{\ell}(\beta)\nonumber\\
        &=\pi_{\ell}\sum_{k=1}^{\infty}b_k(\phi)\bigg[\int_{0}^{2\pi}\int_{0}^{2\pi}\cos(2\pi kH(\alpha))\cos(2\pi kH(\beta))\,\rd F_\ell(\alpha)\,\rd F_{\ell}(\beta) \nonumber\\
        &\qquad\qquad\qquad\qquad\quad+ \int_{0}^{2\pi}\int_{0}^{2\pi}\sin(2\pi kH(\alpha))\sin(2\pi kH(\beta))\,\rd F_{\ell}(\alpha)\,\rd F_{\ell}(\beta)\bigg]\nonumber\\
        &=\pi_{\ell}\sum_{k=1}^{\infty}b_k(\phi)\bigg[\Big(\int_{0}^{2\pi}\cos(2\pi kH(\alpha))\,\rd F_{\ell}(\alpha)\Big)^2 + \Big(\int_{0}^{2\pi}\sin(2\pi k H(\beta))\,\rd F_{\ell}(\beta)\Big)^2\bigg]\nonumber\\
        &=\pi_{\ell}\sum_{k=1}^{\infty}b_k(\phi)\big(\gamma_{\ell,k,1}^2 + \gamma_{\ell,k,2}^2\big),
    \end{align*}
    where the second equality holds by the uniform convergence of the series. Thus, we have that 
    \begin{align}\label{eq:fixed-alt-general-x2}
        \widetilde{\bX}_N\to \bx\quad\text{almost surely}.
    \end{align}
    Then,~\eqref{eq:fixed-alt-general-x} and~\eqref{eq:fixed-alt-general-x2} yield that $\|\bX_N - \bx\|\to0$ almost surely.

    Finally, by conditions~\eqref{m:lip} and~\eqref{m:hom} and Remark~\ref{rmk:m-uniform-lim},
    \begin{align*}
        \sup_{\by\in K}\big|N^{-1}m(N\by)-m_\infty(\by)\big|\to0
    \end{align*}
    for every compact set $K\subset\R^c$. Moreover, $m_\infty$ is continuous. Let
    $\Omega_0:=\{\omega:\bX_N(\omega)\to\bx\}$, so that $\mathrm{P}(\Omega_0)=1$ by the preceding argument, and fix the compact set
    $K:=\{\by\in\R^c:\|\by-\bx\|\leq1\}$. For every $\omega\in\Omega_0$, there exists $N_\omega$ such that $\bX_N(\omega)\in K$ for all $N\geq N_\omega$. Hence, for all such $N$,
    \begin{align*}
        \left|N^{-1}T_{N,c}^{\phi,m}(\omega)-m_\infty(\bx)\right|
        &=\left|N^{-1}m\big(N\bX_N(\omega)\big)-m_\infty(\bx)\right|\\
        &\leq \sup_{\by\in K}\left|N^{-1}m(N\by)-m_\infty(\by)\right|
        +\left|m_\infty\big(\bX_N(\omega)\big)-m_\infty(\bx)\right|\\
        &\to0.
    \end{align*}
    The first term converges to zero by the compact-uniform convergence above, whereas the second does so by the continuity of $m_\infty$ and the convergence $\bX_N(\omega)\to\bx$. Since this holds for every $\omega\in\Omega_0$ and $\mathrm{P}(\Omega_0)=1$, the result follows.
\end{proof}

\begin{proof}[Proof of Corollary~\ref{cor:consistency-Halt-general}]
    Let $\mu_{\ell}$ be the probability measure on $[0,2\pi)$ with cdf $F_\ell$, for any $\ell=1,\ldots,c$. Let $\mu_H$ be the probability measure with cdf $H$. Since $\mu_H(A)=\sum_{\ell=1}^{c} \pi_{\ell} \mu_\ell(A)$ for any Borel set $A$ with $\pi_\ell>0$ for every $\ell=1,\ldots,c$, it follows that $\mu_{\ell}\ll\mu_H$. Then, by the Radon--Nikodym theorem, there exists a function $\rho_\ell:=\rd \mu_{\ell}/\rd \mu_H$ for all $\ell=1,\ldots,c$. In addition, $\rho_\ell$ is bounded $\mu_H$-almost everywhere, since $\mu_\ell(A)\leq \mu_H(A)/\pi_0$ for all Borel sets $A$, where $0<\pi_0:=\min_{1\leq \ell\leq c} \pi_\ell$. Thus,
    \begin{align}\label{eq:Fourier_fixedalt-general}
        \gamma_{\ell,k,r}&=\dfrac{1}{\sqrt{2}}\int_{0}^{2\pi}g_{k,r}(2\pi H(\theta))\,\rd F_{\ell}(\theta)\nonumber\\
        &=\dfrac{1}{\sqrt{2}}\int_{0}^{2\pi}g_{k,r}(2\pi H(\theta))\rho_\ell(\theta)\,\rd \mu_H(\theta).
    \end{align}
    
    Note that under $\Hcal_{1,{\rm fix}}$, there must be at least two distinct values $\ell_1,\ell_2\in\{1,\ldots,c\}$ such that $\mu_{\ell_1}\neq \mu_H$ and $\mu_{\ell_2}\neq \mu_H$. If at most one index $j$ satisfies $\mu_j\neq\mu_H$, then all $\ell\neq j$ have $\mu_\ell=\mu_H$, and $\mu_H=\sum_{\ell\neq j}\pi_\ell\mu_H+\pi_j\mu_j=(1-\pi_j)\mu_H+\pi_j\mu_j$, hence $\mu_j=\mu_H$, which is a contradiction.

    Consequently, $\rho_{\ell_1}$ and $\rho_{\ell_2}$ are not constant $\mu_H$-almost everywhere. Fix $\ell$ to be $\ell_1$ or $\ell_2$. Applying the change of variables theorem, see, e.g., Theorem 3.6.1 in \cite{Bogachev2007|SM}, with $Q_H(u):=\inf\{\theta\in[0,2\pi):H(\theta)\geq u\}$ for $u\in[0,1)$, and $\lambda$ being the Lebesgue measure, we have
    \begin{align*}
    \eqref{eq:Fourier_fixedalt-general}&=\dfrac{1}{\sqrt{2}}\int_{0}^{1}g_{k,r}(2\pi H(Q_H(u)))\rho_\ell(Q_H(u))\,\rd u\\
    &=\dfrac{1}{\sqrt{2}}\int_{0}^{1}g_{k,r}(2\pi u)\rho_\ell(Q_H(u))\,\rd u,
    \end{align*}
    since $\lambda\circ Q_H^{-1}=\mu_H$ and $H(Q_H(u))=u$ for all $u\in[0,1)$ by the continuity of $H$. Note that because $\rho_\ell$ is bounded $\mu_H$-almost everywhere, $\rho_\ell \circ Q_H$ is bounded $\lambda$-almost everywhere; thus $\rho_\ell \circ Q_H \in L^2([0,1], \lambda)$. Also, $\rho_\ell \circ Q_H$ is not constant $\lambda$-almost everywhere. Then, the uniqueness of Fourier coefficients implies that there exists some $k_1\geq 1$ such that $\|\bga_{\ell_1,k_1}\|\neq 0$ and $k_2\geq 1$ such that $\|\bga_{\ell_2,k_2}\|\neq 0$. By condition \eqref{m:pos}, the limit in Theorem~\ref{thm:asymp-Halt-general} is positive, and the consistency follows.
\end{proof}

\begin{proof}[Proof of Proposition~\ref{prp:gkr-asymp-Hlocalt}]
    By Lemma~\ref{lemma:unif-taylor}, for each $\ell=1,\ldots,c$,
    \begin{align}\label{eq:F-taylor}
        F_{\ell,N}(\theta)
        &=F(\theta)+N^{-1/2}\delta_{\ell}\bigl(f(\theta)-f(0)\bigr)+R_{\ell,N}^{F}(\theta),
    \end{align}
    where $\sup_{\theta\in[0,2\pi]}|R_{\ell,N}^{F}(\theta)|=o(N^{-1/2})$.
    Since the periodic extension of $f$ belongs to $C^1(\R)$, the functions $f$ and $f'$ are bounded on $[0,2\pi]$.
    Hence,
    \begin{align}\label{eq:H-taylor}
        H_N(\theta)
        &=\sum_{\ell=1}^{c}\pi_{\ell,N}F_{\ell,N}(\theta)\nonumber\\
        &=F(\theta)+\bigl(f(\theta)-f(0)\bigr)\sum_{\ell=1}^{c}\pi_{\ell,N}\Delta_{\ell,N}
        +\sum_{\ell=1}^{c}\pi_{\ell,N}R_{\ell,N}^{F}(\theta)\nonumber\\
        &=F(\theta)+N^{-1/2}\bigl(f(\theta)-f(0)\bigr)\bpi_N^{\top}\bdelta
        +R_{N}^{H}(\theta),
    \end{align}
    with $R_{N}^{H}(\theta):=\sum_{\ell=1}^{c}\pi_{\ell,N}R_{\ell,N}^{F}(\theta)$, $\sup_{\theta\in[0,2\pi]}|R_{N}^{H}(\theta)|=o(N^{-1/2})$, and $\bpi_N:=(\pi_{1,N},\ldots,\pi_{c,N})^{\top}$.

    Let $\psi_{k,r}(u):=g_{k,r}(2\pi u)$. Then, $\psi'_{k,r}(u)=2\pi g_{k,r}'(2\pi u)$, 
    \begin{align}\label{eq:int-psi-0}
        \int_{0}^{2\pi}\psi_{k,r}(F(\theta))\,\rd F(\theta)=\int_{0}^{1}g_{k,r}(2\pi u)\,\rd u=0,
    \end{align}
    and 
    \begin{align}\label{eq:int-psi-1}
        \int_{0}^{2\pi}\psi'_{k,r}(F(\theta))\,\rd F(\theta)=2\pi\int_{0}^{1}g'_{k,r}(2\pi u)\,\rd u=0.
    \end{align}

    To prove the asymptotic normality of $N^{1/2}(\bG_N-\bmu_N)$, where $\bmu_N:=(\mu_{k,r,\ell,N})$ is indexed conformably with $\bG_N$, we apply Lemma~\ref{lemma:gkr-clt-moving}. To that aim, we compute $\mu_{k,r,\ell,N}$ and the limit of $I_{k,k',r,r',N}^{(i,j,\ell)}$ in \eqref{eq:Ilim}.
    For the first, we apply~\eqref{eq:H-taylor} and Taylor's theorem to $\psi_{k,r}$ around $F(\theta)\in[0,1]$,
    \begin{align*}
        \psi_{k,r}(H_N(\theta))=\psi_{k,r}(F(\theta)) + \psi'_{k,r}(F(\theta))(H_N(\theta)-F(\theta)) + R_{k,r,N}^{\psi}(\theta),
    \end{align*}
    where, for each $\theta\in[0,2\pi]$, there exists $\xi_{N,\theta}$ between $F(\theta)$ and $H_N(\theta)$ such that
    \begin{align*}
        R_{k,r,N}^{\psi}(\theta)=\frac{1}{2}\psi''_{k,r}(\xi_{N,\theta})\big(H_N(\theta)-F(\theta)\big)^2.
    \end{align*}
    Therefore,
    \begin{align*}
        \sup_{\theta\in[0,2\pi]}|R_{k,r,N}^{\psi}(\theta)|
        \leq \frac{1}{2}\sup_{u\in[0,1]}|\psi''_{k,r}(u)|\sup_{\theta\in[0,2\pi]}|H_N(\theta)-F(\theta)|^2
        =O(N^{-1})
        =o(N^{-1/2}),
    \end{align*}
    since $\psi_{k,r}\in C^2([0,1])$ and $\sup_{\theta\in[0,2\pi]}|H_N(\theta)-F(\theta)|=O(N^{-1/2})$, which yields
    \begin{align}
        \mu_{k,r,\ell,N}
        :=\;&\int_{0}^{2\pi}\psi_{k,r}\big(H_N(\theta)\big)\,\rd F_{\ell,N}(\theta)\nonumber\\
        =\;&\int_{0}^{2\pi}\psi_{k,r}\Big(F(\theta)+N^{-1/2}\big(f(\theta)-f(0)\big) \bpi_N^{\top}\bdelta+R_{N}^{H}(\theta)\Big)\,\rd F_{\ell,N}(\theta)\nonumber\\
        =\;&\int_{0}^{2\pi}\psi_{k,r}\big(F(\theta)\big)\,\rd F_{\ell,N}(\theta) + \int_{0}^{2\pi}\psi'_{k,r}\big(F(\theta)\big)\Big(N^{-1/2}\big(f(\theta)-f(0)\big) \bpi_N^{\top}\bdelta+R_{N}^{H}(\theta)\Big)\,\rd F_{\ell,N}(\theta)\nonumber\\
        &+ \int_{0}^{2\pi}R_{k,r,N}^{\psi}(\theta)\,\rd F_{\ell,N}(\theta)\nonumber\\
        =\;&\int_{0}^{2\pi}\psi_{k,r}\big(F(\theta)\big)\,\rd F_{\ell,N}(\theta) + N^{-1/2}\bpi_N^{\top}\bdelta\int_{0}^{2\pi}\psi'_{k,r}\big(F(\theta)\big)\big(f(\theta)-f(0)\big) \,\rd F_{\ell,N}(\theta)\nonumber\\
        &+ \int_{0}^{2\pi}\psi'_{k,r}\big(F(\theta)\big)R_{N}^{H}(\theta)\,\rd F_{\ell,N}(\theta) + o(N^{-1/2})\nonumber\\
        =\;&\int_{0}^{2\pi}\psi_{k,r}\big(F(\theta)\big)\,\rd F_{\ell,N}(\theta) + N^{-1/2}\bpi_N^{\top}\bdelta\int_{0}^{2\pi}\psi'_{k,r}\big(F(\theta)\big)\big(f(\theta)-f(0)\big) \,\rd F_{\ell,N}(\theta)\nonumber\\
        &+ o(N^{-1/2}),\label{eq:mu-localt-1}
    \end{align}
    where in the last equation we used that $\psi'_{k,r}$ is bounded on $[0,1]$, that $\sup_{\theta\in[0,2\pi]}|R_{N}^{H}(\theta)|=o(N^{-1/2})$, and that $F_{\ell,N}$ is a probability measure, so that
    \begin{align*}
        \int_{0}^{2\pi}\psi'_{k,r}\big(F(\theta)\big)R_{N}^{H}(\theta)\,\rd F_{\ell,N}(\theta)=o(N^{-1/2}).
    \end{align*}
    Applying the expansion in Lemma~\ref{lemma:unif-taylor} to $f_{\ell,N}$ with $\Delta=\Delta_{\ell,N}$, it follows that
    \begin{align}\label{eq:taylor-fl}
        f_{\ell,N}(\theta)=f(\theta) + N^{-1/2}f'(\theta)\delta_\ell + R_{\ell,N}^{f}(\theta),\quad\text{with}\quad \sup_{\theta\in[0,2\pi]}|R_{\ell,N}^{f}(\theta)|=o(N^{-1/2}),
    \end{align}
    and, thus,
    \begin{align}
        \eqref{eq:mu-localt-1}=\;&\int_{0}^{2\pi}\psi_{k,r}\big(F(\theta)\big) \Big(f(\theta) + N^{-1/2} f'(\theta)\delta_{\ell} + R_{\ell,N}^{f}(\theta)\Big)\,\rd \theta\nonumber\\
        &+ N^{-1/2}\bpi_N^{\top}\bdelta\int_{0}^{2\pi}\psi'_{k,r}\big(F(\theta)\big)\big(f(\theta)-f(0)\big) \Big(f(\theta) + N^{-1/2}f'(\theta)\delta_{\ell} + R_{\ell,N}^{f}(\theta)\Big)\,\rd \theta + o(N^{-1/2})\nonumber\\
        =\;&\int_{0}^{2\pi}\psi_{k,r}\big(F(\theta)\big) \,\rd F(\theta)+ N^{-1/2}\delta_{\ell}\int_{0}^{2\pi}\psi_{k,r}\big(F(\theta)\big)f'(\theta)\,\rd \theta + \int_{0}^{2\pi}\psi_{k,r}\big(F(\theta)\big)R_{\ell,N}^{f}(\theta)\,\rd \theta\nonumber\\
        &+ N^{-1/2}\bpi_N^{\top}\bdelta\int_{0}^{2\pi}\psi'_{k,r}\big(F(\theta)\big)\big(f(\theta)-f(0)\big)\,\rd F(\theta)\nonumber\\
        &+ N^{-1/2}\bpi_N^{\top}\bdelta\int_{0}^{2\pi}\psi'_{k,r}\big(F(\theta)\big)\big(f(\theta)-f(0)\big)R_{\ell,N}^{f}(\theta)\,\rd \theta + o(N^{-1/2})\nonumber\\
        =\;&\int_{0}^{2\pi}\psi_{k,r}\big(F(\theta)\big) \,\rd F(\theta)+ N^{-1/2}\delta_{\ell}\int_{0}^{2\pi}\psi_{k,r}\big(F(\theta)\big)f'(\theta)\,\rd \theta\nonumber\\
        &+ N^{-1/2}\bpi_N^{\top}\bdelta\int_{0}^{2\pi}\psi'_{k,r}\big(F(\theta)\big)\big(f(\theta)-f(0)\big)\,\rd F(\theta) + o(N^{-1/2}),
        \label{eq:mu-localt-2}
    \end{align}
    where the third equality uses that $\psi_{k,r}$ and $\psi'_{k,r}$ are bounded on $[0,1]$, that $f$ is bounded on $[0,2\pi]$, and that $\sup_{\theta\in[0,2\pi]}|R_{\ell,N}^{f}(\theta)|=o(N^{-1/2})$. Consequently, the terms
    \begin{align*}
        N^{-1/2}\bpi_N^{\top}\bdelta\int_{0}^{2\pi}\psi'_{k,r}\big(F(\theta)\big)\big(f(\theta)-f(0)\big)R_{\ell,N}^{f}(\theta)\,\rd \theta\quad\text{and}\quad\int_{0}^{2\pi}\psi_{k,r}\big(F(\theta)\big)R_{\ell,N}^{f}(\theta)\,\rd \theta
    \end{align*}
    are both $o(N^{-1/2})$.
    Due to~\eqref{eq:int-psi-0} and~\eqref{eq:int-psi-1}, it holds that
    \begin{align*}
        \eqref{eq:mu-localt-2}=\;&N^{-1/2}\delta_{\ell}\int_{0}^{2\pi}\psi_{k,r}(F(\theta))f'(\theta)\,\rd \theta + N^{-1/2}\bpi_N^{\top}\bdelta\int_{0}^{2\pi}\psi'_{k,r}(F(\theta))f^2(\theta)\,\rd \theta + o(N^{-1/2})\nonumber\\
        =\;&N^{-1/2}(\bpi_N-\be_{\ell})^{\top}\bdelta\int_{0}^{2\pi}\psi'_{k,r}(F(\theta))f^2(\theta)\,\rd \theta + o(N^{-1/2})\nonumber\\
        =\;&N^{-1/2}(\bpi_N-\be_{\ell})^{\top}\bdelta\beta_{k,r} + o(N^{-1/2}),
    \end{align*}
    where the second equality uses integration by parts, yielding
    \begin{align*}
        \int_{0}^{2\pi}\psi_{k,r}(F(\theta))f'(\theta)\,\rd \theta&=-\int_{0}^{2\pi} \psi_{k,r}'(F(\theta)) f(\theta)^2 \,\rd \theta,
    \end{align*}
    after noting that $\psi_{k,r}(1)=\psi_{k,r}(0)$ and $f(0)=f(2\pi)$.
    
    We now compute $I_{k,k',r,r',N}^{(i,j,\ell)}$. Simplifying the expansion from~\eqref{eq:F-taylor}, we can write
    \begin{align*}
        F_{\ell,N}(\theta)
        &=F(\theta)+\widetilde{R}_{\ell,N}^{F}(\theta),\quad\text{with}\quad \sup_{\theta\in[0,2\pi]}|\widetilde{R}_{\ell,N}^{F}(\theta)|=O(N^{-1/2}),
    \end{align*}
    which yields
    \begin{align}
        I_{k,k',r,r',N}^{(i,j,\ell)}=\;&\iint_{0<x<y<2\pi} F_{\ell,N}(x)(1-F_{\ell,N}(y)) \psi'_{k,r}\big(H_N(x)\big)\psi'_{k',r'}\big(H_N(y)\big)\,\rd F_{i,N}(x)\,\rd F_{j,N}(y)\nonumber\\
        &+ \iint_{0<y<x<2\pi} F_{\ell,N}(y)(1-F_{\ell,N}(x)) \psi'_{k,r}\big(H_N(x)\big) \psi'_{k',r'}\big(H_N(y)\big)\,\rd F_{i,N}(x)\,\rd F_{j,N}(y)\nonumber\\
        =\;&\iint_{0<x<y<2\pi} \Big(F(x)+\widetilde{R}_{\ell,N}^{F}(x)\Big)\Big(1-F(y)-\widetilde{R}_{\ell,N}^{F}(y)\Big)\nonumber\\
        &\qquad\times \psi'_{k,r}\big(H_N(x)\big)\psi'_{k',r'}\big(H_N(y)\big)\,\rd F_{i,N}(x)\,\rd F_{j,N}(y)\nonumber\\
        &+ \iint_{0<y<x<2\pi} \Big(F(y)+\widetilde{R}_{\ell,N}^{F}(y)\Big)\Big(1-F(x)-\widetilde{R}_{\ell,N}^{F}(x)\Big)\nonumber\\
        &\qquad\times \psi'_{k,r}\big(H_N(x)\big) \psi'_{k',r'}\big(H_N(y)\big)\,\rd F_{i,N}(x)\,\rd F_{j,N}(y).\label{eq:Ilim-localt-1}
    \end{align}
    By~\eqref{eq:F-taylor} and~\eqref{eq:taylor-fl}, there exist functions $A_{\ell,N}^{(1)},A_{\ell,N}^{(2)},B_{i,j,N}:[0,2\pi]^2\to\R$ such that
    \begin{align*}
        F_{\ell,N}(x)\bigl(1-F_{\ell,N}(y)\bigr)&=F(x)\bigl(1-F(y)\bigr)+A_{\ell,N}^{(1)}(x,y),\\
        F_{\ell,N}(y)\bigl(1-F_{\ell,N}(x)\bigr)&=F(y)\bigl(1-F(x)\bigr)+A_{\ell,N}^{(2)}(x,y),\\
        f_{i,N}(x)f_{j,N}(y)&=f(x)f(y)+B_{i,j,N}(x,y),
    \end{align*}
    with
    \begin{align*}
        \sup_{x,y\in[0,2\pi]}\Big(|A_{\ell,N}^{(1)}(x,y)|+|A_{\ell,N}^{(2)}(x,y)|+|B_{i,j,N}(x,y)|\Big)=O(N^{-1/2}).
    \end{align*}
    Therefore,
    \begin{align}
        \eqref{eq:Ilim-localt-1}=\;&\iint_{0<x<y<2\pi} \Big(F(x)\big(1-F(y)\big)+A_{\ell,N}^{(1)}(x,y)\Big)\psi'_{k,r}\big(H_N(x)\big)\psi'_{k',r'}\big(H_N(y)\big)\nonumber\\
        &\qquad\times\Big(f(x)f(y)+B_{i,j,N}(x,y)\Big)\,\rd x\,\rd y\nonumber\\
        &+ \iint_{0<y<x<2\pi} \Big(F(y)\big(1-F(x)\big)+A_{\ell,N}^{(2)}(x,y)\Big)\psi'_{k,r}\big(H_N(x)\big)\psi'_{k',r'}\big(H_N(y)\big)\nonumber\\
        &\qquad\times\Big(f(x)f(y)+B_{i,j,N}(x,y)\Big)\,\rd x\,\rd y.\label{eq:Ilim-localt-2}
    \end{align}
    Also, since $\psi'_{k,r}\in C^1([0,1])$ and $\sup_{\theta\in[0,2\pi]}|H_N(\theta)-F(\theta)|=O(N^{-1/2})$, there exists $C_{k,r,N}:[0,2\pi]\to\R$ such that
    \begin{align*}
        \psi'_{k,r}(H_N(\theta))=\psi'_{k,r}(F(\theta))+C_{k,r,N}(\theta)
        \quad\text{and}\quad
        \sup_{\theta\in[0,2\pi]}|C_{k,r,N}(\theta)|=O(N^{-1/2}).
    \end{align*}
    Since $F$, $f$, $\psi'_{k,r}$, and $\psi'_{k',r'}$ are bounded on their domains, these uniform estimates imply that there exists a function $R_N:[0,2\pi]^2\to\R$ with $\sup_{x,y\in[0,2\pi]}|R_N(x,y)|=O(N^{-1/2})$ such that
    \begin{align}
        \eqref{eq:Ilim-localt-2}=\;&\iint_{0<x<y<2\pi}\Big(F(x)\big(1-F(y)\big)f(x)f(y)\psi'_{k,r}(F(x))\psi'_{k',r'}(F(y))+R_N(x,y)\Big)\,\rd x\,\rd y\nonumber\\
        &+\iint_{0<y<x<2\pi}\Big(F(y)\big(1-F(x)\big)f(x)f(y)\psi'_{k,r}(F(x))\psi'_{k',r'}(F(y))+R_N(x,y)\Big)\,\rd x\,\rd y.\label{eq:Ilim-localt-3}
    \end{align}
    Since the integration domain has finite area, $\iint_{(0,2\pi)^2}|R_N(x,y)|\,\rd x\,\rd y=O(N^{-1/2})$, and therefore
    \begin{align*}
        \eqref{eq:Ilim-localt-3}=\;&\iint_{0<x<y<2\pi} F(x)\big(1 - F(y)\big)\psi'_{k,r}(F(x))\psi'_{k',r'}(F(y)) \,\rd F(x)\,\rd F(y)\nonumber\\
        &+ \iint_{0<y<x<2\pi} F(y)\big(1 - F(x)\big)\psi'_{k,r}(F(x))\psi'_{k',r'}(F(y))\,\rd F(x)\,\rd F(y)+ O(N^{-1/2}).
    \end{align*}
    
    Thus, by \eqref{eq:Ilim-H0},
    \begin{align*}
        \lim_{N\to\infty} I_{k,k',r,r',N}^{(i,j,\ell)}&=\delta_{kk'}\delta_{rr'},
    \end{align*}
    and the assumptions of Lemma~\ref{lemma:gkr-clt-moving} are satisfied. To obtain the covariance matrix, we follow the computations done in the proof of Proposition~\ref{prp:gkr-asymp-H0}, and we obtain~\eqref{eq:sigma0_diag},~\eqref{eq:sigma0_diag2} and~\eqref{eq:sigma0_off}, proving that the limiting covariance matrix is the same as under $\Hcal_0$, that is, $\bSigma_0$. Hence, $N^{1/2}(\bG_N - \bmu_N)\inlaw \mathcal{N}_D(\mathbf{0},\bSigma_0)$.
    
    Letting $\bbeta:=\lim_{N\to\infty}N^{1/2}\bmu_N$, we have for each $k=1,\ldots,K$, $r=1,2$, and $\ell=1,\ldots,c$,
    \begin{align*}
        \eta_{k,r,\ell}=\lim_{N\to\infty} N^{1/2}\mu_{k,r,\ell,N} = \beta_{k,r}(\bpi-\be_{\ell})^{\top}\bdelta,
    \end{align*}
    since $\bpi_N\to\bpi$.
     Then, by Slutsky's theorem,
    \begin{align*}
        N^{1/2}\bG_N = N^{1/2}(\bG_N - \bmu_N) + N^{1/2}\bmu_N\inlaw \mathcal{N}_D(\bbeta,\bSigma_0)
    \end{align*}
    and the result is proven.
\end{proof}

\begin{proof}[Proof of Theorem~\ref{thm:asymp-Hlocalt-general}]
    Since $m$ is continuous due to \eqref{m:lip}, writing $n_\ell G^2_{k,r,\ell,N}=\pi_{\ell,N}(N^{1/2}G_{k,r,\ell,N})^2$ and noting $\bpi_N\to\bpi$, the result follows from Proposition~\ref{prp:gkr-asymp-Hlocalt}, Slutsky's theorem, and the continuous mapping theorem, exactly as in the fixed-$K$ step of the proof of Theorem~\ref{thm:asymp-H0-general}.
\end{proof}

\begin{proof}[Proof of Corollary~\ref{cor:asymp-Hlocalt}]
    Considering $m_{\rm sum}(\bx):=\sum_{\ell=1}^{c}x_\ell$, we only need to show that
    \begin{align}\label{eq:weak-limit-eq-Hlocalt}
        T_{\infty,c,K}^{\phi, m_{\rm sum},{\rm loc}}
        \equald T_{\infty,c,K}^{\phi,{\rm loc}}.
    \end{align}
    We can write
    \begin{align*}
        T_{\infty,c,K}^{\phi,m_{\rm sum},{\rm loc}}=
        2^{-1}\bZ^{\top}\big(\operatorname{diag}(b_1(\phi),\ldots,b_K(\phi))\otimes\bI_2\otimes\bPi\big)\bZ,
    \end{align*}
    where $\bZ\sim \mathcal{N}_{D}(\bbeta,\bSigma_0)$.
    
    Due to the diagonal block structure, we work with each block $\bZ_{k,r}\sim\mathcal{N}_{c}(\bbeta_{k,r},\bS)$ with $\bS=\diag{1/\bpi}-\mathbf{1}_{c}\mathbf{1}_{c}^{\top}$ and $\bbeta_{k,r}=(\eta_{k,r,1},\ldots,\eta_{k,r,c})^{\top}$, so that
    \begin{align*}
        \bZ^{\top}\big(\operatorname{diag}(b_1(\phi),\ldots,b_K(\phi))\otimes\bI_2\otimes\bPi\big)\bZ=\sum_{k=1}^{K}\sum_{r=1}^{2}b_k(\phi) \bZ_{k,r}^{\top}\bPi\bZ_{k,r}.
    \end{align*}
    Since $\bZ\sim\mathcal{N}_D(\bbeta,\bSigma_0)$ with $\bSigma_0=\bI_{2K}\otimes \bS$, the vectors $\bZ_{k,r}$ are mutually independent.
    
    Choose $\bbeta_{k,r}^{\ast}$ such that $\bbeta_{k,r}=\bS^{1/2}\bbeta_{k,r}^{\ast}$, which is possible because $\bbeta_{k,r}\in\operatorname{Im}(\bS^{1/2})$. To prove this fact, note that $\bS$ and $\bS^{1/2}$ are positive semidefinite and symmetric matrices. Then, $\operatorname{Im}(\bS)=(\ker(\bS))^{\perp}$, and $\operatorname{Im}(\bS^{1/2})=\operatorname{Im}(\bS)$. Noting that $\ker(\bS)=\operatorname{span}\{\bpi\}$, it follows that $\operatorname{Im}(\bS)=\bpi^{\perp}$. Finally, observe that $\bbeta_{k,r}\perp \bpi$, since
    \begin{align*}
        \bpi^{\top}\bbeta_{k,r}&=\beta_{k,r}\sum_{\ell=1}^{c}\pi_\ell(\bpi - \be_\ell)^{\top}\bdelta=0,
    \end{align*}
    and the statement follows.
    
    Then, note that $\bZ_{k,r}\equald\bS^{1/2}\bW_{k,r}$ where $\{\bW_{k,r}:1\leq k\leq K,\ r=1,2\}$ are mutually independent and $\bW_{k,r}\sim\mathcal{N}_c(\bbeta_{k,r}^{\ast}, \bI_c)$, and we have
    \begin{align*}
        \bZ_{k,r}^{\top}\bPi\bZ_{k,r}\equald\bW_{k,r}^{\top}(\bS^{1/2})^{\top}\bPi\bS^{1/2}\bW_{k,r}.
    \end{align*}
    By the diagonalization of $\bD:=(\bS^{1/2})^{\top}\bPi\bS^{1/2}$ carried out in the proof of Corollary~\ref{cor:asymp-H0-sum}, there exist an orthogonal matrix $\bU$ and $\bLambda=\diag{\bI_{c-1},0}$ such that
    \begin{align*}
        \bW_{k,r}^{\top}\bD\bW_{k,r}=(\bU\bW_{k,r})^{\top}\bLambda(\bU\bW_{k,r}).
    \end{align*}
    
    Hence, $\bY_{k,r}:=\bU\bW_{k,r}\sim\mathcal{N}_c(\bU\bbeta_{k,r}^{\ast}, \bI_{c})$,
    \begin{align*}
        \bY_{k,r}^{\top}\bLambda\bY_{k,r}=\sum_{j=1}^{c-1}Y_{k,r,j}^2\sim \chi^2_{c-1}(\lambda_{k,r}),
    \end{align*}
    and the noncentrality parameter is
    \begin{align*}
        \lambda_{k,r}=\|\bLambda\bU\bbeta_{k,r}^{\ast}\|^2&=(\bbeta_{k,r}^{\ast})^{\top}\bU^{\top}\bLambda\bU\bbeta_{k,r}^{\ast}=(\bbeta_{k,r}^{\ast})^{\top}(\bS^{1/2})^{\top}\bPi\bS^{1/2}\bbeta_{k,r}^{\ast}\\
        &=\bbeta_{k,r}^{\top}\bPi\bbeta_{k,r}\\
        &=\beta_{k,r}^2\sum_{\ell=1}^{c}\pi_\ell((\bpi-\be_{\ell})^{\top}\bdelta)^2\\
        &=\beta_{k,r}^2\lrp{\bdelta^{\top}\bPi\bdelta-(\bpi^{\top}\bdelta)^2},
    \end{align*}
    which is clearly nonnegative.
    By independence across $(k,r)$,~\eqref{eq:weak-limit-eq-Hlocalt} follows.
\end{proof}

\section{Technical lemmas}\label{sec:lemma}

This section collects the technical lemmas used in Section~\ref{sec:proofs}. Lemma~\ref{lemma:puri} establishes a triangular-array central limit theorem under moving distributions, while Lemma~\ref{lemma:bSigma} provides the finite-sample covariance identities used to derive its limiting covariance matrix. By applying Lemma~\ref{lemma:puri} to linear combinations and using the Cramér--Wold device, Lemma~\ref{lemma:gkr-clt-moving} establishes a joint central limit theorem for centered finite vectors of empirical harmonic averages evaluated at the uniform scores. This result underlies the proofs of Propositions~\ref{prp:gkr-asymp-H0} and~\ref{prp:gkr-asymp-Hlocalt}.

Lemma~\ref{lemma:puri} appears as Lemma 4.1 in \cite{Mardia1970|SM}, where it is stated without proof. The published statement contains two apparent typographical errors. First, one of his conditions reads $J_{\ell,N}(1)=o(N^{-1/2})$, whereas the intended condition is $J_{\ell,N}(1)=o(N^{1/2})$, as stated in \eqref{cond:puri3}. Second, the entries of the limiting covariance matrix $\bSigma$ are expressed in terms of quantities that depend on $N$, which should be replaced by their corresponding limits. We provide a corrected statement and, for completeness, a proof based on a multivariate triangular-array central limit theorem, with the Lindeberg condition verified through a Lyapunov-type moment bound.

\begin{lemma}[Lemma 4.1 in \cite{Mardia1970|SM}]\label{lemma:puri}
    For each $N\geq c$ and each $\ell=1,\ldots,c$, let $X_{N,1}^{(\ell)}, \ldots,\allowbreak X_{N,n_{\ell}}^{(\ell)}$ be iid random variables with common continuous cdf $F_{\ell,N}$, with $n_\ell\geq 1$ and $N=\sum_{\ell=1}^{c}n_\ell$, and assume the samples are mutually independent. Letting $\pi_{\ell,N}=n_\ell/N$, assume that, for each $\ell=1,\ldots,c$, $\pi_{\ell,N}\to\pi_\ell\in(0,1)$ as $N\to\infty$. Denote by $\widehat{F}_{\ell,N}$ the empirical cdf based on the $\ell$th sample, and let $\widehat{H}_N(x)=\sum_{\ell=1}^{c}\pi_{\ell,N}\widehat{F}_{\ell,N}(x)$ and $H_N(x)=\sum_{\ell=1}^{c}\pi_{\ell,N}F_{\ell,N}(x)$. Consider
    $$
    T_{\ell,N} := \int_{-\infty}^{+\infty} J_{\ell,N}(\widehat{H}_N(x))\,\rd \widehat{F}_{\ell,N}(x)
    $$
    and
    $$
    \mu_{\ell,N} := \int_{-\infty}^{+\infty} J_{\ell}(H_N(x))\,\rd F_{\ell,N}(x),
    $$
    where for each $N$, $J_{1,N},\ldots,J_{c,N}$ are functions such that:
    \begin{enumerate}[label=($c$\textit{\arabic*}),ref=\textit{$c$\arabic*}]
        \item $J_\ell(H)=\lim_{N\to\infty} J_{\ell,N}(H)$ exists for $0<H<1$ and is not constant on $(0,1)$,\label{cond:puri1}
        \item $\int_{\{x:0<\widehat{H}_N(x)<1\}} (J_{\ell,N}(\widehat{H}_N) - J_\ell(\widehat{H}_N))\,\rd \widehat{F}_{\ell,N}(x) = o_{\rm P}(N^{-1/2})$,\label{cond:puri2}
        \item $J_{\ell,N}(1)=o(N^{1/2})$, \label{cond:puri3}
        \item $|\rd^r J_\ell(H)/\rd H^r|\leq K(H(1-H))^{-r-1/2+\delta_0}$ for $r=0,1,2$, some $\delta_0>0$, and some $K>0$ independent of $F_{1,N},\ldots,F_{c,N}$ or $N$.\label{cond:puri4}
    \end{enumerate}
    Let
    \begin{align*}
        A_{ij,N}^{(\ell)}(J_{\ell_1}, J_{\ell_2}):=&\;\iint_{x<y} F_{\ell,N}(x)(1-F_{\ell,N}(y))J_{\ell_1}'(H_N(x))J_{\ell_2}'(H_N(y))\,\rd F_{i,N}(x)\,\rd F_{j,N}(y)\\
        &+ \iint_{y<x} F_{\ell,N}(y)(1-F_{\ell,N}(x))J_{\ell_1}'(H_N(x))J_{\ell_2}'(H_N(y))\,\rd F_{i,N}(x)\,\rd F_{j,N}(y),
    \end{align*}
    and assume that $A_{ij}^{(\ell)}(J_{\ell_1},J_{\ell_2}):=\lim_{N\to\infty}A_{ij,N}^{(\ell)}(J_{\ell_1},J_{\ell_2})$ exists for every $\ell,i,j,\ell_1,\ell_2\in\{1,\ldots,c\}$ such that $A_{ij}^{(\ell)}(J_{\ell_1},J_{\ell_2})$ appears in \eqref{eq:sigma_diag}--\eqref{eq:sigma_off}.
    Let $\bT_{N}:=(T_{1,N},\ldots,T_{c,N})^{\top}$ and $\bmu_N:=(\mu_{1,N},\ldots,\mu_{c,N})^{\top}$.
    Then, 
    $$
    N^{1/2}(\bT_{N}-\bmu_N)\inlaw \mathcal{N}_c(\mathbf{0},\bSigma),
    $$
    where $\bSigma=(\sigma_{ij})_{i,j=1}^c$ is given by
    \begin{align}
        \sigma_{ii}=\sum_{\substack{j=1\\j\neq i}}^{c}\Big(\pi_{j}[A_{ii}^{(j)}(J_i, J_i) + (\pi_{j}/\pi_{i})A_{jj}^{(i)}(J_i,J_i)] + (1/\pi_{i})\sum_{\substack{k=1\\k\neq j\\k\neq i}}^{c} \pi_{j}\pi_{k}A_{jk}^{(i)}(J_i,J_i)\Big)\label{eq:sigma_diag}
    \end{align}
    and
    \begin{align}
        \sigma_{ij}=\sum_{\ell=1}^{c}\pi_{\ell}[A_{ij}^{(\ell)}(J_i,J_j)-A_{\ell j}^{(i)}(J_i,J_j) - A_{\ell i}^{(j)}(J_j,J_i)],\quad i\neq j.\label{eq:sigma_off}
    \end{align}
\end{lemma}

\begin{proof}[Proof of Lemma~\ref{lemma:puri}]
    This proof follows the same general decomposition as the proof of Theorem 6.1 in \cite{Puri1964|SM}, but it differs from Puri's argument in two aspects. First, we work directly with the vector $N^{1/2}(\bT_N-\bmu_N)$, instead of first proving scalar asymptotic normality and then deriving the joint limit. Second, we keep the $N$-dependence of $F_{\ell,N}$, $H_N$, and the covariance terms throughout, so that moving alternatives are handled directly by a multivariate triangular-array CLT. In particular, unlike Puri's strategy through Lemma 5.1 and Corollary 5.1, no separate uniformity argument is needed to pass from fixed to moving distributions.

    Throughout, we will use the notation $C$ as a generic constant that does not depend on $N$ or the distributions $F_{1,N},\ldots,F_{c,N}$, and whose value may change from line to line. Since $c$ is fixed, $c\ge2$, $n_\ell\ge1$, and $\pi_{\ell,N}\to\pi_\ell\in(0,1)$ for all $\ell$, there exists $\pi_0>0$ such that $\pi_0\leq\pi_{\ell,N}\leq1-\pi_0$ for all $\ell=1,\ldots,c$ and all $N\geq c$.
    
    For each $\ell=1,\ldots,c$, the following decomposition holds,
    \begin{align*}
        T_{\ell, N}=\mu_{\ell, N}+B_{1N}^{(\ell)}+B_{2N}^{(\ell)}+R_{N}^{(\ell)},
    \end{align*}
    where
    \begin{align*}
        B_{1N}^{(\ell)}&=\int_{\{x:0<H_N(x)<1\}}J_{\ell}(H_N(x))\, \rd(\widehat{F}_{\ell,N}-F_{\ell,N})(x),\\
        B_{2N}^{(\ell)}&=\int_{\{x:0<H_N(x)<1\}}(\widehat{H}_N(x)-H_N(x))J_{\ell}'(H_N(x))\, \rd F_{\ell,N}(x).
    \end{align*}
    For each fixed $\ell=1,\ldots,c$, it can be shown that higher-order terms satify $R_N^{(\ell)}=o_{\rm P}(N^{-1/2})$. This follows by repeating the proof in \citet[Section~10 of the Appendix]{Puri1964|SM} using conditions~\eqref{cond:puri2}--\eqref{cond:puri4} and the corresponding decomposition of $R_N^{(\ell)}$. Hence, if $\bB_{1N}:=(B_{1N}^{(1)},\ldots,B_{1N}^{(c)})^\top$, $\bB_{2N}:=(B_{2N}^{(1)},\ldots,B_{2N}^{(c)})^\top$, and $\bR_N:=(R_N^{(1)},\ldots,R_N^{(c)})^\top$, then
    \begin{align*}
        N^{1/2}(\bT_N-\bmu_N)=N^{1/2}(\bB_{1N}+\bB_{2N})+N^{1/2}\bR_N,
    \end{align*}
    with $N^{1/2}\bR_N=o_{\rm P}(1)$ because $c$ is fixed. Therefore, by Slutsky's theorem, it suffices to prove that
    \begin{align*}
        N^{1/2}(\bB_{1N}+\bB_{2N})\inlaw \mathcal{N}_c(\mathbf{0},\bSigma).
    \end{align*}

    Integrating by parts coordinatewise, for each $\ell=1,\ldots,c$, we have
    \begin{align*}
        N^{1/2}(B_{1N}^{(\ell)}+B_{2N}^{(\ell)})=
        \dfrac{1}{\sqrt{n_\ell}}\sum_{i=1}^{n_\ell} Z_{N,i}^{(\ell)}
        -\sum_{\substack{m=1\\m\neq \ell}}^{c}\dfrac{1}{\sqrt{n_m}}\sum_{i=1}^{n_m} Y_{N,i}^{(m,\ell)},
    \end{align*}
    where, for $m\neq \ell$,
    \begin{align*}
        Y_{N,i}^{(m,\ell)}&:=\sqrt{\pi_{m,N}}\lrp{B_{\ell,N}(X_{N,i}^{(m)})-\Es{B_{\ell,N}(X_{N,i}^{(m)})}{F_{m,N}}},\quad i=1,\ldots,n_m,
    \end{align*}
    and, for $m=1,\ldots,c$,
    \begin{align*}
        Z_{N,i}^{(m)}&:=\dfrac{1}{\sqrt{\pi_{m,N}}}\lrp{C_{m,N}(X_{N,i}^{(m)}) - \Es{C_{m,N}(X_{N,i}^{(m)})}{F_{m,N}}},\quad i=1,\ldots,n_m,
    \end{align*}
    where
    \begin{align*}
        B_{\ell,N}(x):=\int_{x_{0,N}}^{x}J_{\ell}'(H_N(y))\,\rd F_{\ell,N}(y)\quad\text{and}\quad C_{m,N}(x):=J_m(H_N(x))-\pi_{m,N}B_{m,N}(x),
    \end{align*}
    with a deterministic $x_{0,N}$ chosen such that, arbitrarily, $H_N(x_{0,N})=1/2$. Note that it exists because $H_N$ is continuous.

    For each $m=1,\ldots,c$ and $i=1,\ldots,n_m$, define the vector
    \begin{align*}
        \bU_{N,i}^{(m)}:=\lrp{U_{N,i,1}^{(m)},\ldots,U_{N,i,c}^{(m)}}^\top,\quad\text{where}\quad
        U_{N,i,\ell}^{(m)}:=\begin{cases}
            Z_{N,i}^{(m)}/\sqrt{n_m}, & \ell=m,\\
            -Y_{N,i}^{(m,\ell)}/\sqrt{n_m}, & \ell\neq m.
        \end{cases}
    \end{align*}
    Then
    \begin{align*}
        N^{1/2}(\bB_{1N}+\bB_{2N})=\sum_{m=1}^{c}\sum_{i=1}^{n_m}\bU_{N,i}^{(m)}.
    \end{align*}
    By construction, for each fixed $m$, the vectors $\bU_{N,1}^{(m)},\ldots,\bU_{N,n_m}^{(m)}$ are iid, have mean $\mathbf{0}$, and the collections corresponding to different $m$ are independent.

    From condition~\eqref{cond:puri4} and proceeding as in the proof of Sub-Lemma 5.1 in \cite{Puri1964|SM}, for some $\eta>0$ it holds that
    \begin{align*}
    \Es{|B_{\ell,N}(X_{N,1}^{(m)})|^{2+\eta}}{F_{m,N}}<C,\quad m\neq\ell,
    \qquad\text{and}\qquad
    \Es{|C_{m,N}(X_{N,1}^{(m)})|^{2+\eta}}{F_{m,N}}<C.
    \end{align*}
    Therefore,
    \begin{align*}
        \Es{|Y_{N,1}^{(m,\ell)}|^{2+\eta}}{F_{m,N}}<C,\qquad \Es{|Z_{N,1}^{(m)}|^{2+\eta}}{F_{m,N}}<C,
    \end{align*}
    and each coordinate of $\bU_{N,i}^{(m)}$ has a finite moment of order $2+\eta$. Since $c$ is fixed,
    \begin{align*}
        \Es{\|\bU_{N,i}^{(m)}\|^{2+\eta}}{F_{m,N}}
        &\leq c^{\eta/2}n_m^{-(1+\eta/2)}\lrpbigg{\Es{|Z_{N,i}^{(m)}|^{2+\eta}}{F_{m,N}}
        +\sum_{\substack{\ell=1\\\ell\neq m}}^{c}\Es{|Y_{N,i}^{(m,\ell)}|^{2+\eta}}{F_{m,N}}}
        \leq C n_m^{-(1+\eta/2)},
    \end{align*}
    ensuring the finiteness of the second moment as well. Moreover, for every $m=1,\ldots,c$,
    \begin{align*}
        \lim_{N\to\infty}\sum_{i=1}^{n_m}\Es{\left\|\bU_{N,i}^{(m)}\right\|^{2+\eta}}{F_{m,N}}
        \leq \lim_{N\to\infty} C n_m^{-\eta/2}= 0.
    \end{align*}
    Note that this condition implies the multivariate Lindeberg condition, since for every $\varepsilon>0$,
    \begin{align*}
        \sum_{i=1}^{n_m}\Es{\|\bU_{N,i}^{(m)}\|^2\mathds{1}_{\{\|\bU_{N,i}^{(m)}\|>\varepsilon\}}}{F_{m,N}}
        &\leq \varepsilon^{-\eta}\sum_{i=1}^{n_m}\Es{\|\bU_{N,i}^{(m)}\|^{2+\eta}}{F_{m,N}}\to 0.
    \end{align*}

    For each fixed $m$, let
    \begin{align*}
        \bSigma_N^{(m)}:=\E{\lrp{\sum_{i=1}^{n_m}\bU_{N,i}^{(m)}}\lrp{\sum_{i=1}^{n_m}\bU_{N,i}^{(m)}}^{\top}}.
    \end{align*}
    By Lemma~\ref{lemma:bSigma}, every entry of $\bSigma_N^{(m)}$ is a finite linear combination of quantities of the form $A_{ij,N}^{(m)}(J_r,J_s)$, with coefficients involving $\pi_{j,N}$, products $\pi_{j,N}\pi_{k,N}$, and the factor $1/\pi_{m,N}$. Since $\pi_{j,N}\to\pi_j$, $\pi_{m,N}\to\pi_m>0$, and the corresponding limits $A_{ij}^{(m)}(J_r,J_s)$ exist by hypothesis, it follows that $\bSigma_N^{(m)}\to\bSigma^{(m)}$ entrywise for some matrix $\bSigma^{(m)}$.
    Thus, the multivariate CLT for triangular arrays, see, e.g., Theorem 6.22 in \cite{Henze2024|SM}, gives for each $m=1,\ldots,c$,
    \begin{align*}
        \sum_{i=1}^{n_m}\bU_{N,i}^{(m)} \inlaw \mathcal{N}_c(\mathbf{0},\bSigma^{(m)}).
    \end{align*}

    Since the vectors corresponding to different $m$ are independent, it holds that
    for every $\mathbf{t}\in\R^c$,
    \begin{align*}
        \E{\exp\lrp{i\mathbf{t}^{\top}\lrp{\sum_{m=1}^{c}\sum_{i=1}^{n_m}\bU_{N,i}^{(m)}}}}
        &=\prod_{m=1}^{c}\E{\exp\lrp{i\mathbf{t}^{\top}\lrp{\sum_{i=1}^{n_m}\bU_{N,i}^{(m)}}}}\\
        &\to \prod_{m=1}^{c}\exp\lrp{-\frac12 \mathbf{t}^{\top}\bSigma^{(m)}\mathbf{t}}
        =\exp\lrp{-\frac12 \mathbf{t}^{\top}\bSigma \mathbf{t}},
    \end{align*}
    where $\bSigma:=\sum_{m=1}^{c}\bSigma^{(m)}$. Hence, by L\'evy's continuity theorem,
    \begin{align*}
        N^{1/2}(\bB_{1N}+\bB_{2N}) = \sum_{m=1}^{c}\sum_{i=1}^{n_m}\bU_{N,i}^{(m)} \inlaw \mathcal{N}_c(\mathbf{0},\bSigma).
    \end{align*}
    The diagonal entries of $\bSigma$ are given by
    \begin{align*}
        (\bSigma)_{\ell\ell}=\lim_{N\to\infty}\lrpBigg{\Vs{Z_{N,1}^{(\ell)}}{F_{\ell,N}}+\sum_{\substack{m=1\\m\neq \ell}}^{c}\Vs{Y_{N,1}^{(m,\ell)}}{F_{m,N}}},
    \end{align*}
    while for $\ell\neq \ell'$,
    \begin{align*}
        (\bSigma)_{\ell\ell'}&=\lim_{N\to\infty}\lrpBigg{\sum_{\substack{m=1\\m\neq \ell\\m\neq \ell'}}^{c}\Es{Y_{N,1}^{(m,\ell)}Y_{N,1}^{(m,\ell')}}{F_{m,N}}
        -\Es{Z_{N,1}^{(\ell)}Y_{N,1}^{(\ell,\ell')}}{F_{\ell,N}}
        -\Es{Y_{N,1}^{(\ell',\ell)}Z_{N,1}^{(\ell')}}{F_{\ell',N}}}.
    \end{align*}
    Using the covariance identities collected in Lemma~\ref{lemma:bSigma}, these yield precisely \eqref{eq:sigma_diag} and \eqref{eq:sigma_off}, and the result is proven.
\end{proof}

\begin{lemma}\label{lemma:gkr-clt-moving}
    Let $K\geq 1$ be fixed. For each $N\geq c$ and each $\ell=1,\ldots,c$, let $\Theta_{N,1}^{(\ell)}, \ldots, \Theta_{N,n_{\ell}}^{(\ell)}$ be iid random variables with common continuous cdf $F_{\ell,N}$ on $[0,2\pi)$, with $n_\ell\geq 1$ and $N=\sum_{\ell=1}^{c}n_\ell$, and assume the samples are mutually independent. Letting $\pi_{\ell,N}=n_\ell/N$, assume that, for each $\ell=1,\ldots,c$, $\pi_{\ell,N}\to\pi_\ell\in(0,1)$ as $N\to\infty$. Let $\widehat{F}_{\ell,N}$, $\widehat{H}_N$, and $H_N$ be as defined in Lemma~\ref{lemma:puri}. Also, let $\psi_{k,r}(u):=g_{k,r}(2\pi u)$ for $u\in[0,1]$, $k=1,\ldots,K$, and $r=1,2$.

    For $k=1,\ldots,K$, $r=1,2$, and $\ell=1,\ldots,c$, let
    \begin{align*}
        G_{k,r,\ell,N}
        &:=\dfrac{1}{n_\ell}\sum_{i=1}^{n_\ell}\psi_{k,r}(\widehat{H}_N(\Theta_{N,i}^{(\ell)}))
        =\int_{0}^{2\pi}\psi_{k,r}(\widehat{H}_N(\theta))\,\rd \widehat{F}_{\ell,N}(\theta),\\
        \mu_{k,r,\ell,N}
        &:=\int_{0}^{2\pi}\psi_{k,r}(H_N(\theta))\,\rd F_{\ell,N}(\theta).
    \end{align*}
    Let $D:=2Kc$, $\bG_N:=(\bG_{1,N}^{\top},\ldots, \bG_{K,N}^{\top})^{\top}\in\R^{D}$ with $\bG_{k,N}:=(\bG_{k,1,N}^{\top}, \bG_{k,2,N}^{\top})^{\top}$ where $\bG_{k,r,N}:=(G_{k,r,1,N},\ldots, G_{k,r,c,N})^{\top}$, and $\bmu_N:=(\mu_{k,r,\ell,N})\in\R^{D}$ indexed conformably.
    
    For $i,j,\ell=1,\ldots,c$, $k,k'=1,\ldots,K$, and $r,r'=1,2$, define
    \begin{align*}
        I_{k,k',r,r',N}^{(i,j,\ell)}:=&\;\iint_{0<x<y<2\pi} F_{\ell,N}(x)(1-F_{\ell,N}(y)) \psi'_{k,r}(H_N(x))\psi'_{k',r'}(H_N(y))\,\rd F_{i,N}(x)\,\rd F_{j,N}(y)\nonumber\\
        &+ \iint_{0<y<x<2\pi} F_{\ell,N}(y)(1-F_{\ell,N}(x)) \psi'_{k,r}(H_N(x)) \psi'_{k',r'}(H_N(y))\,\rd F_{i,N}(x)\,\rd F_{j,N}(y).
    \end{align*}
    Assume that, for every choice of indices $(i,j,\ell,k,k',r,r')$, the limit
    \begin{align}
        I_{k,k',r,r'}^{(i,j,\ell)}:=\lim_{N\to\infty}I_{k,k',r,r',N}^{(i,j,\ell)}\label{eq:Ilim}
    \end{align}
    exists as a finite real number. Then there exists a unique symmetric matrix $\bSigma\in\R^{D\times D}$ such that
    \begin{align*}
        N^{1/2}(\bG_N-\bmu_N)\inlaw \mathcal{N}_D(\mathbf{0},\bSigma).
    \end{align*}
\end{lemma}

\begin{proof}[Proof of Lemma~\ref{lemma:gkr-clt-moving}]
    Fix $\bba=(a_{k,r,\ell})\in\R^D$, indexed conformably with $\bG_N$. For each $\ell=1,\ldots,c$, let $\bba_\ell:=(a_{k,r,\ell})_{k=1,\ldots,K;\,r=1,2}\in\R^{2K}$ denote the block corresponding to sample $\ell$, and define
    \begin{align*}
        J_{\ell,\bba}(u):=\sum_{k=1}^{K}\sum_{r=1}^{2}a_{k,r,\ell}\psi_{k,r}(u),\qquad u\in[0,1].
    \end{align*}
    Also let
    \begin{align*}
        G_{\ell,\bba,N}
        &:=\int_{0}^{2\pi}J_{\ell,\bba}(\widehat{H}_N(\theta))\,\rd \widehat{F}_{\ell,N}(\theta),&
        \mu_{\ell,\bba,N}
        &:=\int_{0}^{2\pi}J_{\ell,\bba}(H_N(\theta))\,\rd F_{\ell,N}(\theta),
    \end{align*}
    and write $\bG_{\bba,N}:=(G_{1,\bba,N},\ldots,G_{c,\bba,N})^\top$ and $\bmu_{\bba,N}:=(\mu_{1,\bba,N},\ldots,\mu_{c,\bba,N})^\top$. Then
    \begin{align*}
        \bba^\top\bG_N=\sum_{\ell=1}^{c}\sum_{k=1}^{K}\sum_{r=1}^{2}a_{k,r,\ell}G_{k,r,\ell,N}
        =\sum_{\ell=1}^{c}G_{\ell,\bba,N},
    \end{align*}
    and therefore
    \begin{align*}
        N^{1/2}\bba^{\top}(\bG_N-\bmu_N)=N^{1/2}\mathbf{1}_{c}^{\top}(\bG_{\bba,N}-\bmu_{\bba,N}).
    \end{align*}

    Assume first that $\bba_\ell\neq\mathbf{0}$ for every $\ell=1,\ldots,c$. We apply Lemma~\ref{lemma:puri} to the vector $\bG_{\bba,N}$ with $J_{\ell,\bba,N}\equiv J_{\ell,\bba}$. Condition~\eqref{cond:puri1} holds because $J_{\ell,\bba}$ is a nontrivial finite Fourier combination and hence is not constant. Condition~\eqref{cond:puri2} holds because $J_{\ell,\bba}=J_{\ell,\bba,N}$. Condition~\eqref{cond:puri3} follows from $J_{\ell,\bba}(1)=O(1)=o(N^{1/2})$. Finally, condition~\eqref{cond:puri4} holds because $J_{\ell,\bba}$ and its first two derivatives are bounded on $[0,1]$, so for any fixed $\delta_0\in(0,1/2)$ the required bound follows.

    For $i,j,\ell,\ell_1,\ell_2\in\{1,\ldots,c\}$, the quantities appearing in Lemma~\ref{lemma:puri} satisfy
    \begin{align*}
        A_{ij,N}^{(\ell)}(J_{\ell_1,\bba},J_{\ell_2,\bba})
        =\sum_{k,k'=1}^{K}\sum_{r,r'=1}^{2}a_{k,r,\ell_1}a_{k',r',\ell_2}I_{k,k',r,r',N}^{(i,j,\ell)}.
    \end{align*}
    Since the limits of $I_{k,k',r,r',N}^{(i,j,\ell)}$ exist by hypothesis and the sums are finite, the corresponding limits $\smash{A_{ij}^{(\ell)}(J_{\ell_1,\bba},J_{\ell_2,\bba})}$ required by Lemma~\ref{lemma:puri} also exist. Therefore, $\smash{N^{1/2}(\bG_{\bba,N}-\bmu_{\bba,N})\inlaw \mathcal{N}_c(\mathbf{0},\bSigma_\bba)}$ for some covariance matrix $\bSigma_\bba$, and consequently $N^{1/2}\bba^\top(\bG_N-\bmu_N)\inlaw \mathcal{N}(0,\sigma^2_\bba)$,
    where $\sigma^2_\bba:=\mathbf{1}_{c}^{\top}\bSigma_\bba\mathbf{1}_{c}$. Indeed, substituting
    $
        A_{ij}^{(\ell)}(J_{\ell_1,\bba},J_{\ell_2,\bba})
        =\sum_{k,k'=1}^{K}\sum_{r,r'=1}^{2}a_{k,r,\ell_1}a_{k',r',\ell_2}I_{k,k',r,r'}^{(i,j,\ell)}
    $
    into \eqref{eq:sigma_diag}--\eqref{eq:sigma_off} shows that $\sigma^2_\bba$ is a quadratic form in $\bba$.

    Now suppose that $\bba_\ell=\mathbf{0}$ for some $\ell$, and let $\cL:=\{\ell\in\{1,\ldots,c\}:\bba_\ell\neq\mathbf{0}\}$.
    Assume $\cL$ is nonempty; otherwise $\bba=\mathbf{0}$ and there is nothing to prove. Choose $\tilde{\bba}$ so that $\tilde{\bba}_\ell=\bba_\ell$ for $\ell\in\cL$ and $\tilde{\bba}_\ell$ is any fixed nonzero block for $\ell\notin\cL$. The previous step yields $N^{1/2}(\bG_{\tilde{\bba},N}-\bmu_{\tilde{\bba},N})\inlaw \mathcal{N}_c(\mathbf{0},\bSigma_{\tilde{\bba}})$.
    
    Moreover, $\bba^\top(\bG_N-\bmu_N)
    =\sum_{\ell\in\cL}(G_{\ell,\tilde{\bba},N}-\mu_{\ell,\tilde{\bba},N})
    =\mathbf{1}_{\cL}^\top(\bG_{\tilde{\bba},N}-\bmu_{\tilde{\bba},N})$,
    where $\mathbf{1}_{\cL}:=(\mathds{1}_{\{\ell\in\cL\}})_{\ell=1}^{c}\in\R^{c}$. Hence, by the continuous mapping theorem, $N^{1/2}\bba^\top(\bG_N-\bmu_N)\inlaw \mathcal{N}(0,\sigma^2_\bba)$,
    with $\sigma^2_\bba:=\mathbf{1}_{\cL}^\top\bSigma_{\tilde{\bba}}\mathbf{1}_{\cL}=\sum_{(i,j)\in\cL\times\cL}\sigma_{ij,\tilde{\bba}}$ and $\sigma_{ij,\tilde{\bba}}$ given by \eqref{eq:sigma_diag}--\eqref{eq:sigma_off} using $A_{ij}^{(\ell)}(J_{\ell_1,\tilde{\bba}}, J_{\ell_2,\tilde{\bba}})$. Since $J_{\ell,\tilde{\bba}}=J_{\ell,\bba}$ for every $\ell\in\cL$, $\sigma^2_\bba=\sum_{(i,j)\in\cL\times\cL}\sigma_{ij,\bba}$. Moreover, since $J_{\ell,\bba}=0$ for every $\ell\notin\cL$, $\sigma_{\bba}^2$ coincides with the full-index homogeneous degree-two polynomial in $\bba$ that the previous case yields.

    Therefore, there exists a unique symmetric matrix $\bSigma\in\R^{D\times D}$ such that $\sigma^2_\bba=\bba^\top\bSigma\bba$, $\bba\in\R^{D}$.
    Since for every $\bba\in\R^{D}$,
    $N^{1/2}\bba^\top(\bG_N-\bmu_N)\inlaw \mathcal{N}(0,\bba^\top\bSigma\bba)$,
    the Cramér--Wold device yields
    \begin{align*}
        N^{1/2}(\bG_N-\bmu_N)\inlaw \mathcal{N}_D(\mathbf{0},\bSigma),
    \end{align*}
    which proves the result.
\end{proof}

\begin{lemma}\label{lemma:bSigma}
    Assume the notation and assumptions of Lemma~\ref{lemma:puri}. Fix $N\geq c$ and $m\in\{1,\ldots,c\}$. Define
    \begin{align*}
        Z_{N,1}^{(m)}&:=\dfrac{1}{\sqrt{\pi_{m,N}}}\lrp{C_{m,N}(X_{N,1}^{(m)}) - \Es{C_{m,N}(X_{N,1}^{(m)})}{F_{m,N}}}
    \end{align*}
    and
    \begin{align*}
        Y_{N,1}^{(m,\ell)}&:=\sqrt{\pi_{m,N}}\lrp{B_{\ell,N}(X_{N,1}^{(m)})-\Es{B_{\ell,N}(X_{N,1}^{(m)})}{F_{m,N}}},\qquad \ell\neq m,
    \end{align*}
    where
    \begin{align*}
        B_{r,N}(x):=\int_{x_{0,N}}^{x}J_{r}'(H_N(y))\,\rd F_{r,N}(y),\, r=1,\ldots,c,
    \text{ and }
        C_{m,N}(x):=J_m(H_N(x))-\pi_{m,N}B_{m,N}(x)
    \end{align*}
    with $x_{0,N}$ chosen so that $H_N(x_{0,N})=1/2$. Then,
    \begin{align*}
    \Es{Z_{N,1}^{(m)}Y_{N,1}^{(m,\ell)}}{F_{m,N}}&=\sum_{\substack{j=1\\j\neq m}}^{c}\pi_{j,N}
    A_{j\ell,N}^{(m)}(J_m,J_{\ell}),\qquad\ell\neq m,\\
    \Es{Y_{N,1}^{(m,\ell)}Y_{N,1}^{(m,\ell')}}{F_{m,N}}&=\pi_{m,N}A_{\ell\ell',N}^{(m)}(J_\ell, J_{\ell'}),\qquad\ell,\ell'\neq m,\\
    \Vs{Z_{N,1}^{(m)}}{F_{m,N}}&=\dfrac{1}{\pi_{m,N}}\sum_{\substack{j=1\\j\neq m}}^{c}\pi_{j,N}^2A_{jj,N}^{(m)}(J_m,J_m)+\dfrac{1}{\pi_{m,N}}\sum_{\substack{j=1\\j\neq m}}^{c}\sum_{\substack{k=1\\k\neq j\\k\neq m}}^{c}\pi_{j,N}\pi_{k,N}A_{jk,N}^{(m)}(J_m,J_m).
\end{align*}
\end{lemma}

\begin{proof}[Proof of Lemma~\ref{lemma:bSigma}]
    For all $j=1,\ldots,c$, let $B_{j,m,N}(x):=\int_{x_{0,N}}^{x}J_{m}'(H_N(y))\,\rd F_{j,N}(y)$,
    so that $B_{m,m,N}=B_{m,N}$.
    Note that
    \begin{align*}
        J_{m}(H_N(x))&=J_{m}(H_N(x_{0,N}))+\int_{x_{0,N}}^{x}J_{m}'(H_N(y))\,\rd H_N(y)\\
        &=J_m(H_N(x_{0,N}))+\sum_{j=1}^{c}\pi_{j,N}B_{j,m,N}(x),
    \end{align*}
    and therefore
    \begin{align*}
        C_{m,N}(x):=J_m(H_N(x))-\pi_{m,N}B_{m,N}(x)=J_m(H_N(x_{0,N}))+\sum_{\substack{j=1\\j\neq m}}^{c}\pi_{j,N}B_{j,m,N}(x).
    \end{align*}
    More generally, for every $j,m\in\{1,\ldots,c\}$, 
    \begin{align*}
        B_{j,m,N}(X_{N,1}^{(m)})-\Es{B_{j,m,N}(X_{N,1}^{(m)})}{F_{m,N}}
        =\int_{-\infty}^{+\infty}\lrpBig{\mathds{1}_{\{y\leq X_{N,1}^{(m)}\}}-(1-F_{m,N}(y))}J_{m}'(H_N(y))\,\rd F_{j,N}(y)
    \end{align*}
    because
    \begin{align*}
        \int_{x_{0,N}}^{X_{N,1}^{(m)}}J_{m}'(H_N(y))\,\rd F_{j,N}(y)
        =\int_{-\infty}^{+\infty}\lrpBig{\mathds{1}_{\{y\leq X_{N,1}^{(m)}\}}-\mathds{1}_{\{y\leq x_{0,N}\}}}J_{m}'(H_N(y))\,\rd F_{j,N}(y)
    \end{align*}
    and the deterministic term $\mathds{1}_{\{y\leq x_{0,N}\}}$ cancels after centering.

    The first mixed term, for $m\neq\ell$, is
    \begin{align*}
        &\Es{Z_{N,1}^{(m)}Y_{N,1}^{(m,\ell)}}{F_{m,N}}\\
        &=\Es{\lrp{C_{m,N}(X_{N,1}^{(m)}) - \Es{C_{m,N}(X_{N,1}^{(m)})}{F_{m,N}}}\lrp{B_{\ell,N}(X_{N,1}^{(m)})-\Es{B_{\ell,N}(X_{N,1}^{(m)})}{F_{m,N}}}}{F_{m,N}}\\
        &=\mathrm{E}_{F_{m,N}}\Bigg[\lrpBigg{\sum_{\substack{j=1\\j\neq m}}^{c}\pi_{j,N}\int_{-\infty}^{+\infty}\lrpBig{\mathds{1}_{\{x\leq X_{N,1}^{(m)}\}}-(1-F_{m,N}(x))}J_{m}'(H_N(x))\,\rd F_{j,N}(x)}\\
        &\qquad\qquad\quad\times\lrp{\int_{-\infty}^{+\infty}\lrpBig{\mathds{1}_{\{y\leq X_{N,1}^{(m)}\}}-(1-F_{m,N}(y))}J_{\ell}'(H_N(y))\,\rd F_{\ell,N}(y)}\Bigg]\\
        &=\sum_{\substack{j=1\\j\neq m}}^{c}\pi_{j,N}\int_{-\infty}^{+\infty}\int_{-\infty}^{+\infty}\mathrm{E}_{F_{m,N}}\Big[\lrpBig{\mathds{1}_{\{x\leq X_{N,1}^{(m)}\}}-(1-F_{m,N}(x))}\lrpBig{\mathds{1}_{\{y\leq X_{N,1}^{(m)}\}}-(1-F_{m,N}(y))}\Big]\\
        &\qquad\qquad\quad\times J_{m}'(H_N(x))J_{\ell}'(H_N(y))\,\rd F_{j,N}(x)\,\rd F_{\ell,N}(y)\\
        &=\sum_{\substack{j=1\\j\neq m}}^{c}\pi_{j,N}\Bigg(\int_{-\infty}^{+\infty}\int_{-\infty}^{+\infty}(1-F_{m,N}(x\vee y))J_{m}'(H_N(x))J_{\ell}'(H_N(y))\,\rd F_{j,N}(x)\,\rd F_{\ell,N}(y)\\
        &\qquad\qquad\quad-\int_{-\infty}^{+\infty}\int_{-\infty}^{+\infty}(1-F_{m,N}(x))(1-F_{m,N}(y))J_{m}'(H_N(x))J_{\ell}'(H_N(y))\,\rd F_{j,N}(x)\,\rd F_{\ell,N}(y)\Bigg)\\
        &=\sum_{\substack{j=1\\j\neq m}}^{c}\pi_{j,N}A_{j\ell,N}^{(m)}(J_m,J_{\ell}),
    \end{align*}
    where Fubini's theorem justifies interchanging expectation and integration.

    For $m\neq \ell$ and $m\neq\ell'$, we have
    \begin{align*}
        &\Es{Y_{N,1}^{(m,\ell)}Y_{N,1}^{(m,\ell')}}{F_{m,N}}\\
        &=\pi_{m,N}\Es{\lrp{B_{\ell,N}(X_{N,1}^{(m)})-\Es{B_{\ell,N}(X_{N,1}^{(m)})}{F_{m,N}}}\lrp{B_{\ell',N}(X_{N,1}^{(m)})-\Es{B_{\ell',N}(X_{N,1}^{(m)})}{F_{m,N}}}}{F_{m,N}}\\
        &=\pi_{m,N}\mathrm{E}_{F_{m,N}}\bigg[\lrp{\int_{-\infty}^{+\infty}\lrpBig{\mathds{1}_{\{x\leq X_{N,1}^{(m)}\}}-(1-F_{m,N}(x))}J_{\ell}'(H_N(x))\,\rd F_{\ell,N}(x)}\\
        &\qquad\qquad\quad\times\lrp{\int_{-\infty}^{+\infty}\lrpBig{\mathds{1}_{\{y\leq X_{N,1}^{(m)}\}}-(1-F_{m,N}(y))}J_{\ell'}'(H_N(y))\,\rd F_{\ell',N}(y)}\bigg]\\
        &=\pi_{m,N}\int_{-\infty}^{+\infty}\int_{-\infty}^{+\infty}\mathrm{E}_{F_{m,N}}\Big[\lrpBig{\mathds{1}_{\{x\leq X_{N,1}^{(m)}\}}-(1-F_{m,N}(x))}\lrpBig{\mathds{1}_{\{y\leq X_{N,1}^{(m)}\}}-(1-F_{m,N}(y))}\Big]\\
        &\qquad\qquad\quad\times J_{\ell}'(H_N(x))J_{\ell'}'(H_N(y))\,\rd F_{\ell,N}(x)\,\rd F_{\ell',N}(y)\\
        &=\pi_{m,N}\Bigg(\int_{-\infty}^{+\infty}\int_{-\infty}^{+\infty}(1-F_{m,N}(x\vee y))J_{\ell}'(H_N(x))J_{\ell'}'(H_N(y))\,\rd F_{\ell,N}(x)\,\rd F_{\ell',N}(y)\\
        &\qquad\qquad\quad-\int_{-\infty}^{+\infty}\int_{-\infty}^{+\infty}(1-F_{m,N}(x))(1-F_{m,N}(y))J_{\ell}'(H_N(x))J_{\ell'}'(H_N(y))\,\rd F_{\ell,N}(x)\,\rd F_{\ell',N}(y)\Bigg)\\
        &=\pi_{m,N}A_{\ell\ell',N}^{(m)}(J_\ell, J_{\ell'}).
    \end{align*}

    Finally,
    \begin{align*}
        &\mathrm{V}\mathrm{ar}_{F_{m,N}}\lrc{Z_{N,1}^{(m)}}
        =\dfrac{1}{\pi_{m,N}}\mathrm{V}\mathrm{ar}_{F_{m,N}}\lrc{C_{m,N}(X_{N,1}^{(m)})}\\
        &=\dfrac{1}{\pi_{m,N}}\EsBigg{\lrpBigg{\sum_{\substack{j=1\\j\neq m}}^{c}\pi_{j,N}B_{j,m,N}(X_{N,1}^{(m)})-\EsBigg{\sum_{\substack{j=1\\j\neq m}}^{c}\pi_{j,N}B_{j,m,N}(X_{N,1}^{(m)})}{F_{m,N}}}^2}{F_{m,N}}\\
        &=\dfrac{1}{\pi_{m,N}}\mathrm{E}_{F_{m,N}}\Bigg[\Bigg(\sum_{\substack{j=1\\j\neq m}}^{c}\pi_{j,N}\int_{-\infty}^{+\infty}\lrpBig{\mathds{1}_{\{x\leq X_{N,1}^{(m)}\}}-(1-F_{m,N}(x))}J_{m}'(H_N(x))\,\rd F_{j,N}(x)\Bigg)\\
        &\qquad\qquad\quad\times\Bigg(\sum_{\substack{k=1\\k\neq m}}^{c}\pi_{k,N}\int_{-\infty}^{+\infty}\lrpBig{\mathds{1}_{\{y\leq X_{N,1}^{(m)}\}}-(1-F_{m,N}(y))}J_{m}'(H_N(y))\,\rd F_{k,N}(y)\Bigg)\Bigg]\\
        &=\dfrac{1}{\pi_{m,N}}\sum_{\substack{j=1\\j\neq m}}^{c}\sum_{\substack{k=1\\k\neq m}}^{c}\pi_{j,N}\pi_{k,N}\Bigg(\int_{-\infty}^{+\infty}\int_{-\infty}^{+\infty}(1-F_{m,N}(x\vee y))J_m'(H_N(x))J_m'(H_N(y))\,\rd F_{j,N}(x)\,\rd F_{k,N}(y)\\
        &\qquad\qquad\quad-\int_{-\infty}^{+\infty}\int_{-\infty}^{+\infty}(1-F_{m,N}(x))(1-F_{m,N}(y))J_m'(H_N(x))J_m'(H_N(y))\,\rd F_{j,N}(x)\,\rd F_{k,N}(y)\Bigg)\\
        &=\dfrac{1}{\pi_{m,N}}\sum_{\substack{j=1\\j\neq m}}^{c}\sum_{\substack{k=1\\k\neq m}}^{c}\pi_{j,N}\pi_{k,N}A_{jk,N}^{(m)}(J_m,J_m)\\
        &=\dfrac{1}{\pi_{m,N}}\sum_{\substack{j=1\\j\neq m}}^{c}\pi_{j,N}^2A_{jj,N}^{(m)}(J_m,J_m)
        +\dfrac{1}{\pi_{m,N}}\sum_{\substack{j=1\\j\neq m}}^{c}\sum_{\substack{k=1\\k\neq j\\k\neq m}}^{c}\pi_{j,N}\pi_{k,N}A_{jk,N}^{(m)}(J_m,J_m).\tag*{\qedhere}
    \end{align*}
\end{proof}

Lemma~\ref{lemma:exp-G2} gives the exact null second moments of the empirical harmonic averages and is used to control the truncation tail in Theorem~\ref{thm:asymp-H0-general}.

\begin{lemma}\label{lemma:exp-G2}
    Let $k\geq1$, $r=1,2$, $\ell=1,\ldots,c$, and $G_{k,r,\ell,N}={n_{\ell}}^{-1}\sum_{i=1}^{n_\ell}g_{k,r}(c_i^{(\ell)})$. Let $n_\ell\ge1$ for all $\ell=1,\ldots,c$. Then, under~$\Hcal_0$,
    \begin{align*}
            2^{-1}\sum_{r=1}^{2}\Es{n_{\ell}G_{k,r,\ell,N}^2}{0}&=\begin{cases}
            n_{\ell},&k\in N\Z_{+},\\
            1 - \dfrac{n_{\ell}-1}{N-1},&\text{ otherwise.}
        \end{cases}
    \end{align*}
    Moreover, $G_{k,r,\ell,N}^2=2\delta_{r1}$ almost surely for all $k\in N\Z_{+}$.
\end{lemma}

\begin{proof}[Proof of Lemma~\ref{lemma:exp-G2}]
    Under $\Hcal_0$, the sample $\Theta_1,\ldots,\Theta_N$ is iid, thus the joint distribution of \linebreak$(\Theta_1,\ldots,\Theta_N)$ is invariant under permutations, i.e., for every permutation $\pi$, $(\Theta_1,\ldots,\Theta_N)\equald \break(\Theta_{\pi(1)},\ldots,\Theta_{\pi(N)})$. Since the common distribution is continuous, ties have zero probability, and the events $A_\pi=\{\Theta_{\pi(1)}<\cdots<\Theta_{\pi(N)}\}$ form a partition of the sample space up to a set with zero probability. Therefore, permutation invariance implies that $\mathrm{P}(A_\pi)={1}/{N!}$. Let $\ell=1,\ldots,c$ and $i=1,\ldots,n_{\ell}$. Then, the event $\{r_i^{(\ell)}=s\}=\bigcup_{\pi:\pi(s)=i}A_\pi$, and $\mathrm{P}(\bigcup_{\pi:\pi(s)=i}A_\pi)=\#\{\pi:\pi(s)=i\}/N!$. Since $(N-1)!$ is the number of permutations in which $\Theta_i^{(\ell)}$ is the $s$th sorted observation, it follows that $\mathrm{P}(r_i^{(\ell)}=s)=1/N$ for all $s=1,\ldots,N$. Analogously, we have that for $1\leq i\neq j\leq n_\ell$, $\mathrm{P}(r_i^{(\ell)}=s, r_j^{(\ell)}=m)=1/(N(N-1))$ if $s\neq m$, and zero otherwise.
    
    Let $k\geq1$ be such that $k\notin N\Z$. Then, for $i=1,\ldots,n_{\ell}$,
    \begin{align*}
        \Es{\cos^2(2\pi k r_i^{(\ell)}/N)}{0}&=\dfrac{1}{N}\sum_{j = 1}^{N}\cos^2(2\pi k j / N)=\dfrac{1}{2} + \dfrac{1}{N}\sum_{j=1}^{N} \cos(4\pi k j/N)/2=\begin{cases}
            1/2, &2k\notin N\Z,\\
            1, &2k\in N\Z,
        \end{cases}
    \end{align*}
    and, analogously,
    \begin{align*}
        \Es{\sin^2(2\pi k r_i^{(\ell)}/N)}{0}&=
        \begin{cases}
            1/2, &2k\notin N\Z,\\
            0, &2k\in N\Z.
        \end{cases}
    \end{align*}
    For $1\leq i\neq j\leq n_\ell$,
    \begin{align*}
        \Es{\cos(2\pi k r_i^{(\ell)}/N)\cos(2\pi k r_j^{(\ell)}/N)}{0}
        &=\dfrac{1}{N(N-1)}\sum_{1\leq i_1\neq i_2\leq N}\cos(2\pi k i_1/N)\cos(2\pi k i_2/N)\\
        &=\dfrac{1}{N(N-1)}\bigg(\Big(\sum_{i_1=1}^{N}\cos(2\pi k i_1/N)\Big)^2 - \sum_{i_1=1}^{N}\cos^2(2\pi k i_1/N)\bigg)\\
        &=\begin{cases}
            -1/(2(N-1)), &2k\notin N\Z,\\
            -1/(N-1), &2k\in N\Z,
        \end{cases}
    \end{align*}
    and, analogously,
    \begin{align*}
        \Es{\sin(2\pi k r_i^{(\ell)}/N)\sin(2\pi k r_j^{(\ell)}/N)}{0}
        &=\begin{cases}
            -1/(2(N-1)), &2k\notin N\Z,\\
            0, &2k\in N\Z.
        \end{cases}
    \end{align*}

    Then,
    \begin{align*}
        \sum_{r=1}^{2}\Es{n_{\ell}G_{k,r,\ell,N}^2}{0}&={n_{\ell}}^{-1}\sum_{r=1}^{2}\sum_{i,j=1}^{n_\ell}\Es{g_{k,r}(c_i^{(\ell)})g_{k,r}(c_j^{(\ell)})}{0}\\
        &=2{n_{\ell}}^{-1}\sum_{1\leq i<j\leq n_\ell} \sum_{r=1}^{2}\Es{g_{k,r}(c_i^{(\ell)})g_{k,r}(c_j^{(\ell)})}{0} + {n_{\ell}}^{-1}\sum_{i=1}^{n_\ell}\sum_{r=1}^{2}\Es{g_{k,r}^2(c_i^{(\ell)})}{0}\\
        &=2\Big(1 - \dfrac{n_{\ell}-1}{N-1}\Big).
    \end{align*}
    
    For $k\in N\Z_{+}$, the result is straightforward.
\end{proof}

Lemma~\ref{lemma:consistency} establishes the almost-sure uniform replacement of the pooled empirical cdf by its limit inside the kernel, as required for the fixed-alternative limit in
Theorem~\ref{thm:asymp-Halt-general} and the corresponding consistency result. 

\begin{lemma}\label{lemma:consistency}
    Let $\psi:[-1,1]\to\R$ be a continuous function. Let $H$ be as defined in~\eqref{eq:H_limit} and $\widehat{H}_N$ be the ecdf of the pooled sample. Let $\ell=1,\ldots,c$ and $\{\Theta_i^{(\ell)}\}_{i=1}^{n_\ell}$ be the observations corresponding to the $\ell$th sample. Then, under the fixed-distribution setting described in Section~\ref{subsec:nonnull:consistency}, as $N\to\infty$,
    $$\sup_{1\leq i,j\leq n_\ell}\bigg|\psi (\widehat{H}_N(\Theta_i^{(\ell)}) - \widehat{H}_N(\Theta_j^{(\ell)})) - \psi(H(\Theta_i^{(\ell)})-H(\Theta_j^{(\ell)}))\bigg|\to 0\quad\text{almost surely.}$$
\end{lemma}

\begin{proof}[Proof of Lemma~\ref{lemma:consistency}]
    Let $\omega_\psi:[0,+\infty)\to [0,+\infty)$ be the modulus of continuity of $\psi$, given by
    $$
    \omega_{\psi}(\delta):=\sup_{\substack{u,v\in [-1,1]\\|u-v|\leq\delta}}|\psi(u)-\psi(v)|,\qquad\delta\geq0.
    $$
    Since $\psi$ is continuous on the compact set $[-1,1]$, it is uniformly continuous, and $\omega_\psi(\delta)<+\infty$ for all $\delta\geq0$. Moreover, $\omega_\psi$ is non-decreasing, and $\omega_{\psi}(\delta)\to\omega_{\psi}(0)=0$ as $\delta \downarrow0$.
    
    Fix $\ell=1,\ldots,c$ and $1\leq i,j\leq n_\ell$. Then, using the triangle inequality and the monotonicity of $\omega_{\psi}$, it follows that
    \begin{align*}
        \Big|\psi (\widehat{H}_N(\Theta_i^{(\ell)}) - \widehat{H}_N(\Theta_j^{(\ell)})) &- \psi(H(\Theta_i^{(\ell)})-H(\Theta_j^{(\ell)}))\Big|\\
        &\leq \omega_{\psi}\Big(\Big|\widehat{H}_N(\Theta_i^{(\ell)}) - \widehat{H}_N(\Theta_j^{(\ell)}) - H(\Theta_i^{(\ell)})+H(\Theta_j^{(\ell)})\Big|\Big)\\
        &\leq \omega_{\psi}\Big(\Big|\widehat{H}_N(\Theta_i^{(\ell)}) - H(\Theta_i^{(\ell)})\Big|+\Big|\widehat{H}_N(\Theta_j^{(\ell)}) - H(\Theta_j^{(\ell)})\Big|\Big).
    \end{align*}

    Thus, for all $\ell=1,\ldots,c$, it holds that 
    \begin{align*}
        \sup_{1\leq i,j\leq n_\ell}&\Big|\psi (\widehat{H}_N(\Theta_i^{(\ell)}) - \widehat{H}_N(\Theta_j^{(\ell)})) - \psi(H(\Theta_i^{(\ell)})-H(\Theta_j^{(\ell)}))\Big|\\
        &\leq \sup_{1\leq i,j\leq n_\ell}\omega_{\psi}\Big(\Big|\widehat{H}_N(\Theta_i^{(\ell)}) - H(\Theta_i^{(\ell)})\Big|+\Big|\widehat{H}_N(\Theta_j^{(\ell)}) - H(\Theta_j^{(\ell)})\Big|\Big)\\
        &\leq \omega_{\psi}\Big(\sup_{1\leq i,j\leq n_\ell}\left\{\Big|\widehat{H}_N(\Theta_i^{(\ell)}) - H(\Theta_i^{(\ell)})\Big|+\Big|\widehat{H}_N(\Theta_j^{(\ell)}) - H(\Theta_j^{(\ell)})\Big|\right\}\Big)\\
        &= \omega_{\psi}\Big(2\sup_{1\leq i\leq n_\ell}\Big|\widehat{H}_N(\Theta_i^{(\ell)}) - H(\Theta_i^{(\ell)})\Big|\Big)\\
        &\leq\omega_{\psi}\Big(2\sup_{x\in[0,2\pi)}\Big|\widehat{H}_N(x) - H(x)\Big|\Big),
    \end{align*}
    where we have leveraged the monotonicity of $\omega_\psi$.

    By the Glivenko--Cantelli theorem and since $\pi_{\ell,N}\to\pi_\ell$, $\sup_{x\in[0,2\pi)}\big|\widehat{H}_N(x) - H(x)\big|\to 0$ almost surely. Since $\omega_\psi$ is continuous at zero, it follows that
    \begin{align*}
        \sup_{1\leq i,j\leq n_\ell}\Big|\psi (\widehat{H}_N(\Theta_i^{(\ell)}) - \widehat{H}_N(\Theta_j^{(\ell)})) - \psi(H(\Theta_i^{(\ell)})-H(\Theta_j^{(\ell)}))\Big|\leq\omega_{\psi}\Big(2\sup_{x\in[0,2\pi)}\Big|\widehat{H}_N(x) - H(x)\Big|\Big)\to 0
    \end{align*} almost surely, and the result follows.
\end{proof}

Finally, Lemma~\ref{lemma:unif-taylor} provides uniform first-order expansions for a shifted density and its corresponding cdf. These expansions are used in Proposition~\ref{prp:gkr-asymp-Hlocalt}.

\begin{lemma}\label{lemma:unif-taylor}
    Let $F:[0,2\pi]\to[0,1]$ be an absolutely continuous cdf with density $f$, and assume that the periodic extension of $f$ belongs to $C^1(\R)$. For $x\in[0,2\pi]$ and $\Delta\in\R$, write
    \begin{align*}
        f(x+\Delta)=f(x)+f'(x)\Delta+R_2(x,\Delta).
    \end{align*}
    Then,
    \begin{align}\label{eq:unif-remaind}
        \sup_{x\in[0,2\pi]}|R_2(x,\Delta)|=o(|\Delta|)\quad\text{as}\quad\Delta\to0.
    \end{align}
    Moreover, for every $\theta\in[0,2\pi]$,
    \begin{align*}
        \int_{0}^{\theta} f(x+\Delta)\,\rd x
        =F(\theta)+\Delta\bigl(f(\theta)-f(0)\bigr)+R_{2}^{\ast}(\theta,\Delta),
    \end{align*}
    where $R_{2}^{\ast}(\theta,\Delta):=\int_{0}^{\theta}R_2(x,\Delta)\,\rd x$ satisfies
    \begin{align*}
        \sup_{\theta\in[0,2\pi]}|R_{2}^{\ast}(\theta,\Delta)|=o(|\Delta|)\quad\text{as}\quad\Delta\to0.
    \end{align*}
\end{lemma}

\begin{proof}[Proof of Lemma~\ref{lemma:unif-taylor}]
    For $x\in[0,2\pi]$ and $\Delta\in\R$,
    \begin{align*}
        R_2(x,\Delta)
        &=f(x+\Delta)-f(x)-f'(x)\Delta
        =\Delta\int_{0}^{1}\bigl(f'(x+t\Delta)-f'(x)\bigr)\,\rd t.
    \end{align*}

    Since the periodic extension of $f'$ is continuous and $2\pi$-periodic, it is uniformly continuous on $\R$. Let $\epsilon>0$. Then there exists $\delta>0$ such that
    \begin{align*}
        |u-v|<\delta
        \quad\Longrightarrow\quad
        |f'(u)-f'(v)|<\epsilon
    \end{align*}
    for all $u,v\in\R$. Hence, if $0<|\Delta|<\delta$, then for every $x\in[0,2\pi]$,
    \begin{align*}
        \frac{|R_2(x,\Delta)|}{|\Delta|}
        &\leq \int_{0}^{1}|f'(x+t\Delta)-f'(x)|\,\rd t
        <\epsilon.
    \end{align*}
    Taking the supremum over $x\in[0,2\pi]$ proves \eqref{eq:unif-remaind}.

    Integrating the Taylor expansion over $[0,\theta]$, we obtain
    \begin{align*}
        \int_{0}^{\theta} f(x+\Delta)\,\rd x
        &=\int_{0}^{\theta}f(x)\,\rd x
        +\Delta\int_{0}^{\theta}f'(x)\,\rd x
        +\int_{0}^{\theta}R_2(x,\Delta)\,\rd x\\
        &=F(\theta)+\Delta\bigl(f(\theta)-f(0)\bigr)+R_{2}^{\ast}(\theta,\Delta).
    \end{align*}

    Let $\epsilon>0$. By \eqref{eq:unif-remaind}, there exists $\delta>0$ such that, for all $0<|\Delta|<\delta$,
    \begin{align*}
        \sup_{x\in[0,2\pi]}\frac{|R_2(x,\Delta)|}{|\Delta|}<\frac{\epsilon}{2\pi}.
    \end{align*}
    Therefore, for all $\theta\in[0,2\pi]$,
    \begin{align*}
        \frac{|R_{2}^{\ast}(\theta,\Delta)|}{|\Delta|}
        &\leq \int_{0}^{\theta}\frac{|R_2(x,\Delta)|}{|\Delta|}\,\rd x
        \leq 2\pi \sup_{x\in[0,2\pi]}\frac{|R_2(x,\Delta)|}{|\Delta|}
        <\epsilon.
    \end{align*}
    This proves that $R_{2}^{\ast}(\theta,\Delta)=o(|\Delta|)$ uniformly in $\theta\in[0,2\pi]$.
\end{proof}

\section{Additional numerical results}\label{sec:sims}

\subsection{Imbalanced samples}\label{subsec:sims:imbalanced}

In this section, we extend the experiments of Section~\ref{subsec:sim:pow} to study how sample-size imbalance affects the power of the average, maximum, and ${\rm LSE}$ tests. We consider the same one-different-location alternatives, but vary the size of the shifted subsample relative to the unshifted samples, defining $r$ as the ratio of the shifted sample size to each unshifted sample size. Within each scenario, all subsamples are generated from the same parametric family, chosen from \hyperref[fsim1]{(vM)}, \hyperref[fsim2]{(W)}, \hyperref[fsim3]{(SC)}, and~\hyperref[fsim4]{(MvM)}. 

\begin{table}[!b]
\centering
\scalebox{0.6}{
    \begin{tabular}{ccc|rrrr|rrrr}
    \toprule
    \multirow{2}{*}{Family} & \multirow{2}{*}{$c$} & \multirow{2}{*}{$r$} & \multicolumn{4}{c|}{$\rm AD$} & \multicolumn{4}{c}{${\rm SM}_{10}$} \\
    & & & Avg & ${\rm LSE}_{1}$ & ${\rm LSE}_{10}$ & Max & Avg & ${\rm LSE}_{1}$ & ${\rm LSE}_{10}$ & Max \\
    \midrule
    & & $1/2$ & 43.5 & 50.4 & $\mathbf{55.5}$ & 54.6 & 39.6 & 46.3 & 54.5 & $\mathbf{56.2}$ \\ 
    &  & $1$ & 61.2 & 60.3 & $\mathbf{62.9}$ & 60.1 & 56.6 & 57.1 & $\mathbf{59.2}$ & $\mathbf{58.7}$ \\ 
    & \multirow{-3}{*}{3} & $2$ & $\mathbf{67.2}$ & 64.4 & 55.1 & 54.3 & $\mathbf{62.3}$ & $\mathbf{62.8}$ & 51.0 & 51.3 \\ 
    \cline{2-11}
    & & $1/2$ & 33.7 & 41.6 & 43.0 & $\mathbf{43.9}$ & 30.8 & 38.0 & $\mathbf{44.0}$ & $\mathbf{44.1}$ \\ 
    &  & $1$ & 56.5 & $\mathbf{62.2}$ & $\mathbf{62.5}$ & 61.7 & 51.4 & 59.3 & $\mathbf{61.6}$ & $\mathbf{61.9}$ \\ 
    & \multirow{-3}{*}{5} & $2$ & $\mathbf{74.6}$ & $\mathbf{74.2}$ & 70.8 & 71.0 & 69.7 & $\mathbf{72.4}$ & 70.4 & 69.9 \\ 
    \cline{2-11}
    & & $1/2$ & 22.0 & 30.2 & 30.2 & $\mathbf{31.2}$ & 20.3 & 28.7 & 32.1 & $\mathbf{34.0}$ \\ 
    &  & $1$ & 45.2 & 59.7 & $\mathbf{60.8}$ & 59.2 & 40.7 & 57.7 & $\mathbf{61.8}$ & $\mathbf{61.7}$ \\ 
    \multirow{-9}{*}{$\substack{{\rm vM}\\\\\centering\includegraphics[width=1.75cm, clip=true, trim={1.6cm 1.5cm 1.75cm 0.5cm}]{img/alt-shift-vMF.pdf}}$} & \multirow{-3}{*}{10} & $2$ & 73.3 & 84.3 & $\mathbf{85.3}$ & $\mathbf{85.3}$ & 68.0 & 83.7 & $\mathbf{86.5}$ & $\mathbf{87.1}$ \\ 
    \hline
    & & $1/2$ & 22.3 & 26.6 & $\mathbf{32.1}$ & $\mathbf{31.4}$ & 53.3 & 61.4 & 70.2 & $\mathbf{72.5}$ \\ 
    &  & $1$ & $\mathbf{34.8}$ & 32.8 & 33.5 & 31.1 & 73.0 & 75.2 & $\mathbf{77.3}$ & $\mathbf{77.9}$ \\ 
    & \multirow{-3}{*}{3} & $2$ & $\mathbf{40.2}$ & 35.4 & 22.6 & 22.2 & $\mathbf{79.3}$ & $\mathbf{79.0}$ & 68.1 & 68.5 \\ 
    \cline{2-11}
    & & $1/2$ & 17.2 & $\mathbf{19.4}$ & 18.1 & $\mathbf{19.0}$ & 40.8 & 53.0 & $\mathbf{59.8}$ & $\mathbf{59.9}$ \\ 
    &  & $1$ & 30.8 & $\mathbf{32.7}$ & 30.1 & 28.9 & 69.3 & 77.7 & $\mathbf{80.5}$ & $\mathbf{80.9}$ \\ 
    & \multirow{-3}{*}{5} & $2$ & $\mathbf{46.9}$ & 40.5 & 34.9 & 33.4 & 86.5 & $\mathbf{90.0}$ & $\mathbf{90.0}$ & 88.6 \\ 
    \cline{2-11}
    & & $1/2$ & $\mathbf{12.2}$ & $\mathbf{12.6}$ & 10.5 & 11.3 & 26.7 & 39.1 & 45.4 & $\mathbf{47.8}$ \\ 
    &  & $1$ & 23.0 & $\mathbf{27.6}$ & $\mathbf{27.0}$ & 26.1 & 54.8 & 76.6 & 82.5 & $\mathbf{83.4}$ \\ 
    \multirow{-9}{*}{$\substack{{\rm W}\\\\\centering\includegraphics[width=1.75cm, clip=true, trim={1.6cm 1.5cm 1.75cm 0.5cm}]{img/alt-shift-W.pdf}}$} & \multirow{-3}{*}{10} & $2$ & 44.1 & 55.2 & $\mathbf{56.8}$ & $\mathbf{57.6}$ & 86.2 & 96.7 & $\mathbf{98.1}$ & $\mathbf{98.1}$ \\ 
    \hline
    & & $1/2$ & 17.7 & 20.0 & $\mathbf{25.5}$ & 24.1 & 41.2 & 46.3 & 58.2 & $\mathbf{59.6}$ \\ 
    &  & $1$ & $\mathbf{26.8}$ & 24.4 & 25.3 & 23.3 & 58.9 & 61.2 & 61.3 & $\mathbf{63.4}$ \\ 
    & \multirow{-3}{*}{3} & $2$ & $\mathbf{30.6}$ & 26.5 & 17.9 & 18.0 & $\mathbf{65.2}$ & $\mathbf{65.7}$ & 53.2 & 53.3 \\ 
    \cline{2-11}
    & & $1/2$ & 14.2 & $\mathbf{15.3}$ & 13.5 & $\mathbf{15.0}$ & 30.3 & 39.3 & $\mathbf{45.9}$ & $\mathbf{46.0}$ \\ 
    &  & $1$ & 22.5 & $\mathbf{23.4}$ & 21.3 & 19.2 & 53.1 & 62.4 & $\mathbf{65.8}$ & $\mathbf{65.1}$ \\ 
    & \multirow{-3}{*}{5} & $2$ & $\mathbf{35.0}$ & 28.8 & 22.7 & 21.7 & 73.4 & $\mathbf{77.3}$ & 75.3 & 73.9 \\ 
    \cline{2-11}
    & & $1/2$ & $\mathbf{10.7}$ & $\mathbf{10.5}$ & 9.3 & 9.2 & 20.4 & 28.8 & 33.0 & $\mathbf{33.9}$ \\ 
    &  & $1$ & 18.1 & $\mathbf{19.0}$ & 17.7 & 17.5 & 41.6 & 60.6 & 65.2 & $\mathbf{66.3}$ \\ 
    \multirow{-9}{*}{$\substack{{\rm SC}\\\\\centering\includegraphics[width=1.75cm, clip=true, trim={1.6cm 1.5cm 1.75cm 0.5cm}]{img/alt-shift-SC.pdf}}$} & \multirow{-3}{*}{10} & $2$ & 33.8 & $\mathbf{36.3}$ & $\mathbf{37.0}$ & $\mathbf{36.9}$ & 72.1 & 87.0 & $\mathbf{90.9}$ & $\mathbf{91.3}$ \\ 
    \hline
    & & $1/2$ & 7.3 & 7.4 & $\mathbf{8.2}$ & $\mathbf{8.0}$ & 14.8 & 16.7 & 21.4 & $\mathbf{23.0}$ \\ 
    &  & $1$ & $\mathbf{8.2}$ & 7.7 & $\mathbf{8.0}$ & 7.0 & $\mathbf{21.0}$ & 20.2 & 19.9 & $\mathbf{20.5}$ \\ 
    & \multirow{-3}{*}{3} & $2$ & $\mathbf{9.0}$ & $\mathbf{8.5}$ & 7.7 & 8.1 & $\mathbf{23.1}$ & $\mathbf{23.5}$ & 17.6 & 18.4 \\ 
    \cline{2-11}
    & & $1/2$ & $\mathbf{6.7}$ & $\mathbf{6.7}$ & 5.7 & $\mathbf{6.3}$ & 11.5 & 12.9 & $\mathbf{14.5}$ & $\mathbf{14.6}$ \\ 
    &  & $1$ & $\mathbf{8.4}$ & 6.7 & 6.2 & 5.3 & 18.0 & $\mathbf{19.3}$ & 18.0 & 16.9 \\ 
    & \multirow{-3}{*}{5} & $2$ & $\mathbf{9.4}$ & 7.1 & 6.5 & 6.1 & $\mathbf{26.6}$ & 25.2 & 20.1 & 20.1 \\ 
    \cline{2-11}
    & & $1/2$ & $\mathbf{6.3}$ & 5.4 & 5.2 & 5.7 & 9.2 & 9.7 & 9.2 & $\mathbf{10.2}$ \\ 
    &  & $1$ & $\mathbf{7.6}$ & 5.9 & 6.0 & 5.6 & 14.2 & $\mathbf{16.5}$ & 16.4 & $\mathbf{17.1}$ \\ 
    \multirow{-9}{*}{$\substack{{\rm MvM (3)}\\\\\centering\includegraphics[width=1.75cm, clip=true, trim={1.6cm 1.5cm 1.75cm 0.5cm}]{img/alt-shift-MvMF3.pdf}}$} & \multirow{-3}{*}{10} & $2$ & $\mathbf{9.8}$ & 6.4 & 6.0 & 6.1 & 24.0 & $\mathbf{28.2}$ & $\mathbf{28.3}$ & $\mathbf{28.4}$ \\ 
    \hline
    & & $1/2$ & 5.8 & 5.8 & $\mathbf{6.7}$ & 6.2 & 7.1 & 7.0 & 9.1 & $\mathbf{10.1}$ \\ 
    &  & $1$ & 5.4 & 5.6 & $\mathbf{6.0}$ & 5.4 & $\mathbf{8.2}$ & 7.7 & 7.1 & 7.3 \\ 
    & \multirow{-3}{*}{3} & $2$ & 5.9 & 6.1 & $\mathbf{7.0}$ & 6.5 & 8.2 & 8.8 & 8.9 & $\mathbf{9.7}$ \\ 
    \cline{2-11}
    & & $1/2$ & 5.2 & $\mathbf{5.6}$ & 4.9 & 5.3 & 6.1 & $\mathbf{6.3}$ & $\mathbf{6.4}$ & $\mathbf{6.5}$ \\ 
    &  & $1$ & $\mathbf{5.8}$ & $\mathbf{5.7}$ & 5.3 & 4.6 & $\mathbf{7.6}$ & 7.0 & 6.2 & 5.7 \\ 
    & \multirow{-3}{*}{5} & $2$ & $\mathbf{5.8}$ & $\mathbf{5.4}$ & 5.3 & 4.8 & $\mathbf{9.4}$ & 8.3 & 7.5 & 6.7 \\ 
    \cline{2-11}
    & & $1/2$ & $\mathbf{5.6}$ & 4.8 & 4.5 & 4.8 & $\mathbf{6.2}$ & 5.6 & 5.5 & $\mathbf{5.8}$ \\ 
    &  & $1$ & $\mathbf{5.6}$ & 4.9 & 5.2 & 4.4 & $\mathbf{7.3}$ & $\mathbf{7.0}$ & 5.6 & 5.8 \\ 
    \multirow{-9}{*}{$\substack{{\rm MvM (4)}\\\\\centering\includegraphics[width=1.75cm, clip=true, trim={1.6cm 1.5cm 1.75cm 0.5cm}]{img/alt-shift-MvMF.pdf}}$} & \multirow{-3}{*}{10} & $2$ & $\mathbf{6.1}$ & 4.7 & 4.9 & 4.8 & $\mathbf{9.3}$ & 7.8 & 6.1 & 6.0 \\ 
    \bottomrule
    \end{tabular}}
    \caption{\small Empirical rejection proportions (\%) of tests based on $T_{N,c}^{\phi, m}$ with $c$ samples. In each simulation, the total sample size is approximately $N\approx 50c$: $c-1$ unshifted subsamples have $\floor{50c/(c-1+r)}$ observations each and are drawn from the distribution indicated by the first column, with location $\mu_0=\pi/2$ and concentration $\kappa=5$ (blue pdf in the charts), while the remaining subsample has $\ceil{50cr/(c-1+r)}$ observations and comes from a shifted distribution with location $\mu_c=\mu_0 + \pi/10$ and the same concentration (red pdf). Results are based on $M = 10^4$ replications. Within each row and kernel, boldface indicates the maximum power across aggregations. Values exceeding the lower end of the one-sided $95\%$ confidence interval for that maximum are also shown in bold.}
    \label{tbl:pow-imbalanced}
\end{table}

For each $r\in\{1/2, 1, 2\}$ and $c\in\{3,5,10\}$, with $n=50$, we generate $M=10^4$ Monte Carlo replications of total size $N\approx nc$. Each replication contains $c-1$ unshifted subsamples with $\floor{nc/(c-1+r)}$ observations each and one shifted subsample with $\ceil{rnc/(c-1+r)}$ observations. For every replication, the tests based on $\smash{T_{N,c}^{\phi,{\rm avg}}}$, $\smash{T_{N,c}^{\phi,\max}}$, and $\smash{T^{\phi,{\rm LSE}_\kappa}_{N,c}}$ with $\kappa\in\{1,10\}$ are computed using the $\rm AD$ and ${\rm SM}_{10}$ kernels. The asymptotic critical values at the $\alpha=0.05$ level are obtained from $M$ Monte Carlo replications of the $K_{\rm tr}=10^3$ truncation of the limiting null distribution in Theorem~\ref{thm:asymp-H0-general}. The empirical rejection proportions are reported in Table~\ref{tbl:pow-imbalanced}. Note that $r=1$ corresponds to the balanced case shown in Table~\ref{tbl:pow-shift}.

The following conclusions are drawn mainly from the vM, W, SC, and MvM(3) scenarios, since the MvM(4) alternative yields powers close to the nominal level and is therefore less informative for comparing aggregations.
\begin{enumerate}[label=(\textit{\roman*}),ref=\textit{\roman*}]
    \item For $r<1$, max-type aggregations, namely the maximum and ${\rm LSE}_{10}$, are generally the most powerful for $c=3,5$, as expected, since the smallest sample is the one carrying the departure from the null.
    \item For $r>1$, the behavior changes. Since the shifted subsample is larger, the pooled distribution is pulled toward it. As a result, the shifted subsample may become less discrepant relative to the pooled distribution, while the $c-1$ unshifted subsamples all display deviations in the opposite direction. The signal is therefore spread over several component statistics, which favors the average and ${\rm LSE}_{1}$ aggregations, especially for $c=3,5$.
    \item For $c=10$, the transition between max-type and average-type behavior is less clear. Even when $r>1$, the shifted subsample does not dominate the pooled distribution as strongly as it does for smaller $c$, because its relative weight is $r/(c-1+r)$. Thus, both types of aggregation can be competitive, depending on the distributional family and the strength of the signal.
    \item The ${\rm LSE}$ family behaves as expected as an interpolation between the two extremes. The statistic ${\rm LSE}_{10}$ is close to the maximum and performs well when the evidence is concentrated in one component, whereas ${\rm LSE}_{1}$ is closer to the average and benefits when information is spread across several component statistics.
\end{enumerate}

These conclusions support the flexibility of the general aggregation design, as there is no single optimal choice across all sample-size ratios $r$ and values of $c$.

\subsection{Area-based test statistic}\label{subsec:sims:area}

In this section we illustrate the connection between the area-based statistic $T_{N,c}^{\rm A}$ and its geometric motivation, $A_{N,c}$, introduced in Section~\ref{subsec:area}. For each $n\in\{10, 50, 200\}$, we generated two Monte Carlo realizations consisting of $c=3$ samples of size $n$: one under the null hypothesis and one under a von Mises alternative with locations $\mu_\ell = 2\pi(\ell-1)/3, \ell=1,2,3$, and common concentration $\kappa=2$. For each realization, we computed $T_{N,c}^{\rm A}$ and $A_{N,c}$.

Figure~\ref{fig:area-construction} shows the complete geometric construction of the statistic $A_{N,c}$, together with the resulting values of $A_{N,c}$ and $T_{N,c}^{\rm A}$, for each scenario. In each block, the left figure shows a unit circle where each point corresponds to a transformed uniform score, placed at the corresponding angle $\alpha_{i}^{(\ell)}$. The corresponding unit vectors $\smash{\{\bs^{(\ell)}_{i}\}_{i=1}^{n_\ell}}$ are also represented. In the right figure, we represent the $2n_\ell$-polygons, one for each sample, whose edges are obtained by sequentially connecting $\smash{\{\bs^{(\ell)}_{(i)}\}_{i=1}^{n_\ell}}$ and their corresponding opposite vectors. The area of each polygon $T^{(\ell)}$ is displayed. 
\begin{figure}[ht!]
    \centering
    \begin{subfigure}[b]{\linewidth}
        \includegraphics[width=0.49\linewidth,trim={2.1cm 0.55cm 0.45cm 0.9cm},clip=true]{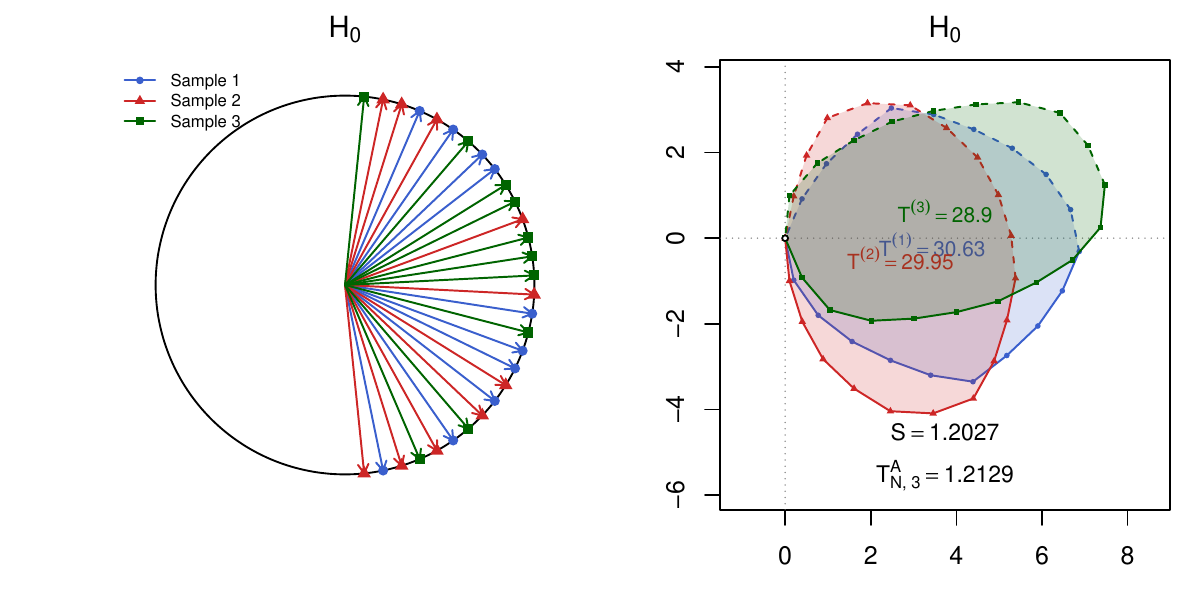}
        \hfill
        \includegraphics[width=0.49\linewidth,trim={2.1cm 0.55cm 0.45cm 0.9cm},clip=true]{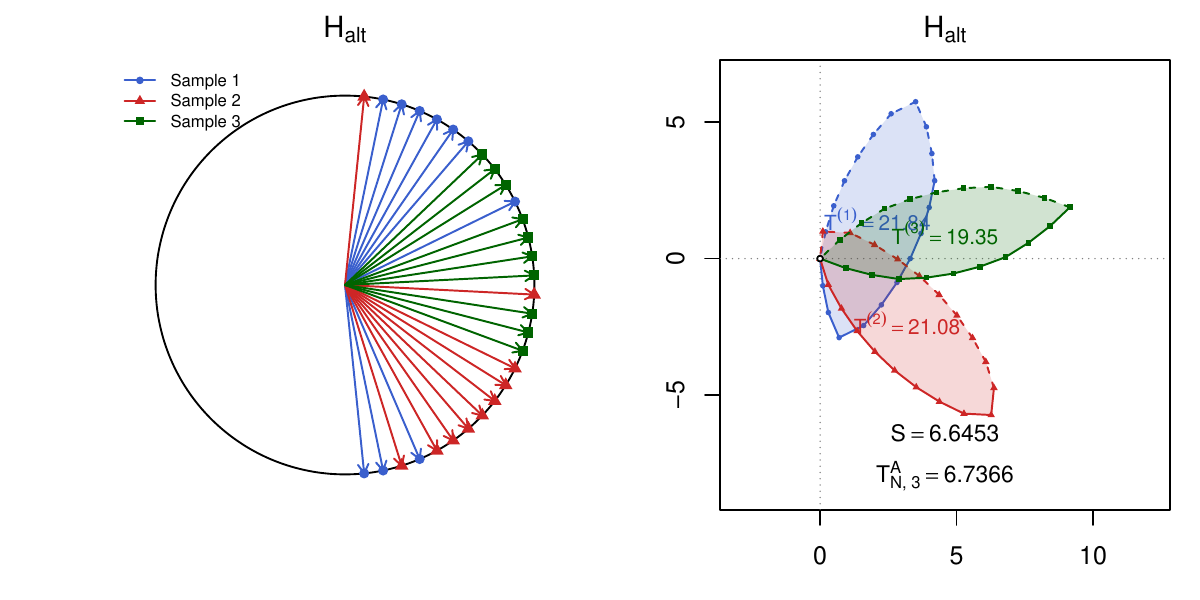}
        \caption{\small $n=10$}
    \end{subfigure}
    \begin{subfigure}[b]{\linewidth}
        \includegraphics[width=0.49\linewidth,trim={2.1cm 0.55cm 0.45cm 0.9cm},clip=true]{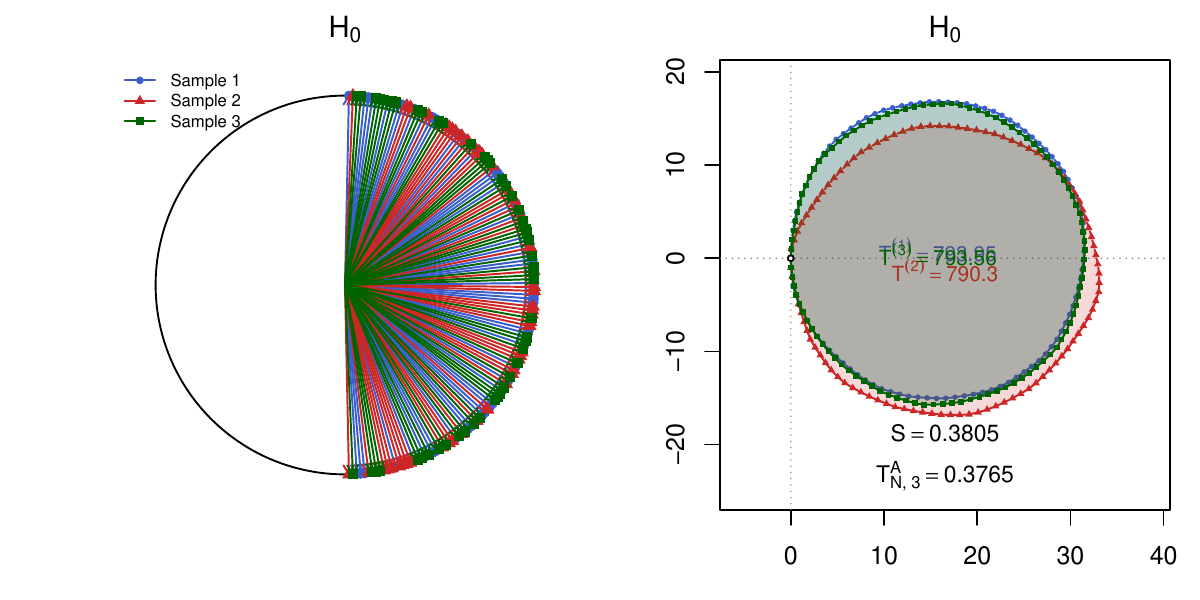}
        \hfill
        \includegraphics[width=0.49\linewidth,trim={2.1cm 0.55cm 0.45cm 0.9cm},clip=true]{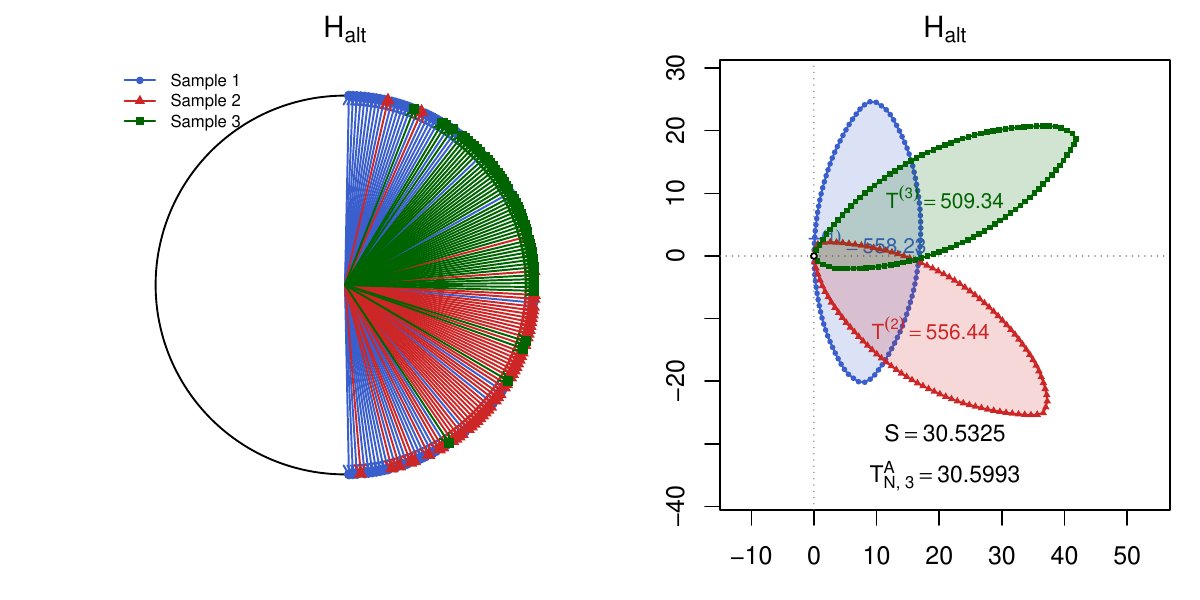}
        \caption{\small $n=50$}
    \end{subfigure}
    \begin{subfigure}[b]{\linewidth}
        \includegraphics[width=0.49\linewidth,trim={2.1cm 0.55cm 0.45cm 0.9cm},clip=true]{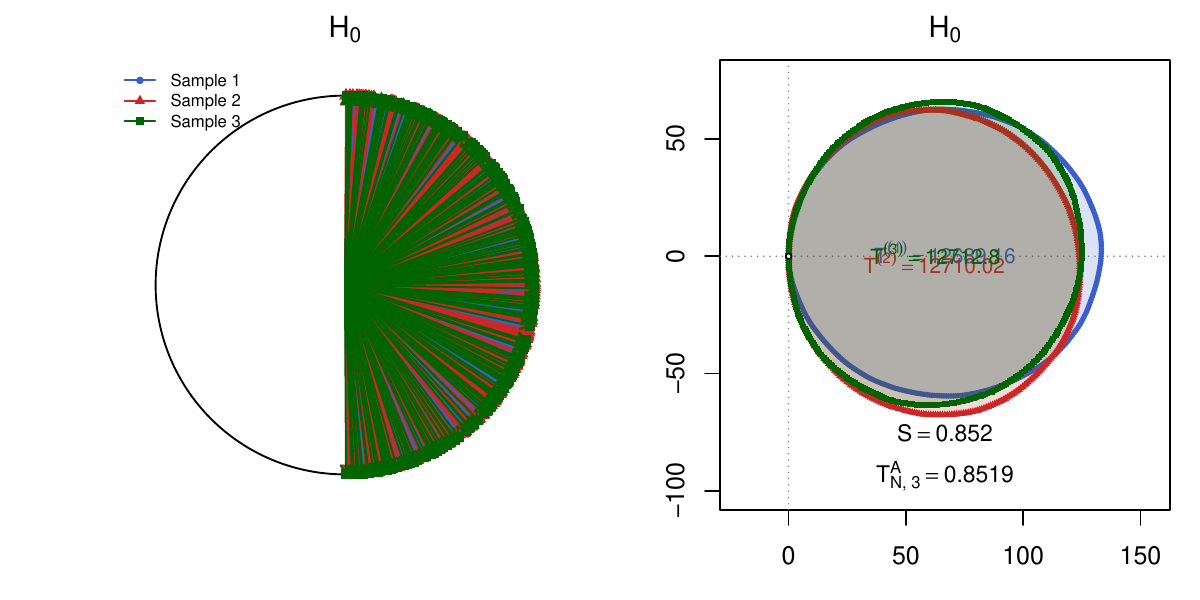}
        \hfill
        \includegraphics[width=0.49\linewidth,trim={2.1cm 0.55cm 0.45cm 0.9cm},clip=true]{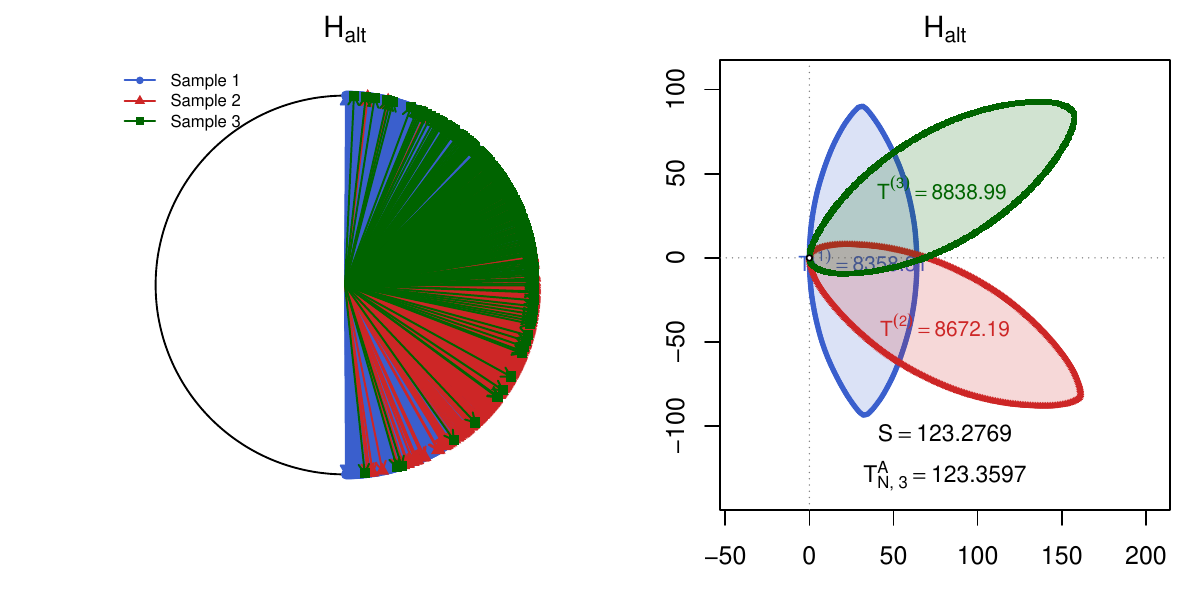}
        \caption{\small $n=200$}
    \end{subfigure}
    \caption{\small Illustration of the area-based statistic $A_{3n,3}$ with $c=3$ balanced samples under the null hypothesis (left column) and under a von Mises equispaced-location alternative (right column). Each sample has size $n$ indicated in the rows. For each panel, the left figure shows the transformed uniform scores and their corresponding unit vectors $\smash{\{\bs_{i}^{(\ell)}\}}_{i=1}^{n_{\ell}}$. The right figure shows the $2n_\ell$-polygons whose edges are obtained by sequentially connecting $\smash{\{\bs_{(i)}^{(\ell)}\}}_{i=1}^{n_{\ell}}$ and their opposite vectors. Polygon labels display the corresponding areas $T^{(\ell)}$, while the titles report the statistic values $A_{N,c}$ and $T_{N,c}^{\rm A}$.
    }\label{fig:area-construction}
\end{figure}

The intuition described in Section~\ref{subsec:area} becomes clear. Under the null hypothesis, the constructions are expected to be nearly regular polygons, as the edge directions are approximately evenly spread over the semicircle. As $n$ increases, these nearly regular polygons approach a circular shape. On the other hand, under the alternative, the edge vectors are clustered around certain locations, resulting in elongated polygons with smaller area. Since $A_{N,c}$ can be viewed as the deficit from the maximal polygon area, these smaller areas lead to larger values of $A_{N,c}$ and, asymptotically, of $T_{N,c}^{\rm A}$. The figure also illustrates the finite-sample approximation linking $A_{N,c}$ and $\smash{T_{N,c}^{\rm A}}$. Since $A_{N,c}$ is based on a finite-sample kernel that converges to the kernel defining $\smash{T_{N,c}^{\rm A}}$, the two statistic values get closer as $n$ increases.

Finally, note that rotating the data can modify the constructed polygons for $A_{N,c}$ in two different ways. First, it can rotate the polygons around the origin, which does not change their areas. Second, it can change the relative positions of some of the edge vectors, depending on the choice of origin for ranking the data, changing the corresponding area. Thus, $A_{N,c}$ is not rotation-invariant. However, asymptotically as $N\to\infty$, this effect becomes negligible. Indeed, $T_{N,c}^{\rm A}$ is rotation-invariant by definition.

\fi


\end{document}

%% file: preamble_arXiv.tex
\usepackage[utf8]{inputenc}
\usepackage[T1]{fontenc}

\usepackage{amsthm}
\usepackage{amsmath}
\usepackage{amsfonts}
\usepackage{amssymb}
\usepackage{dsfont}

\usepackage{graphicx} 
\usepackage{graphbox}
\usepackage{float}
\usepackage{booktabs}
\usepackage{multirow}
\usepackage{rotating}
\usepackage{array}
\usepackage[table,svgnames]{xcolor}
\usepackage{subcaption}
\usepackage[ruled]{algorithm2e}
\usepackage{tikz}
\usetikzlibrary{positioning, angles, quotes, arrows.meta}

\newcolumntype{C}[1]{>{\centering\arraybackslash}p{#1}}
\newcolumntype{R}[1]{>{\raggedleft\let\newline\\\arraybackslash\hspace{0pt}}m{#1}}
\newcolumntype{L}[1]{>{\raggedright\let\newline\\\arraybackslash\hspace{0pt}}m{#1}}

\usepackage{breakcites}

\usepackage{nowidow}

\usepackage{enumitem}

\usepackage[top=1.5cm,bottom=2cm,right=2.5cm,left=2.5cm]{geometry}
\usepackage{titling}

\usepackage{natbib}

\usepackage{ifplatform}

\usepackage{textcomp}
\usepackage{bbm}

\usepackage{comment}

\ifwindows
\fi

\allowdisplaybreaks

\newcommand{\lp}{\left(}
\newcommand{\rp}{\right)}
\newcommand{\lc}{\left[}
\newcommand{\rc}{\right]}

\newcommand{\R}{\mathbb{R}}

\newcommand{\Sp}{\mathbb{S}}

\newcommand{\Z}{\mathbb{Z}}
\newcommand{\N}{\mathbb{N}}
\newcommand{\rd}{\mathrm{d}}
\newcommand{\cL}{\mathcal{L}}

\newcommand{\bc}{\mathbf{c}}
\newcommand{\bx}{\mathbf{x}}
\newcommand{\be}{\mathbf{e}}

\newcommand{\bn}{\mathbf{n}}
\newcommand{\by}{\mathbf{y}}
\newcommand{\bX}{\mathbf{X}}
\newcommand{\bY}{\mathbf{Y}}
\newcommand{\bZ}{\mathbf{Z}}

\newcommand{\bU}{\mathbf{U}}

\newcommand{\bG}{\mathbf{G}}
\newcommand{\bR}{\mathbf{R}}
\newcommand{\bT}{\mathbf{T}}
\newcommand{\bmu}{\boldsymbol\mu}

\newcommand{\bw}{\mathbf{w}}

\newcommand{\bba}{\mathbf{a}}
\newcommand{\br}{\mathbf{r}}
\newcommand{\bs}{\mathbf{s}}

\newcommand{\bbeta}{\boldsymbol\eta}
\newcommand{\bdelta}{\boldsymbol\delta}

\newcommand{\bpi}{\boldsymbol\pi}
\newcommand{\bPi}{\boldsymbol\Pi}

\newcommand{\bga}{\boldsymbol\gamma}

\newcommand{\bTheta}{\boldsymbol\Theta}
\newcommand{\Ical}{\mathcal{I}}

\newcommand{\bSigma}{\boldsymbol\Sigma}
\newcommand{\bLambda}{\boldsymbol\Lambda}

\newcommand{\Hcal}{\mathcal{H}}

\newcommand{\bB}{\mathbf{B}}
\newcommand{\bC}{\mathbf{C}}
\newcommand{\bD}{\mathbf{D}}

\newcommand{\bI}{\mathbf{I}}
\newcommand{\bS}{\mathbf{S}}

\newcommand{\bW}{\mathbf{W}}

\newcommand{\equald}{\stackrel{d}{=}}
\newcommand{\inlaw}{\rightsquigarrow}
\newcommand{\inprob}{\stackrel{\mathrm{P}}{\rightarrow}}

\newcommand{\lrp}[1]{\left(#1\right)}

\newcommand{\lrpBig}[1]{\Big(#1\Big)}
\newcommand{\lrpbigg}[1]{\bigg(#1\bigg)}
\newcommand{\lrpBigg}[1]{\Bigg(#1\Bigg)}
\newcommand{\lrc}[1]{\left[#1\right]}
\newcommand{\lrb}[1]{\left\{#1\right\}}
\newcommand{\floor}[1]{\lfloor #1 \rfloor}
\newcommand{\ceil}[1]{\lceil #1 \rceil}

\newcommand{\E}[1]{\mathrm{E}\lc #1\rc}

\newcommand{\Es}[2]{\mathrm{E}_{#2}\lc #1\rc}
\newcommand{\Vs}[2]{\mathrm{Var}_{#2}\lc #1\rc}

\newcommand{\diag}[1]{\mathrm{diag}\lp #1\rp}

\newcommand{\Esbig}[2]{\mathrm{E}_{#2}\big[ #1\big]}

\newcommand{\EsBigg}[2]{\mathrm{E}_{#2}\Bigg[ #1\Bigg]}

\newcommand{\abs}[1]{\left| #1\right|}

\DeclareFontFamily{OT1}{pzc}{}
\DeclareFontShape{OT1}{pzc}{m}{it}{<-> s * [1.10] pzcmi7t}{}
\DeclareMathAlphabet{\mathpzc}{OT1}{pzc}{m}{it}

\newcommand{\symvMF}{($\vee$)}
\newcommand{\symMvMF}{($\ast$)}
\newcommand{\symbelts}{($\circ$)}
\newcommand{\symcross}{($+$)}

\newtheorem{theorem}{Theorem}[section]
\newtheorem{corollary}{Corollary}[section]
\newtheorem{remark}{Remark}[section]
\newtheorem{proposition}{Proposition}[section]
\newtheorem{lemma}{Lemma}[section]

\newtheorem{examp}{Example}[section]
\newtheorem{example}[examp]{Example}

